\documentclass[a4paper,11pt]{article}
\usepackage{amsmath, amsfonts, amscd, amssymb, amsthm, enumerate, float}

\def\bfB{\mathbf{B}}

\def\id{\mathrm{id}}
\def\ad{\mathrm{ad}}

\newcommand{\Mat}{\operatorname{M}}

\newcommand{\GL}{\operatorname{GL}}
\newcommand{\res}{\operatorname{res}}
\newcommand{\Ker}{\operatorname{Ker}}
\newcommand{\End}{\operatorname{End}}
\newcommand{\Gal}{\operatorname{Gal}}
\newcommand{\Vect}{\operatorname{span}}
\newcommand{\im}{\operatorname{Im}}
\newcommand{\Tr}{\operatorname{Tr}}
\newcommand{\car}{\operatorname{char}}
\newcommand{\tr}{\operatorname{tr}}

\newcommand{\rk}{\operatorname{rk}}

\newcommand{\Root}{\operatorname{Root}}
\renewcommand{\setminus}{\smallsetminus}

\def\F{\mathbb{F}}
\def\K{\mathbb{K}}
\def\R{\mathbb{R}}
\def\C{\mathbb{C}}

\def\N{\mathbb{N}}
\def\Z{\mathbb{Z}}

\renewcommand{\L}{\mathbb{L}}

\def\calA{\mathcal{A}}

\def\calM{\mathcal{M}}

\def\calW{\mathcal{W}}

\def\lcro{\mathopen{[\![}}
\def\rcro{\mathclose{]\!]}}

\theoremstyle{definition}
\newtheorem{Def}{Definition}
\newtheorem{Not}[Def]{Notation}

\theoremstyle{plain}
\newtheorem{theo}{Theorem}[section]
\newtheorem{prop}[theo]{Proposition}
\newtheorem{cor}[theo]{Corollary}
\newtheorem{lemma}[theo]{Lemma}

\theoremstyle{plain}

\theoremstyle{remark}
\newtheorem{Rems}{Remarks}
\newtheorem{Rem}[Rems]{Remark}

\title{The sum and the product of two quadratic matrices}
\author{Cl\'ement de Seguins Pazzis\footnote{Universit\'e de Versailles Saint-Quentin-en-Yvelines, Laboratoire de Math\'ematiques
de Versailles, 45 avenue des \'Etats-Unis, 78035 Versailles cedex, France}
\footnote{e-mail address: dsp.prof@gmail.com}}

\begin{document}

\thispagestyle{plain}

\maketitle

\begin{abstract}
Let $p$ and $q$ be polynomials with degree $2$ over an arbitrary field $\F$.

In the first part of this article, we characterize the matrices that
can be decomposed as $A+B$ for some pair $(A,B)$ of square matrices such that $p(A)=0$ and $q(B)=0$.
The case when both polynomials $p$ and $q$ are split was already known \cite{Bothasquarezero,dSPidem2,dSPsumoftwotriang}.
In the first half of this article, we complete the study by tackling the case when at least one of the polynomials $p$ and $q$ is irreducible over $\F$.

In the second half of the article, we use a similar method to characterize, under the assumption that
$p(0)q(0) \neq 0$, the matrices that
can be decomposed as $AB$ for some pair $(A,B)$ of square matrices such that $p(A)=0$ and $q(B)=0$.
\end{abstract}

\vskip 2mm
\noindent
\emph{AMS Classification :} 15A24; 15B33.

\vskip 2mm
\noindent
\emph{Keywords :} quadratic matrices, rational canonical form, fields with characteristic $2$, companion matrices,
quaternion algebras, Galois theory.

\tableofcontents

\section{Introduction}

\subsection{Basic notation and concepts}

Let $\F$ be an arbitrary field and $\overline{\F}$ be an algebraic closure of it. We denote by $\car(\F)$ the characteristic of $\F$.
We denote by $\Mat_n(\F)$ the algebra of all square matrices with $n$ rows and entries in $\F$, and by $I_n$
its unity. The similarity of two square matrices $A$ and $B$ is denoted by $A \simeq B$.
We denote by $\N$ the set of all non-negative integers, and by $\N^*$ the set of all positive ones.
Given a polynomial $p \in \F[t]$, we denote by $\Root(p)$ the set of all roots of $p$ in $\overline{\F}$, and, if $p$ is non-constant and monic with degree $n$, we denote by $\tr(p)$ the opposite of the coefficient of $p$ on $t^{n-1}$, which we call the trace of $p$.

An element of an $\F$-algebra $\calA$ is called \textbf{quadratic} when it is annihilated by a polynomial of degree $2$ of $\F[t]$.
Basic special cases of such elements are the idempotents ($a^2=a$), the involutions ($a^2=1_\calA$) and the square-zero elements ($a^2=0$).
Given an element $a$ of an $\F$-algebra $\calA$ together with a polynomial $p \in \F[t]$ with degree $2$ such that
$p(a)=0$, we set $a^\star:=(\tr p)\,1_\calA-a$, which we call the \textbf{$p$-conjugate} of $a$, and we note that $aa^\star=a^\star a=p(0)1_\calA$
(in this notation, the polynomial should normally be specified because of the possibility that $a$ be a scalar multiple of $1_\calA$, but it will always be clear from the context). Note that if $p$ is irreducible then $a$ and $a^\star$ are its roots in the quadratic extension $\F[a]$.

The following basic result will be used throughout the article so we state it and prove it right away. 

\begin{lemma}[Basic Commutation Lemma]\label{basiccommutelemma}
Let $p$ and $q$ be monic polynomials of $\F[t]$ with degree $2$. Let $a,b$ be elements of an $\F$-algebra $\calA$
such that $p(a)=q(b)=0$, and denote respectively by $a^\star$ and $b^\star$ the $p$-conjugate of $a$ and the $q$-conjugate of $b$.
Then, $a$ and $b$ commute with $ab^\star+b a^\star$.
\end{lemma}

Note also that
$$a b^\star+ba^\star=\tr(q)\,a+\tr(p)\,b-(ab+ba)=a^\star b+b^\star a.$$

\begin{proof}
On the one hand
$$(ab^\star+b a^\star)a=ab^\star a+p(0)b,$$
and on the other hand, $ab^\star+b a^\star=b^\star a+a^\star b$, whence
$$a(ab^\star+b a^\star)=a(b^\star a+a^\star b)=ab^\star a+p(0)b.$$
Thus, $a$ commutes with $ab^\star+b a^\star$. Symmetrically, so does $b$.
\end{proof}

\subsection{The $(p,q)$-sums problem}

Let $p$ and $q$ be monic polynomials of degree $2$ over $\F$.
An element $x$ of an $\F$-algebra $\calA$ is called a $(p,q)$-\textbf{sum} (respectively, a \textbf{$(p,q)$-difference}) whenever
it splits as $x=a+b$ (respectively, $x=a-b$) where $a,b$ are elements of $\calA$ that satisfy $p(a)=0$ and $q(b)=0$.
In particular, by taking $\calA=\Mat_n(\F)$ or $\calA=\End(V)$ for some vector space $V$ over $\F$, we
have the notion of a $(p,q)$-sum and of a $(p,q)$-difference for square matrices over $\F$ and for endomorphisms of $V$.
Those two notions are easily connected: an element of $\calA$ is a $(p,q)$-sum if and only if it is a $(p,q(-t))$-difference.

We will focus only on $(p,q)$-differences, as many results turn out to have a more elegant formulation
 when expressed in terms of $(p,q)$-differences rather than in terms of $(p,q)$-sums.
Here is the first problem we will address here:
\begin{center}
When is a square matrix a $(p,q)$-difference?
\end{center}
Since the set of all matrices $A \in \Mat_n(\F)$ such that $p(A)=0$ is a union of similarity classes, and
ditto for $q$ instead of $p$, the set of all matrices in $\Mat_n(\F)$ that are $(p,q)$-differences is a union of similarity classes.
Hence, in theory it should be possible to find necessary and sufficient conditions for being a $(p,q)$-difference in terms
of either the Jordan normal form or the rational canonical form.

Before we go on, we need to note that the problem remains essentially unchanged in replacing one of the polynomials
$p$ or $q$ with one of its \emph{translated} polynomials.
For instance, if we let $d \in \F$ and consider $p_1:=p(t+d)$, we note that a matrix $A \in \Mat_n(\F)$
is annihilated by $p$ if and only if $A-dI_n$ is annihilated by $p_1$, and hence
a matrix $M$ is a $(p,q)$-difference if and only is $M-d I_n$ is a $(p_1,q)$-difference.
Throughout the text, we shall say that $p$ \textbf{equals $q$ up to a translation} whenever $q(t)=p(t+d)$ for some scalar $d \in \F$.

Special cases in the above problem were solved starting in the early nineteen nineties.
First, in \cite{HP} Hartwig and Putcha solved the case when $p=q=t^2-t$ over the field of complex numbers,
i.e.\ they determined the matrices that can be written as the difference of two idempotent complex matrices, and
they also determined those that can be written as the sum of two idempotent complex matrices (those problems are easily seen to be equivalent
by noting that a matrix $A$ of $\Mat_n(\F)$ is idempotent if and only if $I_n-A$ is idempotent).
Later, Wang and Wu \cite{WuWang} obtained similar characterizations for sums of two square-zero complex matrices,
and Wang alone \cite{Wang1} obtained a characterization of the matrices that are the sum of an idempotent matrix
and a square-zero one, again over the field of complex numbers.
In all those works, both the results and the methods can be generalized effortlessly to any algebraically closed field with characteristic not $2$.

More recently, the results of the above authors were extended to arbitrary fields, even those with characteristic $2$.
In \cite{dSPidem2}, we managed to obtain a description of all matrices that split into a linear combination of two idempotents
with fixed nonzero coefficients, over an arbitrary field.
It is easily seen that this yields a solution to the above problem in the slightly more general case when both $p$ and $q$
are split polynomials with simple roots. In \cite{Bothasquarezero}, Botha
extended the classification of sums of two square-zero matrices to an arbitrary field
(see also the appendix of \cite{dSP3squarezero} for an alternative proof).
The case when both polynomials $p$ and $q$ are split was finally completed in \cite{dSPsumoftwotriang}, where
Wang's result on the sum of an idempotent matrix and a square-zero one was extended to all fields.

Yet, to this day nothing was known on the case when one of the polynomials $p$ and $q$ is irreducible over $\F$.
It is our ambition here to complete the study by giving a thorough treatment of the remaining cases.

\subsection{The $(p,q)$-products problem}

Let $p$ and $q$ be monic polynomials of degree $2$ over $\F$.
An element $x$ of an $\F$-algebra $\calA$ is called a \textbf{$(p,q)$-product} whenever
it splits as $x=ab$ where $a,b$ are elements of $\calA$ that satisfy $p(a)=0$ and $q(b)=0$.
In particular, by taking $\calA=\Mat_n(\F)$ or $\calA=\End(V)$ for some vector space $V$ over $\F$, we
have the notion of a $(p,q)$-product for square matrices with entries in $\F$ and for endomorphisms of a vector space over $\F$.

Here is the second problem we will tackle:
\begin{center}
When is a square matrix a $(p,q)$-product?
\end{center}
Since the set of all matrices $A \in \Mat_n(\F)$ such that $p(A)=0$ is a union of similarity classes, and
ditto for $q$ instead of $p$, the set of all matrices in $\Mat_n(\F)$ that are $(p,q)$-products is a union of similarity classes. Hence, in theory it should be possible to find necessary and sufficient conditions for a matrix to be a $(p,q)$-product in terms of either its Jordan normal form or its rational canonical form.

Note that, in $\Mat_n(\F)$, the $(p,q)$-products are the $(q,p)$-products, since it is known that every square matrix over a field
is similar to its transpose.

Before we go on, we also need to note that the problem remains essentially unchanged should
$p$ or $q$ be replaced with one of its \emph{homothetic} polynomials:
Given $d \in \F \setminus \{0\}$, we set
$$H_d(p):=d^{-2}p(d t),$$
which is a monic polynomial of $\F[t]$ with degree $2$.
Note that an element $a$ of an $\F$-algebra $\calA$ is annihilated by $p$ if and only if $d^{-1}a$ is annilated by
$H_{d}(p)$. Hence, a matrix $M$ is a $(p,q)$-product if and only if $d^{-1}M$ is an $(H_d(p),q)$-product.
Throughout the text, we shall say that two monic polynomials $p_1$ and $p_2$ with degree $2$ in $\F[t]$ are \textbf{homothetic} over $\F$
whenever $p_2=H_d(p_1)$ for some $d \in \F \setminus \{0\}$.

Note also that if $p(0)q(0) \neq 0$, then a $(p,q)$-product must be invertible.
\vskip 3mm

The topic of $(p,q)$-products has a long history that started in the nineteen sixties:
\begin{itemize}
\item The first result was due to Wonenburger \cite{Wonenburger} who, over a field with characteristic not $2$,
classified the $(t^2-1,t^2-1)$-products in $\Mat_n(\F)$, i.e.\ the products of two involutions.
This result was shortly generalized to all fields by Djokovi\'c \cite{Djokovic}, and rediscovered independently by
Hoffman and Paige \cite{HoffmanPaige}. Famously, the solutions are the
invertible matrices that are similar to their inverse.
\item Almost simultaneously, Ballantine characterized the $(t^2-t,t^2-t)$-products in $\Mat_n(\F)$ (where $\F$ is an arbitrary field). In other words, he classified the matrices that split into the product of two idempotents (he even classified the ones that split into the product of $k$ idempotents for a given positive integer $k$).
\item In a series of articles, Wang obtained an almost complete classification of the remaining cases when
the field is the one of complex numbers (his proofs generalize effortlessly to any algebraically closed field).
In \cite{WangWu}, Wu and him solved the case when both $p$ and $q$ have a nonzero double root (which reduces to the situation where $p=q=(t-1)^2$). In \cite{Wang1} and \cite{WangII}, Wang considered the more general situation where $p(0)q(0) \neq 0$,
with some stringent restrictions in the case when $p$ or $q$ has a double root (essentially, in that situation he only tackled
the case when $p$ has a double root and $q$ has opposite distinct roots, and over an algebraically closed field with characteristic not $2$).
In \cite{WangIII}, he tackled the case when $p(t)=t^2-t$ and $q(0) \neq 0$.
\item In \cite{Novak}, Novak solved the case when $p(t)=q(t)=t^2$, over an arbitrary field.
\item In \cite{BothaUnipotentSquarezero}, Botha solved the case when $p(t)=t^2-t$ and $q(t)=t^2$, over an arbitrary field.
\end{itemize}

Thus, even over an algebraically closed field, the general problem of classifying $(p,q)$-products is still partly open.
Subsequent efforts were made to extend some of the above results to arbitrary fields:

\begin{itemize}
\item In \cite{Bunger}, B\"unger, Kn\"uppel and Nielsen characterized the $(p,p)$-products when $p$ splits over $\F$, $p(0)=1$ and
no fourth root of the unity is a root of $p$.
\item In \cite{Bothaunipotent}, Botha generalized Wang's characterization of the $((t-1)^2,(t-1)^2)$-products to an arbitrary field.
\end{itemize}

Hence, before the present article no general solution to our problem was known.
Even a full solution to the case when both polynomials $p$ and $q$ are split is missing from the literature.
It is our ambition here to contribute to the problem by giving a thorough treatment
of the case when $p(0)q(0) \neq 0$, which essentially amounts to determining the \emph{invertible} $(p,q)$-products.

Thus, after this article, the only remaining case in the $(p,q)$-products problem will be the one when $p(0)=0$.
Different methods are needed to solve this case, and we hope that a general solution to it will be found in the near future.

Assume now that $p(0)q(0) \neq 0$. An element $x$ of an $\F$-algebra $\calA$ is called a $(p,q)$-\textbf{quotient} (in $\calA$)
whenever there exist elements $a$ and $b$ of $\calA$ such that $x=ab^{-1}$ and $p(a)=q(b)=0$
(this is equivalent to $x$ being a $(p,q^\sharp)$-product where $q^\sharp:=q(0)^{-1}t^2 q(t^{-1})$ is the reciprocal polynomial of $q$).
It turns out that the characterization of quotients is more easily expressed than the one of
products, in particular in the case when $p$ and $q$ are irreducible; Therefore, in the
remainder of the article we will only consider the problem of classifying the $(p,q)$-quotients among the automorphisms
of a finite-dimensional vector space. From our results on quotients, giving the corresponding results on products
is an elementary task that requires no further explanation.

\subsection{Basic structure of the article}

The rest of the article is split into three parts. In Section \ref{WPQalgebra}, we study
an algebra that is connected to the two polynomials $p$ and $q$ and
which shares a resemblance with quaternion algebras.
The structural results for this algebra will help us obtain the classification of so-called
\emph{regular} $(p,q)$-differences and of \emph{regular} $(p,q)$-quotients
(in short, a $(p,q)$-difference is regular when it has no eigenvalue in $\Root(p)-\Root(q)$, and, when $p(0)q(0) \neq 0$, a $(p,q)$-quotient is regular when it has no eigenvalue in $\Root(p)\Root(q)^{-1}$).

The last two sections, which constitute the main bulk of this article are devoted to
the complete classification of $(p,q)$-differences (Section \ref{DiffSection}) and to the complete
classification of $(p,q)$-quotients (Section \ref{QuotSection}). In the final appendix, we state and prove a
result on block cyclic matrices which is used throughout the article.

\subsection{Miscellaneous notation and useful facts}

Throughout the article, we need some notation and standard results from the representation theory for one endomorphism.
Given a monic polynomial $r(t)=t^n-\underset{k=0}{\overset{n-1}{\sum}} a_k t^k$ of $\F[t]$,
we denote by
$$C(r):=\begin{bmatrix}
0 &   & & (0) & a_0 \\
1 & 0 & &   & a_1 \\
0 & \ddots & \ddots & & \vdots \\
\vdots & \ddots & & 0 & a_{n-2} \\
(0) & \cdots & 0 &  1 & a_{n-1}
\end{bmatrix}\in \Mat_n(\F)$$
its \textbf{companion matrix}.
Remember that the rational canonical form theorem states that for any endomorphism $u$ of a finite-dimensional vector
space over $\F$, there exists a unique sequence $(r_1,\dots,r_k)$ of monic nonconstant polynomials over $\F$ such that:
\begin{enumerate}[(i)]
\item The endomorphism $u$ is represented in some basis by $C(r_1) \oplus \cdots \oplus C(r_k)$;
\item The polynomial $r_{i+1}$ divides $r_i$ for all $i \in \lcro 1,k-1\rcro$.
\end{enumerate}
The polynomials $r_1,\dots,r_k$ are called the invariant factors of $u$. We extend this finite sequence into an infinite one
$(r_i)_{i \geq 1}$ by setting $r_i:=1$ whenever $i>k$.
We convene that $C(1)$ denotes the $0$ by $0$ matrix.

We also recall the \emph{primary canonical form theorem}: For any endomorphism $u$ of a finite-dimensional vector
space over $\F$, there exists a sequence $(r_1,\dots,r_k)$ of nonconstant polynomials over $\F$, each of which is a power of some monic
irreducible polynomial of $\F[t]$, such that $u$ is represented in some basis by $C(r_1) \oplus \cdots \oplus C(r_k)$.
This sequence is uniquely determined by $u$ up to a permutation of its terms. The polynomials $r_1,\dots,r_k$ are called
the \textbf{elementary invariants} of $u$.

We finish with a definition and some additional notation.

\begin{Def}
Let $(u_n)_{n \geq 1}$ and $(v_n)_{n \geq 1}$ be non-increasing sequences of non-negative integers.
Let $p>0$ be a positive integer.
We say that $(u_n)_{n \geq 1}$ and $(v_n)_{n \geq 1}$ are \textbf{$p$-intertwined} when
$$\forall n \geq 1, \; u_{n+p} \leq v_n \quad \text{and} \quad v_{n+p} \leq u_{n.}$$
\end{Def}

\begin{Not}
Given an endomorphism $u$ of a finite-dimensional vector space $V$ over $\F$, a scalar $\lambda \in \F$
and a positive integer $k$, we set
$$n_k(u,\lambda):=\dim \Ker (u-\lambda\, \id_V)^k-\dim \Ker (u-\lambda\, \id_V)^{k-1}$$
i.e.\ $n_k(u,\lambda)$ is the number of cells of size at least $k$
associated to the eigenvalue $\lambda$ in the Jordan reduction of $u$.
\end{Not}

\section{The key $4$-dimensional algebra}\label{WPQalgebra}

\subsection{Definition and basic facts}\label{keyalgebrasection}

Let $R$ be a commutative unital $\F$-algebra. Let $p(t)=t^2-\lambda t+\alpha$, $q(t)=t^2-\mu t+\beta$
be monic polynomials with degree $2$ over $\F$, and $x$ be an element of $R$.
We set $\delta:=\lambda-\mu$.

We consider the following three matrices of $\Mat_4(R)$ :
$$A:=\begin{bmatrix}
0 & -\alpha & 0 & 0 \\
1 & \lambda & 0 & 0 \\
0 & 0 & 0 & -\alpha \\
0 & 0 & 1 & \lambda
\end{bmatrix}, \;
B:=\begin{bmatrix}
0 & -x & -\beta & -\lambda \beta \\
0 & \mu & 0 & \beta \\
1 & \lambda & \mu & \lambda\mu-x \\
0 & -1 & 0 & 0
\end{bmatrix}$$
and
$$C:=AB=\begin{bmatrix}
0 & -\alpha\mu & 0 & -\alpha \beta \\
0 & \lambda\mu-x & -\beta & 0 \\
0 & \alpha & 0 & 0 \\
1 & 0 & \mu & \lambda\mu-x
\end{bmatrix}.$$
One checks that
\begin{equation}\label{basicrel}
AB+BA=\mu A+\lambda B-x I_4, \quad p(A)=0 \quad \text{and} \quad q(B)=0.
\end{equation}
From there, one deduces that $\Vect_R(I_4,A,B,C)$ is stable under multiplication.
Note from the first columns that $\Vect_R(I_4,A,B,C)$ is the free $R$-submodule of $\Mat_4(R)$
with basis $(I_4,A,B,C)$.

We define $\calW(p,q,x)_R$ as the set $\Vect_R(I_4,A,B,C)$ equipped with its structure of $R$-algebra
inherited from that of $\Mat_4(R)$. We shall simply write $\calW(p,q,x)$ instead of
$\calW(p,q,x)_R$ when the ring $R$ under consideration is obvious from the context.

Note that relations \eqref{basicrel} lead to
\begin{equation}\label{basicrel1}
(A-B)^2-\delta (A-B)=(x-\alpha-\beta)\,I_4
\end{equation}
and, if $q(0) \neq 0$, to 
\begin{equation}\label{basicrel2}
(AB^{-1})^2=-\delta I_4+x \,(AB^{-1}).
\end{equation}

\begin{Rem}
It can be shown that $\calW(p,q,x)_R$ is the quotient algebra of the unital free noncommutative $R$-algebra
in two generators $a$ and $b$ by the two-sided ideal generated by $p(a),q(b)$ and
$a(\mu 1_R-b)+b(\lambda 1_R-a)-x\,1_R$,
but this will be of no use to us in the remainder of the article.
\end{Rem}

Next, we define
$$\Tr : \calW(p,q,x) \rightarrow R$$
as the unique $R$-linear mapping such that
$\Tr(I_4)=2$, $\Tr(A)=\lambda$, $\Tr(B)=\mu$ and $\Tr(C)=\lambda\mu-x$.
For $h \in \calW(p,q,x)$, we set
$$h^\star:=\Tr(h)\,I_4-h \in \calW(p,q,x),$$
which we call the adjoint of $h$, so that $h \mapsto h^\star$ is an involution of the $R$-module $\calW(p,q,x)$.
Moreover, one checks on the basis $(I_4,A,B,C)$ that
$$\forall (h_1,h_2) \in \calW(p,q,x)^2, \; (h_1h_2)^\star=h_2^\star\, h_1^\star.$$

Finally, we define a quadratic form
$$aI_4+bA+cB+dC \mapsto a\bigl(a+\lambda b+\mu c+(\lambda\mu-x)d\bigr)+b\bigl(\alpha b+xc+\alpha\mu d\bigr)+\beta c^2+
\bigl(\lambda\beta c+\alpha \beta d\bigr)d$$
on $\calW(p,q,x)$, which we call the norm of $\calW(p,q,x)$ and denote by $N_{\calW(p,q,x)}$, or more simply by $N$
whenever $p,q,R,x$ are obvious from the context.
Again, one checks that
\begin{equation}\label{adjunctionandNorm}
\forall h \in \calW(p,q,x), \; hh^\star=h^\star h=N(h)\,I_4.
\end{equation}
We denote by $b_N$ the polar form of $N$ as defined by
$$b_N : (h_1,h_2) \in \calW(p,q,x)^2 \mapsto N(h_1+h_2)-N(h_1)-N(h_2),$$
so that
$$\forall (h_1,h_2) \in \calW(p,q,x)^2, \; h_1h_2^\star+h_2h_1^\star=b_N(h_1,h_2)\,I_4.$$

\begin{Rem}\label{isotropicnormremark}
If one of the polynomials $p$ and $q$ splits over $\F$,
then $N$ is isotropic, i.e.\ it vanishes at some non-zero element of $\calW(p,q,x)$.
Indeed, if $p$ has a root $z$ in $\F$, then one checks
that $N(A-z I_4)=0$, since, denoting by $z'$ the second root of $p$, we see that
$$(A-z I_4)(A-zI_4)^\star=(A-z I_4)(A^\star-z I_4)=(A-z I_4)(z' I_4-A)=-p(A)=0.$$
Likewise if $q$ has a root $y$ in $\F$, then $N(B-y I_4)=0$.
\end{Rem}

\subsection{A deeper study of the algebra $\calW(p,q,x)$: when $R$ is a field}\label{Walgebrafieldsection}

Here, we assume that $R$ is a field extension of $\F$, and we denote it by $\L$.
We fix monic polynomials $p=t^2-\lambda t+\alpha$ and $q=t^2-\mu t+\beta$ with degree $2$ in $\F[t]$, and we fix an element $x$ of $\L$.
The norm of $\calW(p,q,x)_\L$ is simply denoted by $N$. We start by analyzing when the quadratic form $N$ is degenerate.

\begin{prop}\label{nondegcharac}
Consider an algebraic closure $\overline{\L}$ of $\L$.
The quadratic form $N$ is degenerate if and only if
there exist elements $x_1,x_2,y_1,y_2$ of $\overline{\L}$ such that $p(t)=(t-x_1)(t-x_2)$, $q(t)=(t-y_1)(t-y_2)$,
and $x=x_1y_1+x_2y_2$.
\end{prop}

\begin{proof}
Let us split $p(t)=(t-x_1)(t-x_2)$ and $q(t)=(t-y_1)(t-y_2)$.
The matrix of the polar form $b_N$ in the basis $(I_4,A,B,C)$ reads
$$M=\begin{bmatrix}
2 & \lambda & \mu & \lambda\mu-x \\
\lambda & 2\alpha & x & \alpha\mu \\
\mu & x & 2\beta & \lambda\beta \\
\lambda\mu-x & \alpha\mu & \lambda\beta & 2 \alpha \beta
\end{bmatrix}.$$
Setting $\gamma:=\lambda^2\beta+\mu^2 \alpha-4\alpha \beta$, a tedious but straightforward computation shows that
\begin{align*}
\det M & = x^4 -2 \lambda \mu  x^3+(2\gamma+\lambda^2\mu^2)x^2- 2\lambda \mu \gamma x + \gamma^2 \\
& = (x^2-\lambda\mu x+\gamma)^2 \\
& = \bigl(x-(x_1y_1+x_2y_2)\bigr)^2\bigl(x-(x_1y_2+x_2y_1)\bigr)^2.
\end{align*}
The claimed result follows.
\end{proof}

Next, in the situation where $N$ is non-degenerate, we further analyze the structure of $\calW(p,q,x)_\L$.

\begin{prop}\label{structureofWprop}
Assume that the quadratic form $N$ is non-degenerate.
Then, $\calW(p,q,x)_\L$ is a quaternion algebra over $\L$ and its norm of quaternion algebra is $N$.
\end{prop}

\begin{proof}
First of all, we note that the linear form $\Tr$ on $\calW(p,q,x)$ is nonzero. Indeed, if $\Tr=0$, we would find that
$\car(\F)=2$ and $\lambda=\mu=x=0$, but then
$p(t)=(t-x_1)^2$ and $q(t)=(t-y_1)^2$ for some scalars $x_1$ and $y_1$ in $\overline{\L}$, leading to $x=0=x_1y_1+x_1y_1$
and contradicting the non-degeneracy of $N$ (see Proposition \ref{nondegcharac}).

Next, we consider the linear hyperplane $H:=\Ker \Tr$ of the $\L$-vector space $\calW(p,q,x)$.
The radical of the restriction of $N$ to $H$ has dimension at most $1$,
whence we can find a $2$-dimensional subspace $P$ of $H$ on which $N$ is d-regular.
It follows that $\forall h \in P, \; h^2=-N(h)\,I_4$.
Hence, the identity of $P$ yields a morphism $\varphi : C(-N_{|P}) \rightarrow \calW(p,q,x)$ of $\L$-algebras
whose domain is the Clifford algebra $C(-N_{|P})$ over $\L$, and whose range includes $P$.
Yet, since $P$ has dimension $2$ and $N_{|P}$ is non-degenerate, it is known (for
fields with characteristic not 2, see \cite{invitquad} p.528 Theorem 1.1.5, otherwise
see \cite{invitquad} p.744 Theorem 2.2.3) that the $\L$-algebra
$C(-N_{|P})$ is simple. It follows that $\varphi$ is injective, and since $\dim_\L C(-N_{|P})=4=\dim_\L \calW(p,q,x)$,
we deduce that $\varphi$ is an isomorphism. Hence, $\calW(p,q,x)$ is a quaternion algebra.

Yet, in a quaternion algebra $C$ over $\L$, the set of all $z \in C$ such that $z^2 \in \L 1_C$
splits uniquely as the union $(\L 1_C) \cup H'$ for some linear hyperplane $H'$ of $C$ whose elements are called the pure quaternions,
and there is a unique antiautomorphism $h \mapsto \overline{h}$ of the $\L$-algebra $C$, called the conjugation,
whose restriction to $H'$ is $h \mapsto -h$. Yet, in $\calW(p,q,x)$, for every element $h$, we have both
$hh^\star=\Tr(h)h-h^2$ and $hh^\star \in \L 1_{\calW(p,q,x)}$, whence
$h^2 \in \L 1_{\calW(p,q,x)}$ if and only if $h \in \L 1_{\calW(p,q,x)}$ or $\Tr(h)=0$.
Hence, in the quaternion algebra $\calW(p,q,x)$, the pure quaternions are the elements of $H$.
Since $h^\star=-h$ for all such quaternions, and $h \mapsto h^\star$ is an antiautomorphism of $\calW(p,q,x)$,
we conclude that $h \mapsto h^\star$ is the conjugation of the quaternion algebra $\calW(p,q,x)$. Hence, Formula \eqref{adjunctionandNorm}
entails that the norm of the quaternion algebra $\calW(p,q,x)$ is $N$.
\end{proof}

From the classification of quaternion algebras  (for
fields with characteristic not 2, see \cite{invitquad} p.528 Theorem 1.1.5, otherwise
see \cite{invitquad} p.744 Theorem 2.2.3), we can conclude on the structure of $\calW(p,q,x)$
when $N$ is non-degenerate.

\begin{cor}\label{structureofWcor}
Assume that the quadratic form $N_{\calW(p,q,x)}$ is non-degenerate. \\
If it is non-isotropic, then $\calW(p,q,x)_\L$ is a skew field. \\
Otherwise, the $\L$-algebra $\calW(p,q,x)_\L$ is isomorphic to $\Mat_2(\L)$.
\end{cor}

\subsection{A deeper study of the algebra $\calW(p,q,x)$: when $R$ is a local quotient of $\F[t]$}

The following result generalizes the last statement of Corollary \ref{structureofWcor}.

\begin{prop}\label{localstructureprop}
Let $r$ be an irreducible polynomial of $\F[t]$, and $n \in \N^*$ be a non-zero integer.
Set $R:=\F[t]/(r^n)$ and let $x$ be the class of some polynomial of $\F[t]$ in $R$,
and $\overline{x}$ be the class of the same polynomial in the residue class field $\L:=\F[t]/(r)$.
Assume finally that the norm of $\calW(p,q,\overline{x})_\L$ is non-degenerate and isotropic.
Then, the $R$-algebra $\calW(p,q,x)_R$ is isomorphic to $\Mat_2(R)$.
\end{prop}

\begin{proof}
Denote by $\epsilon$ the class of $r$ in $\F[t]/(r^n)$.
With the construction from Section \ref{keyalgebrasection}, it is obvious that $\calW(p,q,\overline{x})_{\L}$
is naturally isomorphic to the quotient of $\calW(p,q,x)_R$ by the two-sided ideal $\epsilon \calW(p,q,x)_R$,
and we shall make this identification throughout the proof.
We will denote by $\overline{N}$ the norm of $\calW(p,q,\overline{x})_{\L}$, while $N$ still denotes the one of $\calW(p,q,x)_R$.

By Corollary \ref{structureofWcor}, we know that there exists an isomorphism $\varphi : \calW(p,q,\overline{x})_\L \overset{\simeq}{\rightarrow} \Mat_2(\L)$
of $\L$-algebras.
Moreover, such an isomorphism must be compatible with the enriched structure of quaternion algebra
(with its conjugation and norm). Yet, in $\Mat_2(\L)$ the conjugation is the classical adjunction $M \mapsto M^{\ad}$
(where $M^{\ad}$ is the transpose of the comatrix of $M$), and the norm is the determinant.

Given an arbitrary commutative $\F$-algebra $\calM$ and an element $u$ of it, we
say that a pair $(X,Y)$ of elements of $\calW(p,q,u)_\calM$ is \textbf{adapted} whenever it satisfies the following two conditions:
\begin{enumerate}[(i)]
\item $b_N(I_4,X)=b_N(I_4,Y)=0$, $N(X)=N(Y)=0$ and $b_N(X,Y)=-1$.
\item $(I_4,X,Y,XY)$ is a basis of the $\calM$-module $\calW(p,q,u)_\calM$.
\end{enumerate}

In the quaternion algebra $\Mat_2(\L)$, we see that $I_2 E_{1,2}^\ad+E_{1,2} I_2^{\ad}=-E_{1,2}+E_{1,2}=0$
and likewise with $E_{2,1}$ instead of $E_{1,2}$. Moreover $\det(E_{1,2})=0=\det(E_{2,1})$, and finally
$E_{1,2}E_{2,1}^\ad+E_{2,1} E_{1,2}^\ad=-E_{1,2}E_{2,1}-E_{2,1}E_{1,2}=-I_2$.
Finally, $(I_2,E_{1,2},E_{2,1},E_{1,2}E_{2,1})$ is a basis of the $\L$-vector space $\Mat_2(\L)$.
Hence, the pair $\bigl(\varphi^{-1}(E_{1,2}),\varphi^{-1}(E_{2,1})\bigr)$ is adapted in $\calW(p,q,\overline{x})_\L$.

Assuming for a moment that we have an adapted pair $(X,Y)$ in $\calW(p,q,x)_R$,
we claim that $\calW(p,q,x)_R$ is isomorphic to $\Mat_2(R)$.
Indeed, first of all we note that $b_N(I_4,X)=b_N(I_4,Y)=0$ means that
$X^\star=-X$ and $Y^\star=-Y$. Then, it follows from $N(X)=N(Y)=0$ that $X^2=Y^2=0$.
Finally, $b_N(X,Y)=-1$ reads $XY^\star+YX^\star=-I_4$, that is $XY+YX=I_4$.
Thus, condition (ii) yields an isomorphism $\psi : \calW(p,q,x)_R \rightarrow \Mat_2(R)$ of $R$-modules that maps
$I_4,X,Y,XY$ respectively to $I_2,E_{1,2},E_{2,1},E_{1,1}$, and the identities $X^2=Y^2=0$ and $YX=I_4-XY$
show that $\psi$ is actually a ring homomorphism.

It remains to prove that there exists an adapted pair in $\calW(p,q,x)_R$.
To do so, we shall use Hensel's method. Given $M \in \calW(p,q,x)_R$, we denote by $\overline{M}$
its class modulo $\epsilon$, and we shall see $\overline{M}$ as an element of the ring $\calW(p,q,\overline{x})_\L$.
Let $k \in \lcro 1,n-1\rcro$, and $(X_k,Y_k) \in \calW(p,q,x)_R^2$ be such that:
\begin{enumerate}[(a)]
\item $b_N(I_4,X_k)=0$ mod $\epsilon^k$, $b_N(I_4,Y_k)=0$ mod $\epsilon^k$, $N(X_k)=0$ mod $\epsilon^k$, $N(Y_k)=0$ mod $\epsilon^k$
and $b_N(X_k,Y_k)=-1$ mod $\epsilon^k$.
\item The family $(I_4,\overline{X_k},\overline{Y_k},\overline{X_kY_k})$ is a basis of the $\L$-vector space $\calW(p,q,\overline{x})_\L$.
\end{enumerate}
Then, we construct a pair $(X_{k+1},Y_{k+1}) \in \calW(p,q,x)_R$ such that $X_{k+1}=X_k$ mod $\epsilon^k$,
$Y_{k+1}=Y_k$ mod $\epsilon^k$, and the pair $(X_{k+1},Y_{k+1}) \in \calW(p,q,x)_R$ satisfies the above conditions at the step $k+1$.
To do so, let $Z \in \calW(p,q,x)_R$ be arbitrary, and set $X_{k+1}:=X_k+\epsilon^k Z$.
We write $b_N(I_4,X_k)=\epsilon^k h_1$ and $N(X_k)=\epsilon^k h_2$ for some $h_1,h_2$ in $\calW(p,q,x)_R$.
Then, $b_N(I_4,X_{k+1})=\epsilon^k (h_1+b_N(I_4,Z))$ and $N(X_{k+1})=\epsilon^k (h_2+b_N(X_k,Z))$ mod $\epsilon^{k+1}$.
Since $b_{\overline{N}}$ is non-degenerate and $\overline{I_4},\overline{X_k}$ are linearly independent over $\L$, the linear forms
$b_{\overline{N}}(\overline{I_4},-)$ and $b_{\overline{N}}(\overline{X_k},-)$ are independent, which shows that the linear system of equations
$$\begin{cases}
b_{\overline{N}}(\overline{I_4},U)=-\overline{h_1} \\
b_{\overline{N}}(\overline{X_k},U)=-\overline{h_2}
\end{cases}$$
has a solution $U$ in $\calW(p,q,\overline{x})_\L$. Lifting $U$, we recover that $Z \in \calW(p,q,x)_R$
can be chosen so as to have $b_N(I_4,Z)=-h_1$ mod $\epsilon$ and $b_N(X_k,Z)=-h_2$ mod $\epsilon$.
We choose such a $Z$ from now on, and hence we have $b_N(I_4,X_{k+1})=0$ mod $\epsilon^{k+1}$ and
$N(X_{k+1})=0$ mod $\epsilon^{k+1}$. Note that $b_N(X_{k+1},Y_k)=b_N(X_k,Y_k)$ mod $\epsilon^k$
whence $b_N(X_{k+1},Y_k)=-1$ mod $\epsilon^k$.

Next, let $T \in \calW(p,q,x)_R$ and set $Y_{k+1}:=Y_k+\epsilon^k T$.
We find three elements $h_3,h_4$ and $h_5$ of $\calW(p,q,x)_R$ such that
$b_N(I_4,Y_k)=\epsilon^k h_3$, $N(Y_k)=\epsilon^k h_4$ and $b_N(X_{k+1},Y_k)=-1+\epsilon^k h_5$.
As before, the linear system of equations
$$\begin{cases}
b_{\overline{N}}(\overline{I_4},V)=-\overline{h_3} \\
b_{\overline{N}}(\overline{Y_k},V)=-\overline{h_4} \\
b_{\overline{N}}(\overline{X_{k+1}},V)=-\overline{h_5}
\end{cases}$$
has a solution $V$ in $\calW(p,q,\overline{x})_{\mathbb{L}}$, and hence $T$ can be chosen as a representative of it,
in which case we find $b_N(I_4,Y_{k+1})=0$ mod $\epsilon^{k+1}$, $N(Y_{k+1})=0$ mod $\epsilon^{k+1}$ and
$b_N(X_{k+1},Y_{k+1})=-1$ mod $\epsilon^{k+1}$.
Hence, condition (a) is satisfied at the rank $k+1$ by $(X_{k+1},Y_{k+1})$. On the other hand, since condition (b)
is satisfied at the rank $k$ by $(X_k,Y_k)$, while $X_k$ and $X_{k+1}$ have the same reduction modulo $\epsilon$, and $Y_k$ and $Y_{k+1}$
have the same reduction modulo $\epsilon$, we obtain that condition (b) is also satisfied
by $(X_{k+1},Y_{k+1})$.

As we have shown that there exists an adapted pair in $\calW(p,q,\overline{x})_\L$, this construction yields, by induction,
a pair $(X,Y) \in \calW(p,q,x)_R^2$ that satisfies condition (i)
and for which $(I_4,\overline{X},\overline{Y},\overline{X}\overline{Y})$ is a basis of the $\L$-vector space $\calW(p,q,\overline{x})_\L$.
Since $R$ is a local ring with residue class field $\L$ and $\calW(p,q,x)_R$ is a free $R$-module with dimension $4$,
this shows that $(I_4,X,Y,XY)$ is a basis of the $R$-module $\calW(p,q,x)_R$.
Therefore, $(X,Y)$ is an adapted pair in $\calW(p,q,x)_R$, and a previous remark helps us conclude that
the $R$-algebra $\calW(p,q,x)_R$ is isomorphic to $\Mat_2(R)$.
\end{proof}

\section{The difference of two quadratic matrices}\label{DiffSection}

\subsection{The basic splitting}

Let $u$ be an endomorphism of a finite-dimensional vector space $V$ over $\F$.
Let $p$ and $q$ be monic polynomials with degree $2$ over $\F$,
which we write
$$p(t)=t^2-\lambda t+\alpha \quad \text{and} \quad q(t)=t^2-\mu t+\beta,$$
and set
$$\delta:=\lambda-\mu=\tr(p)-\tr(q).$$
The \textbf{s-fundamental polynomial} of the pair $(p,q)$ is defined as the resultant
$$F_{p,q}(t):=\res\bigl(p(x),q(x-t)\bigr) \in \F[t],$$
which is a polynomial of degree $4$.
More explicitly, if we split $p(z)=(z-x_1)(z-x_2)$ and $q(z)=(z-y_1)(z-y_2)$ in $\overline{\F}[z]$,
then
$$F_{p,q}(t)=\prod_{1 \leq i,j \leq 2}\bigl(t-(x_i-y_j)\bigr)=p(t+y_1)\,p(t+y_2)=q(x_1-t)\,q(x_2-t).$$
We set
$$E_{p,q}(u):=\underset{n \in \N}{\bigcup} \Ker F_{p,q}(u)^n \quad \text{and} \quad R_{p,q}(u):=\underset{n \in \N}{\bigcap} \im F_{p,q}(u)^n.$$
Hence, $V=E_{p,q}(u) \oplus R_{p,q}(u)$, and the endomorphism $u$ stabilizes both linear subspaces
$E_{p,q}(u)$ and $R_{p,q}(u)$.
The endomorphism $u$ is called \textbf{d-exceptional} with respect to $(p,q)$ (respectively, \textbf{d-regular} with respect to $(p,q)$)
whenever $E_{p,q}(u)=V$ (respectively, $R_{p,q}(u)=V$).
In other words, $u$ is d-exceptional (respectively, d-regular) with respect to $(p,q)$
if and only all the eigenvalues of $u$ in $\overline{\F}$ belong to $\Root(p)-\Root(q)$
(respectively, no eigenvalue of $u$ in $\overline{\F}$ belongs to $\Root(p)-\Root(q)$).

The endomorphism of $E_{p,q}(u)$ (respectively, of $R_{p,q}(u)$)
induced by $u$ is always d-exceptional (respectively, always d-regular) with respect to $(p,q)$
and we call it the \textbf{d-exceptional part} (respectively, the \textbf{d-regular part}) of $u$ with respect to $(p,q)$.

Finally, we set
$$\Lambda_{p,q}:=t^2+\bigl(2\bigl(p(0)+q(0)\bigr)-(\tr p)(\tr q)\bigr)\, t+F_{p,q}(0).$$
One sees that, for
$$v:=u^2-\delta u,$$
we have
$$\bigl(u-(x_1-y_1)\,\id_V\bigr)\bigl(u-(x_2-y_2)\,\id_V\bigr)=v+(x_1-y_1)(x_2-y_2)\,\id_V$$
and likewise
$$\bigl(u-(x_1-y_2)\,\id_V\bigr)\bigl(u-(x_2-y_1)\,\id_V\bigr)=v+(x_1-y_2)(x_2-y_1)\,\id_{V.}$$
Hence, a straightforward computation yields
$$F_{p,q}(u)=\Lambda_{p,q}(v).$$

\begin{Rem}\label{dconjugateremark}
Let $\calA$ be an $\F$-algebra, and let $a,b$ be elements of $\calA$ such that $p(a)=q(b)=0$.
Denote by $a^\star$ the $p$-conjugate of $a$
and by $b^\star$ the $q$-conjugate of $b$. Then, one sees that
\begin{equation}\label{u2-deltauvsconj}
(a-b)^2-\delta(a-b)=ab^\star+ba^\star-(p(0)+q(0))\,1_\calA.
\end{equation}
\end{Rem}

Our first basic result follows:

\begin{prop}\label{dseparEetR}
The endomorphism $u$ is a $(p,q)$-difference if and only if
both its d-exceptional part and its d-regular part are $(p,q)$-differences.
\end{prop}

The proof of this result will use the following Corollary of the Basic Commutation Lemma:

\begin{lemma}[Commutation Lemma]\label{dcommutelemma}
Let $p$ and $q$ be monic polynomials of $\F[t]$ with degree $2$, and let $a$ and $b$ be endomorphisms of a vector space $V$
such that $p(a)=q(b)=0$. Then, both $a$ and $b$ commute with $(a-b)^2-(\tr(p)-\tr(q))(a-b)$.
\end{lemma}

\begin{proof}
This follows from the Basic Commutation Lemma and from identity \eqref{u2-deltauvsconj}.
\end{proof}

\begin{proof}[Proof of Proposition \ref{dseparEetR}]
The ``if" part is obvious. Conversely, assume that $u$ is a $(p,q)$-difference,
and split $u=a-b$ where $a$ and $b$ are endomorphisms of $V$ such that $p(a)=0$ and $q(b)=0$.
By the Commutation Lemma, both $a$ and $b$ commute with $v:=u^2-\delta u$.
Hence, $a$ and $b$ commute with $\Lambda_{p,q}(v)=F_{p,q}(u)$, and it follows that both stabilize $E_{p,q}(u)$ and $R_{p,q}(u)$.
Denote by $a'$ and $b'$ (respectively, by $a''$ and $b''$) the endomorphisms of $E_{p,q}(u)$ (respectively, of $R_{p,q}(u)$)
induced by $a$ and $b$. Then, the d-exceptional part of $u$ is $a'-b'$, and the d-regular part of $u$ is $a''-b''$.
Obviously, $p$ annihilates $a'$ and $a''$, and $q$ annihilates $b'$ and $b''$, which yields
that both the d-exceptional and the d-regular part of $u$ are $(p,q)$-differences.
\end{proof}

From there, it is clear that classifying $(p,q)$-differences amounts to classifying the d-exceptional ones and the
d-regular ones. The easier classification is the latter: as we shall see, it involves little discussion on the specific polynomials
$p$ and $q$ under consideration (whether they are split or not over $\F$, separable or not over $\F$, etc).
In contrast, the classification of d-exceptional $(p,q)$-differences involves a tedious case-by-case study.

\subsection{Statement of the results}\label{dresultssection}

We are now ready to state our results.
We shall frame them in terms of direct-sum decomposability.

Let $u$ be an endomorphism of a nonzero finite-dimensional vector space $V$.
Assume that $V$ splits into $V_1 \oplus V_2$, and that each linear subspace $V_1$ and $V_2$ is stable under $u$ and nonzero,
and both induced endomorphisms $u_{|V_1}$ and $u_{|V_2}$ are $(p,q)$-differences.
Then, $u$ is obviously a $(p,q)$-difference.
In the event when such a decomposition exists we shall say that $u$ is
a \textbf{decomposable} $(p,q)$-difference, otherwise and if $u$ is a $(p,q)$-difference, we shall say that $u$ is
an \textbf{indecomposable} $(p,q)$-difference. Obviously, if $V$ is nonzero then every $(p,q)$-difference in $\End(V)$
is the direct sum of indecomposable ones. Hence, it suffices to describe the indecomposable $(p,q)$-differences.

Moreover, if a $(p,q)$-difference is indecomposable then by Proposition \ref{dseparEetR} it is either d-regular or d-exceptional.

In each one of the following tables, we give a set of matrices. Each matrix represents an indecomposable $(p,q)$-difference,
and every indecomposable $(p,q)$-difference in $\End(V)$ is represented by one of those matrices, in some basis.
Throughout the classification, we set
$$\delta:=\tr p-\tr q.$$

We start with d-regular $(p,q)$-differences. In that situation the classification is rather simple:

\begin{table}[H]
\begin{center}
\caption{The classification of indecomposable d-regular $(p,q)$-differences.}
\label{dfigure1}
\begin{tabular}{| c | c |}
\hline
Representing matrix & Associated data  \\
\hline
\hline
 & $n\in \N^*$, $r \in \F[t]$ irreducible and monic,  \\
$C\bigl(r^n(t^2-\delta t)\bigr)$ & $r(t^2-\delta t)$ has no root in $\Root(p)-\Root(q)$ \\
& $N_{\calW(p,q,y+p(0)+q(0))_\L}$ is isotropic over $\L:=\F[t]/(r)$ \\
& for $y:=\overline{t}$ in $\L$ \\
\hline
$C\bigl(r^n(t^2-\delta t)\bigr)$ & $n\in \N^*$, $r \in \F[t]$ irreducible and monic, \\
$\oplus$  & $r(t^2-\delta t)$ has no root in $\Root(p)-\Root(q)$ \\
$C\bigl(r^n(t^2-\delta t)\bigr)$ & $N_{\calW(p,q,y+p(0)+q(0))_\L}$ is non-isotropic over $\L:=\F[t]/(r)$ \\
& for $y:=\overline{t}$ in $\L$ \\
\hline
\end{tabular}
\end{center}
\end{table}

Remember, by Remark \ref{isotropicnormremark}, that the norm of $\calW(p,q,x)_R$
is isotropic whenever one of $p$ and $q$ splits in $\F[t]$.

Next, we tackle the indecomposable d-exceptional $(p,q)$-difference. Here, there are many cases to consider.
We start with the known ones, in which both $p$ and $q$ are split over $\F$.
The three situations are described in the following tables
(see \cite{Bothasquarezero} for Table \ref{dfigure2}, \cite{dSPidem2} for Table \ref{dfigure3}, and \cite{dSPsumoftwotriang} for Table \ref{dfigure4}).

\begin{table}[H]
\begin{center}
\caption{The classification of indecomposable d-exceptional $(p,q)$-differences: When both $p$ and $q$ are split with a double root.}
\label{dfigure2}
\begin{tabular}{| c | c |}
\hline
Representing matrix & Associated data  \\
\hline
\hline
$C\bigl((t-x)^n\bigr)$ & $n \in \N^*$, $x\in \Root(p)-\Root(q)$ \\
\hline
\end{tabular}
\end{center}
\end{table}

\begin{table}[H]
\begin{center}
\caption{The classification of indecomposable d-exceptional $(p,q)$-differences: When both $p$ and $q$ are split with simple roots.}
\label{dfigure3}
\begin{tabular}{| c | c |}
\hline
Representing matrix & Associated data  \\
\hline
\hline
 & $n \in \N^*$, \\
$C\bigl((t-x)^n\bigr) \oplus C\bigl((t-\delta+x)^n\bigr)$ & $x \in \Root(p)-\Root(q)$  \\
& such that $x \neq \delta -x$ \\
\hline
 & $n \in \N$, \\
$C\bigl((t-x)^{n+1}\bigr) \oplus C\bigl((t-\delta+x)^{n}\bigr)$ & $x \in \Root(p)-\Root(q)$  \\
& such that $x \neq \delta -x$ \\
\hline
 & $n\in \N^*$,  \\
$C\bigl((t-x)^{n}\bigr)$ & $x \in \Root(p)-\Root(q)$ \\
& such that $x =\delta-x$ \\
\hline
\end{tabular}
\end{center}
\end{table}

\begin{table}[H]
\begin{center}
\caption{The classification of indecomposable d-exceptional $(p,q)$-differences: When one of $p$ and $q$ is split with a double root
 and the other one is split with simple roots.}
 \label{dfigure4}
\begin{tabular}{| c | c |}
\hline
Representing matrix & Associated data  \\
\hline
\hline
$C\bigl((t-x)^n\bigr) \oplus C\bigl((t-\delta+x)^n\bigr)$ & $n\in \N^*$, $x \in \Root(p)-\Root(q)$  \\
\hline
$C\bigl((t-x)^{n+1}\bigr) \oplus C\bigl((t-\delta+x)^n\bigr)$ & $n\in \N$, $x \in \Root(p)-\Root(q)$  \\
\hline
$C\bigl((t-x)^{n+2}\bigr) \oplus C\bigl((t-\delta+x)^n\bigr)$ & $n\in \N$, $x \in \Root(p)-\Root(q)$  \\
 \hline
\end{tabular}
 \end{center}
\end{table}

Now, we state our new results on the d-exceptional $(p,q)$-differences.
We start with the case when $p$ is irreducible but $q$ is split.
There are two cases to consider, whether the two polynomials obtained by translating
$p$ along the roots of $q$ are equal or not.

\begin{table}[H]
\begin{center}
\caption{The classification of indecomposable d-exceptional $(p,q)$-differences: When $p$ is irreducible, $q=(t-y_1)(t-y_2)$
for some $y_1,y_2$ in $\F$, and $p(t+y_1)=p(t+y_2)$.}
\label{dfigure5}
\begin{tabular}{| c | c |}
\hline
Representing matrix & Associated data  \\
\hline
\hline
$C\bigl(p(t+y)^n\bigr)$ &  $n\in \N^*$, $y \in \Root(q)$  \\
\hline
\end{tabular}
 \end{center}
\end{table}

\begin{table}[H]
\begin{center}
\caption{The classification of indecomposable d-exceptional $(p,q)$-differences: When $p$ is irreducible, $q=(t-y_1)(t-y_2)$
for some $y_1,y_2$ in $\F$, and $p(t+y_1)\neq p(t+y_2)$.}
\label{dfigure6}
\begin{tabular}{| c | c |}
\hline
Representing matrix & Associated data  \\
\hline
\hline
$C\bigl(p(t+y_1)^n\bigr) \oplus C\bigl(p(t+y_2)^n\bigr)$ & $n\in \N^*$ \\
\hline
$C\bigl(p(t+y_1)^{n+1}\bigr) \oplus C\bigl(p(t+y_2)^{n}\bigr)$ & $n\in \N$ \\
\hline
$C\bigl(p(t+y_2)^{n+1}\bigr) \oplus C\bigl(p(t+y_1)^{n}\bigr)$ & $n\in \N$ \\
\hline
\end{tabular}
 \end{center}
\end{table}

Next, we consider the situation where $p$ and $q$ are both irreducible in $\F[t]$, with the same
splitting field.

\begin{table}[H]
\begin{center}
\caption{The classification of indecomposable d-exceptional $(p,q)$-differences: When $p$ and $q$ are irreducible with the same splitting field $\L$.}
\label{dfigure7}
\begin{tabular}{| c | c |}
\hline
Representing matrix & Associated data  \\
\hline
\hline
$C\Bigl(\bigl(t^2-\delta t+N_{\L/\F}(x-y)\bigr)^{n}\Bigr)$
&  $n\in \N^*$,  \\
$\oplus$ &  $x \in \Root(p)$, $y \in \Root(q)$ \\
$C\Bigl(\bigl(t^2-\delta t+N_{\L/\F}(x-y)\bigr)^{n}\Bigr)$ & with $x-y \not\in \F$  \\
\hline
$C\Bigl(\bigl(t^2-\delta t+N_{\L/\F}(x-y)\bigr)^{n+1}\Bigr)$
&  $n\in \N$,  \\
$\oplus$ &  $x \in \Root(p)$, $y \in \Root(q)$ \\
$C\Bigl(\bigl(t^2-\delta t+N_{\L/\F}(x-y)\bigr)^{n}\Bigr)$ & with $x-y \not\in \F$  \\
\hline
$C\bigl((t-(x-y))^n\bigr) \oplus C\bigl((t-(x-y))^n\bigr)$ & $n \in \N^*$, $x \in \Root(p)$, \\
& $y \in \Root(q)$ with $x-y \in \F$ \\
\hline
\end{tabular}
 \end{center}
\end{table}

We finish with the case when $p$ and $q$ are both irreducible, with distinct splitting fields.
There are two subcases to consider, whether $p$ and $q$ have the same discriminant or not.
Note that the case when $p$ and $q$ have the same discriminant and distinct splitting fields
can occur only if $\F$ has characteristic $2$.

\begin{table}[H]
\begin{center}
\caption{The classification of indecomposable d-exceptional $(p,q)$-differences: When $p$ and $q$ are irreducible with distinct splitting fields and distinct discriminants.}
\label{dfigure8}
\begin{tabular}{| c | c |}
\hline
Representing matrix & Associated data  \\
\hline
\hline
$C(F_{p,q}^n) \oplus C(F_{p,q}^n)$ & $n \in \N^*$ \\
\hline
$C(F_{p,q}^{n+1}) \oplus C(F_{p,q}^n)$ & $n \in \N$ \\
\hline
\end{tabular}
 \end{center}
\end{table}

\begin{table}[H]
\begin{center}
\caption{The classification of indecomposable d-exceptional $(p,q)$-differences: When $p$ and $q$ are irreducible with distinct splitting fields and the same discriminant.}
\label{dfigure9}
\begin{tabular}{| c | c |}
\hline
Representing matrix & Associated data  \\
\hline
\hline
$C\Bigl(\bigl(t^2-(\tr p)\,t +p(0)+q(0)\bigr)^n\Bigr)$ & \\
$\oplus$ & $n \in \N^*$ \\
$C\Bigl(\bigl(t^2-(\tr p)\,t +p(0)+q(0)\bigr)^n\Bigr)$ &  \\
\hline
\end{tabular}
 \end{center}
\end{table}

\subsection{An example}

Here, we consider the case when $\F$ is the field $\R$ of real numbers, and $p=q=t^2+1$.
In other words, we determine the endomorphisms of a finite-dimensional real vector space $V$ that split
into the difference of two endomorphisms $a$ and $b$ such that $a^2=b^2=-\id_V$.

Here, $\Root(p)-\Root(q)=\{2i,-2i,0\}$ and $\delta=0$.
Let us consider the indecomposable $(p,q)$-differences.
Let $r \in \R[t]$ be an irreducible polynomial such that $r(t^2-\delta t)=r(t^2)$ has no root in $\Root(p)-\Root(q)$.
We set $\L:=\R[t]/(r)$ and we note that the class $\overline{t}$ of $t$ in $\L$ is a root of $r$.
If $r$ has degree $2$, then $\L$ is isomorphic to $\C$, which is algebraically closed, and it follows that
the norm of $\calW(p,q,\overline{t}+2)_{\L}$ is isotropic (as is any quadratic form with dimension at least $2$
over an algebraically closed field).

Assume now that $r$ has degree $1$, and denote by $x$ its root. Since $r(t^2)$ has no root in $\Root(p)-\Root(q)$,
we see that $x \not\in \{-4,0\}$.
Then, the norm of $\calW(p,q,x+2)_{\R}$ reads
$$a I_4+b A+c B+dC \mapsto a^2+b^2+c^2+d^2+(x+2)bc-(x+2)ad,$$
which is equivalent to the orthogonal direct sum of two copies of the quadratic form
$$Q : (a,b) \mapsto a^2+(x+2) ab+b^2.$$
We have $Q(1,0)>0$, and the discriminant of $Q$ equals $\frac{(x+2)^2-4}{4}$. Therefore,
either $|x+2|<2$ and hence $Q$ is positive definite,
or $|x+2|>2$ and $Q$ is isotropic.
It follows that if $x \in (-4,0)$, then
the norm of $\calW(p,q,x+2)_{\R}$ is non-isotropic, otherwise it is isotropic.

Hence, the following table gives a complete list of indecomposable $(t^2+1,t^2+1)$-differences, where
the d-exceptional ones (given in the last three rows) are obtained thanks to Table \ref{dfigure7}:

\begin{table}[H]
\begin{center}
\caption{The classification of indecomposable $(t^2+1,t^2+1)$-differences over $\R$.}
\begin{tabular}{| c | c |}
\hline
Representing matrix & Associated data  \\
\hline
\hline
$C\bigl((t^2-x)^n\bigr) \oplus C\bigl((t^2-x)^n\bigr)$ & $n\in \N^*$, $x \in \left(-4,0\right)$ \\
\hline
$C\bigl((t^2-x)^n\bigr)$ & $n\in \N^*$, $x \in \left(-\infty,-4\right) \cup \left(0,+\infty\right)$ \\
\hline
$C\bigl((t^4+\alpha t^2+\beta)^n\bigr)$ & $n \in \N^*$, $(\alpha,\beta)\in \R^2$ with $\alpha^2<4\beta$ \\
\hline
$C(t^n) \oplus C(t^n)$ & $n \in \N^*$ \\
\hline
$C((t^2+4)^n) \oplus C((t^2+4)^{n})$ & $n \in \N^*$ \\
\hline
$C((t^2+4)^{n+1}) \oplus C((t^2+4)^{n})$ & $n \in \N$ \\
\hline
\end{tabular}
 \end{center}
\end{table}

\subsection{Strategy, and structure of the remainder of the section}

Here, the study has two very distinct parts:
the methods used to characterize the d-regular $(p,q)$-differences are very different from the ones
used to characterize the d-exceptional ones.
For the former, which are dealt with in Section \ref{d-regularsection}, the key is to analyze the structure of the algebra $\calW(p,q,x)_R$ when $R$ is
the local ring $\F[t]/(r^n)$ for some monic irreducible polynomial $r$ and some positive integer $n$.
We shall see in particular that if $n=1$ and the norm of $\calW(p,q,x)_R$ is non-degenerate, then
$\calW(p,q,x)_R$ is actually a quaternion algebra over the field $R$ (see Proposition \ref{structureofWprop}).
The classification of d-regular $(p,q)$-differences will be easily derived from such structural results.
In this study, no specific discussion on the pair $(p,q)$ is required.

The study of d-exceptional $(p,q)$-differences is carried out in the last three sections:
we shall first consider the case when $p$ is irreducible while $q$ is split (Section \ref{d-exceptionalsectionI}),
then the one when $p$ and $q$ are both irreducible with the same splitting field  (Section \ref{d-exceptionalsectionII}),
and we will finish with the case when $p$ and $q$ are both irreducible but with distinct splitting fields (Section \ref{d-exceptionalsectionIII}).
The case when $p$ is split and $q$ is irreducible is easily deduced from the one where $p$ is irreducible and $q$ is split
(indeed, $u$ is a $(p,q)$-difference if and only if $-u$ is a $(q,p)$-difference),
whereas the case when both $p$ and $q$ are split was completed in \cite{Bothasquarezero,dSPidem2,dSPsumoftwotriang}.
In all those sections, several subcases need to be considered.
Two basic techniques will be used:
\begin{itemize}
\item The first one is based on the Commutation Lemma: If $u=a-b$ for some endomorphisms $a$ and $b$ of $V$ such that $p(a)=q(b)=0$, then
$a$ and $b$ commute with $v:=u^2-\delta u$ and hence they can be seen as endomorphisms of the $\F[v]$-module $V$.
Yet, if $u$ is d-exceptional then some power of $\Lambda_{p,q}$ annihilates $v$; if in addition $\Lambda_{p,q}$ is separable
the endomorphism $v$ has a Jordan-Chevalley decomposition $v=s+n$ where $s$ is a semi-simple endomorphism of $V$,
$n$ is a nilpotent endomorphism of $V$, and $s$ and $n$ are polynomials in $v$.
It follows in that case that $a$ and $b$ are endomorphisms of the $\F[s]$-vector space $V$.

In connection with this idea, we shall need results on block-cyclic matrices that are featured in the appendix.
\item The second technique consists in extending the field of scalars to one in which $p$ and $q$
are split, allowing us to use the characterization of $(p,q)$-differences that is already known when both polynomials are split.
\end{itemize}

As the second technique uses the known classification of  d-exceptional $(p,q)$-differences when $p$ and $q$ are both split, it is useful
to restate it somewhat differently than in Section \ref{resultssection}: We will do this in Section \ref{d-exceptionalsection0}.

\subsection{Regular $(p,q)$-differences}\label{d-regularsection}

\subsubsection{The initial reduction}

We start with a partial result on d-regular $(p,q)$-differences.

\begin{prop}\label{basicd-regularreduction}
Let $p$ and $q$ be monic polynomials with degree $2$ over $\F$, and set $\delta:=\tr p-\tr q$.
Let $u$ be an endomorphism of a finite-dimensional vector space $V$ and assume that $u$ is a d-regular $(p,q)$-difference.
Then:
\begin{enumerate}[(a)]
\item Each invariant factor of $u$ has the form $r(t^2-\delta t)$ for some monic polynomial $r$.
\item In some basis of $V$, the endomorphism $u$ is represented by
a block-diagonal matrix in which every diagonal block has the form $C\bigl(r^n(t^2-\delta t)\bigr)$ for some irreducible monic polynomial $r$
and some positive integer $n$. We shall call such a matrix a \textbf{$(p,q)$-reduced canonical form} of $u$.
\end{enumerate}
\end{prop}

It can be easily shown that a $(p,q)$-reduced canonical form is unique up to a permutation of the diagonal blocks.

Before we prove Proposition \ref{basicd-regularreduction}, we need the corresponding special case when both polynomials $p$ and $q$ are split over $\F$:
this result can be obtained by collecting various results from \cite{Bothasquarezero}, \cite{dSPidem2} and \cite{dSPsumoftwotriang}, but we give a synthetic proof here (that uses the same technique as in those articles).

\begin{prop}\label{basicd-regularinvariants}
Let $p$ and $q$ be split monic polynomials with degree $2$ over $\F$, and set $\delta:=\tr p-\tr q$.
Let $u$ be an endomorphism of a finite-dimensional vector space $V$ and assume that $u$ is a d-regular $(p,q)$-difference.
Then, each invariant factor of $u$ is a polynomial in $t^2-\delta t$.
\end{prop}

The proof requires the following basic lemma, which is proved in \cite{dSPidem2} (see Lemma 14 there, in which the assumption
 that $\alpha$ and $\beta$ be nonzero is unnecessary) and which will be used later in this article:

\begin{lemma}\label{dblockmatrixlemma}
Let $r\in \F[t]$ be a monic polynomial with degree $n$, and let $x$ and $y$ be scalars. Then,
$$\begin{bmatrix}
x I_n & C(r) \\
I_n & y I_n
\end{bmatrix} \simeq C\bigl(r((t-x)(t-y))\bigr)$$
\end{lemma}

\begin{cor}\label{dcorblockinvariants}
Let $N$ be an arbitrary matrix of $\Mat_n(\F)$, and let $x$ and $y$ be scalars. Then, the invariant factors of
$$K(N):=\begin{bmatrix}
x I_n & N \\
I_n & y I_n
\end{bmatrix}$$
are polynomials in $(t-x)(t-y)$.
\end{cor}

\begin{proof}[Proof of Corollary \ref{dcorblockinvariants}]
We note that the similarity class of $K(N)$ depends only on that of $N$: Indeed, for all
$P \in \GL_n(\K)$, the invertible matrix $Q:=P \oplus P$ satisfies
$Q K(N)Q^{-1}=K(PNP^{-1})$.
Next, if $N$ splits into $N=N_1 \oplus \cdots \oplus N_r$ for some square matrices $N_1,\dots,N_r$ then, by permuting the basis
vectors, we gather that $K(N) \simeq K(N_1) \oplus \cdots \oplus K(N_r)$.
Considering the rational canonical form $N \simeq C(r_1)\oplus \cdots \oplus C(r_k)$, we
obtain
$$K(N) \simeq C\bigl(r_1\bigl((t-x)(t-y)\bigr)\bigr) \oplus \cdots \oplus  C\bigl(r_k\bigl((t-x)(t-y)\bigr)\bigr).$$
Moreover, the polynomials $r_1\bigl((t-x)(t-y)\bigr),\dots,r_k\bigl((t-x)(t-y)\bigr)$ are all monic and
$r_{i+1}\bigl((t-x)(t-y)\bigr)$ divides $r_{i}\bigl((t-x)(t-y)\bigr)$
for all $i \in \lcro 1,k-1\rcro$. Hence, we have found the invariant factors of $K(N)$, which proves the claimed result.
\end{proof}

\begin{proof}[Proof of Proposition \ref{basicd-regularinvariants}]
Let $a$ and $b$ be endomorphisms of $V$ such that $p(a)=q(b)=0$ and $u=a-b$.
Denote by $x$ (respectively, by $y$) an eigenvalue of $a$ (respectively, of $b$)
with maximal geometric multiplicity, and split $p(t)=(t-x)(t-x')$ and $q(t)=(t-y)(t-y')$.
We claim that
$$\dim \Ker(a-x\,\id_V) \geq \frac{n}{2}\cdot$$
Indeed, since $p(a)=0$ we have $\im(a- x'\, \id_V) \subset \Ker(a-x\,\id_V)$, which yields
$\dim \Ker(a-x\,\id_V)+\dim \Ker(a-x'\,\id_V) \geq n$. Since $\dim \Ker(a-x\,\id_V) \geq \dim \Ker(a-x'\,\id_V)$,
the claimed inequality follows.

Likewise, $\dim \Ker(b-y\,\id_V) \geq \frac{n}{2}\cdot$
Since $u$ is d-regular, any eigenspace of $a$ is linearly disjoint from any eigenspace of $b$.
In particular, $\Ker(a-x \id_V) \cap \Ker(b-y\,\id_V)=\{0\}$.
It follows that $\dim \Ker(a-x\,\id_V)=\frac{n}{2}=\dim \Ker(b-y\,\id_V)$, $n$ is even and
$V=\Ker(a-x \id_V) \oplus \Ker(b-y\,\id_V)$.
Next, we deduce that $\frac{n}{2}=\dim \im (a-x \id_V)$ and $\dim \Ker(a-x' \id_V) \leq \frac{n}{2}$ by choice of $x$.
However, $\im(a-x\id_V) \subset \Ker(a-x' \id_V)$, and hence it follows that
$\im(a-x\id_V) = \Ker(a-x' \id_V)$. Likewise, $\im(b-y\id_V) = \Ker(b-y' \id_V)$,
and it follows that $x'$ has geometric multiplicity $\frac{n}{2}$ with respect to $a$,
and ditto for $y'$ with respect to $b$.
In turn, this shows that $\im(a-x'\id_V)=\Ker(a-x\id_V)$ and $\im(b-y'\id_V)=\Ker(b-y\id_V)$,
and any eigenspace of $a$ is a complementary subspace of any eigenspace of $b$.

Let us write $s:=\frac{n}{2}$ and choose a basis $(e_1,\dots,e_s)$ of
$\Ker(b-y \id_V)$. Then, we have $V=\Ker(b-y \id_V) \oplus \Ker(a-x \id_V)$, whence
$(e_{s+1},\dots,e_n):=((a-x\id_V)(e_1),\dots,(a-x\id_V)(e_s))$ is a basis of
$\im(a-x \id_V)=\Ker (a-x' \id_V)$.
Since $\Ker(b-y \id_V) \oplus \Ker (a-x' \id_V)=V$, we deduce that
$\bfB:=(e_1,\dots,e_n)$ is a basis of $V$. Obviously
$$\Mat_\bfB(a)=\begin{bmatrix}
x I_s & 0 \\
I_n & x' I_s
\end{bmatrix}.$$
On the other hand, since $\Ker(b-y \id_V)=\im(b-y' \id_V)$, we find
$$\Mat_\bfB(b)=\begin{bmatrix}
y I_s & N \\
0 & y' I_s
\end{bmatrix}$$
for some matrix $N \in \Mat_s(\F)$.
Hence,
$$\Mat_\bfB(u)=\begin{bmatrix}
(x-y) I_s & -N \\
I_n & (x'-y') I_s
\end{bmatrix}.$$
Since $(t-(x-y))(t-(x'-y'))=t^2-\delta t+(x-y)(x'-y')$, we conclude from Lemma \ref{dblockmatrixlemma} that
all the invariant factors of $u$ are polynomials in $t^2-\delta t$.
\end{proof}

\begin{proof}[Proof of Proposition \ref{basicd-regularreduction}]
We start with point (a). Let us extend the field of scalars to $\overline{\F}$.
The resulting extension $\overline{u}$ of $u$ is still a $(p,q)$-difference.
Hence, by Corollary \ref{dcorblockinvariants} its invariant factors are $p_1(t^2-\delta t),\dots,p_r(t^2-\delta t)$
for some monic polynomials $p_1,\dots,p_r$ of $\overline{\F}[t]$ such that $p_{i+1}$ divides $p_i$ for all $i \in \lcro 1,r-1\rcro$.
Yet, the invariant factors of $\overline{u}$ are known to be the ones of $u$.
Finally, given a monic polynomial $h \in \overline{\F}[t]$ such that $h(t^2-\delta t)\in \F[t]$,
we obtain by downward induction that all the coefficients of $h$ belong to $\F$:
Indeed, if we write $h(t)=t^N-\underset{i=0}{\overset{N-1}{\sum}} \alpha_i\,t^i$ and we know that $\alpha_{N-1},\dots,\alpha_{k+1}$ all belong to $\F$
for some $k \in \lcro 0,N-1\rcro$, then $\underset{i=0}{\overset{k}{\sum}} \alpha_i\,(t^2-\delta t)^i=(t^2-\delta t)^N-\underset{i=k+1}{\overset{N-1}{\sum}} \alpha_i\,(t^2-\delta t)^i$
belongs to $\F[t]$, and by considering the coefficient on $t^{2k}$, we gather that $\alpha_k \in \F$.
It follows that $p_1,\dots,p_r$ all belong to $\F[t]$, which completes the proof of statement (a).

From point (a), we easily derive point (b): indeed, consider an invariant factor
$r(t^2-\delta t)$ of $u$ for some monic polynomial $r \in \F[t]$. Then, we split $r=r_1^{n_1}\cdots r_k^{n_k}$ where $r_1,\dots,r_k$ are pairwise distinct
irreducible monic polynomials of $\F[t]$, and $n_1,\dots,n_k$ are positive integers.
Then, the polynomials $r_1^{n_1}(t^2-\delta t),\dots,r_k^{n_k}(t^2-\delta t)$ are pairwise coprime and their product equals
$r(t^2-\delta t)$, whence
$$C\bigl(r(t^2-\delta t)\bigr) \simeq C\bigl(r_1^{n_1}(t^2-\delta t)\bigr) \oplus \cdots \oplus C\bigl(r_k^{n_k}(t^2-\delta t)\bigr).$$
Using point (a), we deduce that statement (b) holds true.
\end{proof}

\subsubsection{Application of the $\calW(p,q,x)$ algebra to the characterization of d-regular $(p,q)$-differences}

We are now ready to complete our study of d-regular $(p,q)$-differences.
An additional definition will be useful in this prospect:

\begin{Def}
Let $p$ and $q$ be monic polynomials with degree $2$ in $\F[t]$.
Set $\delta:=\tr(p)-\tr(q)$.
Let $r$ be an irreducible monic polynomial of $\F[t]$, and set
$\L:=\F[t]/(r)$. Denote by $\overline{t}$ the class of $t$ in $\L$.
We say that $r$ has:
\begin{itemize}
\item \textbf{Type $1$ with respect to $(p,q)$} if $r(t^2-\delta t)$ has no root in $\Root(p)-\Root(q)$ and
the norm of $\calW\bigl(p,q,\overline{t}+p(0)+q(0)\bigr)_{\L}$ is isotropic.
\item \textbf{Type $2$ with respect to $(p,q)$} if $r(t^2-\delta t)$ has no root in $\Root(p)-\Root(q)$ and
the norm of $\calW\bigl(p,q,\overline{t}+p(0)+q(0)\bigr)_{\L}$ is nonisotropic.
\end{itemize}
\end{Def}

First of all, we use the structural results on $\calW(p,q,x)_R$ to obtain various $(p,q)$-differences.
Our first result is actually not restricted to d-regular $(p,q)$-differences and will be used later in the article.

\begin{lemma}[Duplication Lemma]\label{dduplicationlemma}
Let $p$ and $q$ be monic polynomials of $\F[t]$ with degree $2$, and set $\delta:=\tr p-\tr q$.
Let $r$ be a nonconstant monic polynomial of $\F[t]$.
Then, $C\bigl(r(t^2-\delta t)\bigr) \oplus C\bigl(r(t^2-\delta t)\bigr)$ is a $(p,q)$-difference.
\end{lemma}

\begin{proof}
We work with the commutative $\F$-algebra $R:=\F[C(r)]$, which is isomorphic to the quotient ring $\F[t]/(r)$,
and with the element $x:=\bigl(p(0)+q(0)\bigr)\,1_R+C(r)$.
Using $q(B)=0$, it is easily seen that
$\bfB:=(I_2,A-B,B,(A-B)B)$ is still a basis of the free $R$-module $\calW(p,q,x)$.
Then, we consider the endomorphisms $a : X \mapsto AX$ and $b:X \mapsto BX$ of $\calW(p,q,x)$.
Since $p(A)=0$ and $q(B)=0$, we get $p(a)=0$ and $q(b)=0$.
Denote by $A'$ and $B'$ the respective matrices of $a$ and $b$ in $\bfB$.
Using $(A-B)^2=\delta (A-B)+\bigl(x-p(0)-q(0)\bigr) I_4$, we get
that
$$A'-B'=\begin{bmatrix}
0 & C(r) & 0 & 0  \\
1_R & \delta\,1_R & 0 & 0 \\
0 & 0 & 0 & C(r) \\
0 & 0 & 1_R & \delta\,1_R
\end{bmatrix},$$
whence the matrix $A'-B'$ of $\Mat_{4d}(\F)$ (where $d$ denotes the degree of $r$) is similar to $C\bigl(r(t(t-\delta))\bigr) \oplus C\bigl(r(t(t-\delta))\bigr)$ by Lemma \ref{dblockmatrixlemma}.
Since $p(A')=0$ and $q(B')=0$, the conclusion follows.
\end{proof}

Our next result deals with certain companion matrices that are associated with irreducible polynomials of Type 1.

\begin{lemma}\label{dType1blocklemma}
Let $p$ and $q$ be monic polynomials of $\F[t]$ with degree $2$, and set $\delta:=\tr p-\tr q$.
Let $r$ be an irreducible monic polynomial of $\F[t]$ of Type 1 with respect to $(p,q)$.
Then, for all $n \in \N^*$, the companion matrix $C\bigl(r^n(t^2-\delta t)\bigr)$ is a $(p,q)$-difference.
\end{lemma}

\begin{proof}
We naturally identify $\L$ with the subalgebra $\F[C(r)]$ of $\Mat_d(\F)$, where $d$ denotes the degree of $r$.
Let $n \in \N^*$. Set $R:=\F[C(r^n)]$, seen as a subalgebra of $\Mat_{nd}(\F)$, and set $x:=C(r^n)+\bigl(p(0)+q(0)\bigr)I_{nd}$.
The $\F$-algebra $R$ is isomorphic to $\F[t]/(r^n)$.
By Proposition \ref{localstructureprop}, it follows that $\calW(p,q,x)_R$ is isomorphic to $\Mat_2(R)$.
We choose an isomorphism $\varphi : \calW(p,q,x)_R \overset{\simeq}{\longrightarrow} \Mat_2(R)$,
and we set $a:=\varphi(A)$ and $b:=\varphi(B)$. Note that $p(a)=q(b)=0$,
whereas $c:=a-b$ satisfies $c(c-\delta I_2)=\bigl(x-(p(0)+q(0))1_R\bigr)\,I_2$.
The mapping $X \in R^2 \mapsto cX \in R^2$ yields an endomorphism $\overline{c}$ of $\L^2$.
Yet, $\overline{c}$ cannot be a scalar multiple of the identity (otherwise, $(I_2,a,b,ab)$ would not be a basis of the $R$-module $\Mat_2(R)$).
Hence, we find a vector $e$ of $\L^2$ such that $\bigl(e,\overline{c}(e)\bigr)$ is a basis of $\L^2$.
Lifting $e$ to a vector $E$ of $R^2$, we deduce that $(E,cE)$ is a basis of the $R$-module $R^2$.
Hence, composing $\varphi$ with an additional interior automorphism of the $R$-algebra $\Mat_2(R)$,
we see that no generality is lost in assuming that the first column of $c$
reads $\begin{bmatrix}
0_R \\
1_R
\end{bmatrix}$. Then, $c(c-\delta I_2)=\bigl(x-(p(0)+q(0))1_R\bigr)\,I_2$ yields
$$c=\begin{bmatrix}
0 & C(r^n) \\
1_R & \delta 1_R
\end{bmatrix}.$$
It follows that the matrix
$\begin{bmatrix}
0 & C(r^n) \\
I_{nd} & \delta I_{nd}
\end{bmatrix}$ of $\Mat_{2nd}(\F)$ is a $(p,q)$-difference. By Lemma \ref{dblockmatrixlemma}, this matrix is similar to
$C\bigl(r^n(t^2-\delta t)\bigr)$, which completes the proof.
\end{proof}

Combining Lemma \ref{dduplicationlemma} with Lemma \ref{dType1blocklemma}, we conclude that the implication (iii) $\Rightarrow$ (i)
in the following theorem holds true.

\begin{theo}[Classification of d-regular $(p,q)$-differences]\label{d-regulartheo}
Let $p$ and $q$ be monic polynomials of degree $2$ in $\F[t]$. Let $u$ be an endomorphism of a finite-dimensional vector space $V$ over
$\F$. Assume that $u$ is d-regular with respect to $(p,q)$ and set $\delta:=\tr(p)-\tr(q)$.
The following conditions are equivalent:
\begin{enumerate}[(i)]
\item The endomorphism $u$ is a $(p,q)$-difference.
\item The invariant factors of $u$ read $p_1(t^2-\delta t), \dots, p_{2n-1}(t^2-\delta t),p_{2n}(t^2-\delta t),\dots$
where, for every irreducible monic polynomial $r \in \F[t]$
that has Type $2$ with respect to $(p,q)$ and every positive integer $n$ the polynomials $p_{2n-1}$ and $p_{2n}$ have the
same valuation with respect to $r$.
\item There is a basis of $V$ in which $u$ is represented by a block-diagonal matrix in which every diagonal block
equals either $C\bigl(r^n(t^2-\delta t)\bigr)$ for some irreducible monic polynomial $r \in \F[t]$ of Type $1$ with respect to $(p,q)$
and some $n \in \N^*$,
or $C\bigl(r^n(t^2-\delta t)\bigr) \oplus C\bigl(r^n(t^2-\delta t)\bigr)$ for some irreducible monic polynomial $r \in \F[t]$
and some $n \in \N^*$.
\end{enumerate}
\end{theo}

Note that this result, combined with the observation that
$C\bigl(r^n(t^2-\delta t)\bigr)$ is d-regular with respect to $(p,q)$ for every monic polynomial $r \in \F[t]$
such that $r(t^2-\delta t)$ has no root in $\Root(p)-\Root(q)$, yields the classification of indecomposable d-regular $(p,q)$-differences as given
in Table \ref{dfigure1}.
Moreover, by using the method from the last part of the proof of Proposition \ref{basicd-regularreduction},
it is easily seen that conditions (ii) and (iii) are equivalent.

In order to conclude on Theorem \ref{d-regulartheo}, it only remains to prove that condition (i) implies condition (ii),
which we shall now do thanks to the structural results on $\calW(p,q,x)_R$.

\begin{proof}[Proof of the implication (i) $\Rightarrow$ (ii)]
Let us assume that $u$ is a $(p,q)$-difference.
Let $r$ be an irreducible monic polynomial of $\F[t]$, with Type 2 with respect to $(p,q)$. Let $n \in \N^*$.
All we need is to prove that, in the canonical form of $u$ from Proposition \ref{basicd-regularreduction},
the number $m$ of diagonal blocks that equal $C\bigl(r^n(t^2-\delta t)\bigr)$ is even.

Let us choose endomorphisms $a$ and $b$ of $V$ such that $u=a-b$ and $p(a)=q(b)=0$.
By the Commutation Lemma (Lemma \ref{dcommutelemma}), we know that $a$ and $b$ commute with $v:=u^2-\delta u$, and hence
all three endomorphisms $a,b,u$ yield endomorphisms $\overline{a}$, $\overline{b}$ and $\overline{u}$
of the vector space $E:=\Ker \bigl(r^n(v)\bigr)/\Ker \bigl(r^{n-1}(v)\bigr)$ such that $\overline{u}=\overline{a}-\overline{b}$, and
$r$ annihilates $\overline{v}:=\overline{u}^2-\delta \overline{u}$.
Again, $\overline{a}$ and $\overline{b}$ commute with $\overline{v}$, and hence
they are endomorphisms of the $\F[\overline{v}]$-module $E$.
Since $r$ is irreducible, we have $\F[\overline{v}] \simeq \F[t]/(r)$, and
$\L:=\F[\overline{v}]$ is a field. We shall write $y:=\overline{v}$, which we see as an element of $\L$.

On the other hand, we see that $2m\deg(r)$ is the dimension of the $\F$-vector space $E$, and hence
$2m$ is the dimension of the $\L$-vector space $E$.

Using the structure of $\L$-vector space, we can write $\overline{u}^2-\delta \overline{u}=y\,\id_E$,
and hence $\overline{a}$ and $\overline{b}$ yield a representation of the $\L$-algebra $\calW\bigl(p,q,y+p(0)+q(0)\bigr)_{\L}$
on the $\L$-vector space $E$. By Corollary \ref{structureofWcor}, the algebra
$\calW\bigl(p,q,y+p(0)+q(0)\bigr)_{\L}$ is a $4$-dimensional skew-field over $\L$, whence the $\L$-vector space
$E$ is isomorphic to a power of $\calW\bigl(p,q,y+p(0)+q(0)\bigr)_{\L}$, and it follows that its dimension over $\L$ is
a multiple of $4$. Therefore, $m$ is a multiple of $2$, which completes the proof.
\end{proof}

Therefore, we have completed the classification of d-regular $(p,q)$-differences.

\subsection{Exceptional $(p,q)$-differences (I): When both $p$ and $q$ are split}\label{d-exceptionalsection0}

Here, we restate the known results on the characterization of d-exceptional $(p,q)$-differences in the
case when both polynomials $p$ and $q$ are split. The three following results are taken from \cite{Bothasquarezero}, \cite{dSPsumoftwotriang} and \cite{dSPidem2} (respectively):

\begin{theo}\label{theoPandQsplitSoleRoot}
Let $p$ and $q$ be split monic polynomials with degree $2$ in $\F[t]$, each with a sole root.
Write $p(t)=(t-x)^2$ and $q(t)=(t-y)^2$.
Then, an endomorphism $u$ of a finite-dimensional vector space over $\F$  is a d-exceptional $(p,q)$-difference if and only if
the characteristic polynomial of $u$ is a power of $t-(x-y)$.
\end{theo}

\begin{theo}\label{theoPandQsplitSoleRootforQ}
Let $p$ and $q$ be split monic polynomials with degree $2$ in $\F[t]$.
Assume that $p$ has two simple roots $x_1$ and $x_2$ and that $q$ has a double root $y$.
Let $u$ be an endomorphism of a finite-dimensional vector space over $\F$.
Then, the following conditions are equivalent:
\begin{enumerate}[(i)]
\item $u$ is a d-exceptional $(p,q)$-difference.
\item $u$ is triangularizable with eigenvalues in $\{x_1-y,x_2-y\}$, and the sequences
$\bigl(n_k(u,x_1-y)\bigr)_{k \geq 1}$ and $\bigl(n_k(u,x_2-y)\bigr)_{k \geq 1}$ are $2$-intertwined.
\end{enumerate}
\end{theo}

\begin{theo}\label{theoPandQsplitSingleroots}
Let $p$ and $q$ be split monic polynomials with degree $2$ in $\F[t]$, both with simple roots.
Let $u$ be an endomorphism of a finite-dimensional vector space over $\F$.
Set $\delta:=\tr p-\tr q$. Then, the following conditions are equivalent:
\begin{enumerate}[(i)]
\item $u$ is a d-exceptional $(p,q)$-difference.
\item $u$ is triangularizable with eigenvalues in $\Root(p)-\Root(q)$, and
for every $x \in \Root(p)-\Root(q)$ such that $x \neq \delta-x$, the sequences
$\bigl(n_k(u,x)\bigr)_{k \geq 1}$ and $\bigl(n_k(u,\delta-x)\bigr)_{k \geq 1}$ are $1$-intertwined.
\end{enumerate}
\end{theo}

\subsection{Exceptional $(p,q)$-differences (II): When $p$ is irreducible but $q$ is not}\label{d-exceptionalsectionI}

In this section and the following ones, we fix two monic polynomials $p$ and $q$ of degree $2$ in $\F[t]$.
We set $\delta:=\tr p-\tr q$.

Here, we assume that $p$ is irreducible but not $q$.
We split $q(t)=(t-y_1)(t-y_2)$ with $y_1$ and $y_2$ in $\F$.
Then, $F_{p,q}(t)=p(t+y_1)p(t+y_2)$, the polynomials $p(t+y_1)$ and $p(t+y_2)$
are monic with degree $2$ in $\F[t]$ and they are irreducible.

Hence:
\begin{itemize}
\item Either $p(t+y_1) \neq p(t+y_2)$, in which case an endomorphism of
a finite-dimensional vector space over $\F$ is d-exceptional with respect to $(p,q)$ if and only if
the irreducible monic divisors of its minimal polynomial belong to $\bigl\{p(t+y_1),p(t+y_2)\bigr\}$;
\item Or $p(t+y_1) = p(t+y_2)$, in which case an endomorphism of
a finite-dimensional vector space over $\F$ is d-exceptional with respect to $(p,q)$ if and only if
its minimal polynomial is a power of $p(t+y_1)$.
\end{itemize}

Note that, given $\lambda \in \F$, one has $p(t+\lambda)=p(t)$ if and only if $\lambda=0$
or $\car(\F)=2$ and $\lambda=\tr(p)$. Hence,
$p(t+y_1)=p(t+y_2)$ if and only if $y_1=y_2$ or $\car(\F)=2$ and $\tr(p)=\tr(q)$.

\subsubsection{A common lemma}

\begin{lemma}\label{commonQsplitsnotP}
Let $p$ be a  monic polynomial with degree $2$ over $\F$, and let $y_1$ and $y_2$ be scalars in $\F$.
Set $q:=(t-y_1)(t-y_2)$. \\
Then, for all $n \in \N$, the companion matrices $C\bigl(p(t+y_1)^{n+1}\,p(t+y_2)^n\bigr)$ and $C\bigl(p(t+y_1)^{n}\,p(t+y_2)^{n}\bigr)$
are $(p,q)$-differences.
\end{lemma}

\begin{proof}
Let $s$ be a non-negative integer.
We extend $(y_1,y_2)$ into a $2$-periodical sequence $(y_k)_{k \geq 1}$.
Set $K:=\begin{bmatrix}
0 & 1 \\
0 & 0
\end{bmatrix} \in \Mat_2(\F)$.
Set $$A_s:=\bigl(C\bigl(p(t+y_1)\bigr)+ y_1 I_2\bigr) \oplus \cdots \oplus \bigl(C\bigl(p(t+y_s)\bigr)+y_s I_2\bigr) \in \Mat_{2s}(\F)$$
and
$$B_s:=\begin{bmatrix}
y_1 I_2 & 0 & \cdots & \cdots & (0) \\
-K & y_2 I_2 & \ddots & & \vdots \\
0 & \ddots & \ddots & \ddots & \vdots \\
\vdots & & \ddots & y_{s-1} I_2 & 0 \\
(0) & \cdots & 0 & -K & y_s I_2
\end{bmatrix} \in \Mat_{2s}(\F).$$
Note that $p(A_s)=0=q(B_s)$ and that
$$A_s-B_s= \begin{bmatrix}
C\bigl(p(t+y_1)\bigr) & 0 & \cdots & \cdots & (0) \\
K & C\bigl(p(t+y_2)\bigr) & \ddots & & \vdots \\
0 & \ddots & \ddots & \ddots & \vdots \\
\vdots & & \ddots & C\bigl(p(t+y_{s-1})\bigr) & 0 \\
(0) & \cdots & 0 & K & C\bigl(p(t+y_{s})\bigr)
\end{bmatrix}.$$
Hence, we recognize that $A_s-B_s$ is cyclic with characteristic polynomial $\underset{k=1}{\overset{s}{\prod}} p(t+y_k)$.
If $s$ is even, this characteristic polynomial equals $p(t+y_1)^{s/2} p(t+y_2)^{s/2}$;
otherwise it equals $p(t+y_1)^{(s+1)/2} p(t+y_2)^{(s-1)/2}$.
Varying $s$ yields the claimed result.
\end{proof}

\subsubsection{The case when $p(t+y_1)=p(t+y_2)$}

The following theorem is an obvious consequence of Lemma \ref{commonQsplitsnotP} and of the considerations of the start of Section
\ref{d-exceptionalsectionI}:

\begin{theo}\label{theoQsplitswithdouble}
Let $p$ and $q$ be monic polynomials with degree $2$ of $\F[t]$.
Assume that $p$ is irreducible over $\F$ and that $q$ splits over $\F$.
Assume furthermore that $q$ has a double root or that $\car(\F)=2$ and $\tr(q)=\tr(p)$.
Choose a root $y$ of $q$ in $\F$.
Then, given an endomorphism $u$ of a finite-dimensional vector space $V$ over $\F$, the following conditions are equivalent:
\begin{enumerate}[(i)]
\item $u$ is a d-exceptional $(p,q)$-difference.
\item The minimal polynomial of $u$ is a power of $p(t+y)$.
\item In some basis of $V$, the endomorphism $u$ is represented by a block-diagonal matrix in which every diagonal block
equals $C\bigl(p(t+y)^n\bigr)$ for some positive integer $n$.
\end{enumerate}
\end{theo}

\subsubsection{The case when $p(t+y_1) \neq p(t+y_2)$}

\begin{theo}
Let $p$ and $q$ be monic polynomials with degree $2$ of $\F[t]$.
Assume that $p$ is irreducible over $\F$ and that $q$ splits over $\F$ with distinct roots $y_1$ and $y_2$.
Assume furthermore that if $\car(\F)=2$ then $\tr(p) \neq \tr(q)$. Set $\delta:=\tr p-\tr q$.
Let $u$ be an endomorphism of  a finite-dimensional vector space $V$ over $\F$. The following conditions are then equivalent:
\begin{enumerate}[(i)]
\item $u$ is a  d-exceptional $(p,q)$-difference.
\item The irreducible monic divisors of the minimal polynomial of $u$ belong to $\bigl\{p(t+y_1),p(t+y_2)\bigr\}$.
Moreover, in writing the invariant factors of $u$ as $r_1=p(t+y_1)^{\alpha_1}p(t+y_2)^{\beta_1},\dots,r_k=p(t+y_1)^{\alpha_k}p(t+y_2)^{\beta_k},\dots$, we have $|\alpha_k-\beta_k| \leq 1$ for all $k \in \N^*$.

\item In some basis of $V$, the endomorphism $u$ is represented by a block-diagonal matrix in which every diagonal block
equals $C\bigl(p(t+y_1)^n\bigr) \oplus C\bigl(p(t+y_2)^n\bigr)$,
$C\bigl(p(t+y_1)^n\bigr) \oplus C\bigl(p(t+y_2)^{n-1}\bigr)$
or $C\bigl(p(t+y_1)^{n-1}\bigr) \oplus C\bigl(p(t+y_2)^n\bigr)$ for some positive integer $n$.
\end{enumerate}
\end{theo}

\begin{proof}
Lemma \ref{commonQsplitsnotP} shows that condition (iii) implies condition (i).
Moreover, it is also clear that condition (ii) implies condition (iii).

Now, we assume that property (i) holds, and we aim at proving property (ii).
Note first that $p$ and $q$ have distinct discriminants.
Indeed, if $\car(\F) \neq 2$ then this comes from the assumption that
$p$ is irreducible and $q$ is not; Otherwise this comes from the assumption that $\tr p \neq \tr q$.

For $k \in \N^*$, denote by $a_k$ (respectively, by $b_k$)
the number of elementary factors of $u$ of the form $p(t+y_1)^l$ (respectively, of the form $p(t+y_2)^l$) for some integer $l \geq k$.
Let us temporarily assume that $(a_k)_{k \geq 1}$ and $(b_k)_{k \geq 1}$ are $1$-intertwined.
Then, we claim that condition (ii) holds.
Indeed, since $u$ is d-exceptional with respect to $(p,q)$ we already know that the monic irreducible divisors of its minimal polynomial
are among $p(t+y_1)$ and $p(t+y_2)$.
Next, denote by $r_1,\dots,r_k,\dots$ the invariant factors of $u$, and further write
$r_k=p(t+y_1)^{\alpha_k}p(t+y_2)^{\beta_k}$ for some non-negative integers $\alpha_k$ and $\beta_k$.
Assume that, for some positive integer $k$, we have $\alpha_k>\beta_k+1$. Setting $s:=\beta_k+2$, we get that
$$a_s \geq k \quad \text{and} \quad b_{s-1}<k,$$
which contradicts the assumption that $(a_i)_{i \geq 1}$ and $(b_i)_{i \geq 1}$ be $1$-intertwined. Hence, $\alpha_k \leq \beta_k+1$, and likewise
$\beta_k \leq \alpha_k+1$. Hence condition (ii) holds.

It remains to prove that $(a_k)_{k \geq 1}$ and $(b_k)_{k \geq 1}$ are $1$-intertwined.
To do so, we need to distinguish between two cases.
In both of them, we denote by $\L$ the splitting field of $p$ (as defined by $\L:=\F[t]/(p)$) and
we consider the $\L$-vector space $V^{\L}:=V \otimes_{\F} \L$ and the endomorphism
$u^{\L}$ of $V^{\L}$ deduced from $u$ by extending the field of scalars. Note that $u^{\L}$ is a $(p,q)$-difference
in the algebra of all endomorphisms of $V^\L$.

\begin{itemize}
\item Assume first that $p$ has two distinct roots $x_1$ and $x_2$ in $\L$. \\
Since $p$ and $q$ have distinct discriminants, we have $x_1-y_1 \neq x_2-y_2$.
It follows that, for all $k \in \N^*$,
$$n_k\bigl(u^\L,x_1-y_1\bigr)=a_k \quad \text{and} \quad n_k\bigl(u^\L,x_2-y_2\bigr)=b_k,$$
and Theorem \ref{theoPandQsplitSingleroots} yields that
the sequences $\Bigl(n_k\bigl(u^\L,x_1-y_1\bigr)\Bigr)_{k \geq 1}$ and $\Bigl(n_k\bigl(u^\L,x_2-y_2\bigr)\Bigr)_{k \geq 1}$
are $1$-intertwined.

\item Assume now that $p$ has a sole root $x$ in its splitting field $\L$ (note that this can only happen
if $\car(\F)=2$). \\
Then, as $x-y_1 \neq x-y_2$, one sees that, for all $k \in \N^*$,
$$n_{2k}\bigl(u^\L,x-y_1\bigr)=a_k \quad \text{and} \quad
n_{2k}\bigl(u^\L,x-y_2\bigr)=b_k.$$
Here, Theorem \ref{theoPandQsplitSoleRootforQ} applies to $-u^{\L}$ and to the pair $(q,p)$, and hence the sequences
$\Bigl(n_k\bigl(u^\L,x-y_1\bigr)\Bigr)_{k \geq 1}$ and
$\Bigl(n_k\bigl(u^\L,x-y_2\bigr)\Bigr)_{k \geq 1}$ are $2$-intertwined.
\end{itemize}
Hence, in any case we see that $(a_k)_{k \geq 1}$ and $(b_k)_{k \geq 1}$ are $1$-intertwined, which completes the proof that condition (i)
implies condition (ii).
\end{proof}

\subsection{Exceptional $(p,q)$-differences (III): When $p$ and $q$ are irreducible with the same splitting field}\label{d-exceptionalsectionII}

Here, we assume that $p$ and $q$ are both irreducible, with the same splitting field $\L$ over $\F$ (in $\overline{\F}$).
Denote by $\sigma$ the non-identity automorphism of $\L$ over $\F$ if $\L$ is separable over $\F$,
otherwise set $\sigma:=\id_\L$. In any case, splitting $p(t)=(t-x_1)(t-x_2)$ and $q(t)=(t-y_1)(t-y_2)$
over $\L$, we find that $\sigma$ exchanges $x_1$ and $x_2$, and that it exchanges $y_1$ and $y_2$.
Hence, $(x_1-y_1)(x_2-y_2)=(x_1-y_1)\sigma(x_1-y_1)=N_{\L/\F}(x_1-y_1)$ and
likewise $(x_1-y_2)(x_2-y_1)=N_{\L/\F}(x_1-y_2)$.
Hence, with $\delta:=\tr(p)-\tr(q)$, we have
$$F_{p,q}(t)=\bigl(t^2-\delta t+N_{\L/\F}(x_1-y_1)\bigr)\bigl(t^2-\delta t+N_{\L/\F}(x_1-y_2)\bigr)$$
and
$$\Lambda_{p,q}(t)=\bigl(t+N_{\L/\F}(x_1-y_1)\bigr)\bigl(t+N_{\L/\F}(x_1-y_2)\bigr).$$
Next, assume that $F_{p,q}$ has a root in $\F$. Then,
we have respective roots $x$ and $y$ of $p$ and $q$, together with some scalar $d\in \F$ such that $x=y+d$.
Hence, $\sigma(x)=\sigma(y)+d$ and it follows that $q(t)=p(t+d)$.

Conversely, assume that $q(t)=p(t+d)$ for some $d \in \F$.
Then, we see that
$$F_{p,q}(t)=(t-d)^2 \bigl((t-d)^2-\Delta\bigr)$$
where $\Delta$ denotes the discriminant of $q$ (that is, $(y_1-y_2)^2$), which equals that of $p$.
Then:
\begin{itemize}
\item Either $\car(\F) \neq 2$, in which case
$(t-d)^2-\Delta=p\Bigl(t-d+\frac{\tr p}{2}\Bigr)$ is irreducible over $\F$.
\item Or $\car(\F)=2$ and $(t-d)^2-\Delta=(t-d-\tr q)^2=(t-d-\tr p)^2$.
\end{itemize}

Finally, let $u$ be an endomorphism of a finite-dimensional vector space $V$ and assume that $u$ is d-exceptional with respect to $(p,q)$.
Then, $v:=u^2-\delta u$ is annihilated by some power of $\Lambda_{p,q}$, which is split over $\F$;
hence, $v$ is triangularizable. In the next section, we study the case when $v$ has a sole eigenvalue.

\subsubsection{When $v$ has a sole eigenvalue}

\begin{lemma}\label{CNsamesplitbasicLemma}
Let $p$ and $q$ be split monic polynomials with degree $2$ over $\F$, and assume that both $p$ and $q$ are irreducible, with the same
splitting field. Write $p(t)=(t-x_1)(t-x_2)$ and $q(t)=(t-y_1)(t-y_2)$. Let $a,b$ be endomorphisms of a finite-dimensional vector
space $V$, and assume that $p(a)=q(b)=0$. Denote respectively by $a^\star$ and $b^\star$ the $p$-conjugate of $a$ and the $q$-conjugate of $b$.
Set $w:=ab^\star+ba^\star$ and $z:=x_1y_2+x_2y_1$. Let $k \in \N^*$ be such that
$n_k(w,z)=n_{k+1}(w,z)$. Then, $n_{k+1}(w,z)$ is a multiple of $4$.
\end{lemma}

\begin{proof}
Our first step consists in reducing the situation to the one where $p=q$.
We note that the situation is unchanged by performing two basic transformations:
\begin{itemize}
\item Let $d \in \F$, set $b':=b-d\,\id_V$, and note that $q(t+d)=(t+d-y_1)(t+d-y_2)$ annihilates $b'$.
With $(b')^\star:=\tr(q(t+d))\,\id_V-b'=(\tr q-2d)\,\id_V-b'=b^\star-d\, \id_V$, we set
$w':=a(b')^\star+b' a^\star=w-d a-da^\star=w-(d \tr p)\,\id_V$, and
with $z':=x_1(y_2-d)+x_2(y_1-d)$, we have $z'=z-d(x_1+x_2)=z-d \tr p$, whence $n_k(w,z)=n_k(w',z')$.

\item Let $d \in \F \setminus \{0\}$, set $b':=d b$, and note that $d^2 q(d^{-1}t)=(t-dy_1)(t-dy_2)$ annihilates $b'$.
With $(b')^\star:=\tr(d^2 q(d^{-1}t))\id_V-b'=(d \tr q)\id_V-d b=d b^\star$, we set
$w':=a(b')^\star+b' a^\star=dw$, and with $z':=x_1(dy_2)+x_2(dy_1)$, we have $z'=dz$, whence
$n_k(w,z)=n_k(w',z')$.
\end{itemize}
Now, denoting by $\L$ the splitting field of $p$ in $\overline{\F}$, the pair $(1,x_1)$ is a basis of the $\F$-vector space $\L$,
and hence there is a pair $(\gamma,\eta) \in \F^2$ such that $y_1=\gamma +\eta x_1$.
Note that $\eta \neq 0$ since $y_1 \not\in \F$. If $\L$ is inseparable over $\F$, then we also have $y_2=\gamma+\eta x_2$;
otherwise, we denote by $\sigma$ the non-identity element in the Galois group of $\L$ over $\F$, and we find
$y_2=\sigma(y_1)=\gamma+\eta\, \sigma(x_1)=\gamma+\eta\, x_2$. Thus, by combining the above two basic transformations,
we see that the pair $(y_1,y_2)$ can be replaced with $(x_1,x_2)$ without losing any generality whatsoever.

Hence, in the remainder of the proof, we shall consider only the case when $(x_1,x_2)=(y_1,y_2)$.
Note in particular that $p=q$ and that $z=2x_1x_2=2p(0)$.

We assume that $n_k(w,z)=n_{k+1}(w,z)$, and we shall prove that $s:=n_{k+1}(w,z)$ is a multiple of $4$.
Set $$E:=\Ker\bigl((w-z\,\id_V)^{k+1}\bigr)/\Ker\bigl((w-z\,\id_V)^{k-1}\bigr)$$
and
$$F:=\Ker\bigl((w-z\,\id_V)^{k}\bigr)/\Ker\bigl((w-z\,\id_V)^{k-1}\bigr).$$
Since $n_k(w,z)=n_{k+1}(w,z)$, we see that $E$ has dimension $2 s$ and
$F$ has dimension $s$.

By the Basic Commutation Lemma, we know that $a$ and $b$ commute with $w$, and hence they induce
endomorphisms of $E$ which we respectively denote by $a'$ and $b'$.
We still write $(a')^\star:=(\tr p)\,\id_E-a'$ and $(b')^\star:=(\tr p)\,\id_E-b'$, and we note that
$$w'=a'(b')^\star+b'(a')^\star \quad \text{and} \quad F=\Ker(w'-z\, \id_E).$$
Moreover, since $\dim E=2 \dim F$ and $(w'-z \,\id_E)^2=0$, we find that $F=\im(w'-z\, \id_E)$.

Note that $p(a')=p(b')=0$.
Since $p$ is irreducible and $a'$ stabilizes $F$, we can choose a complementary subspace $G$ of $F$ in $E$
that is stable under $a'$. Moreover, $w'-z\,\id_E$ induces an isomorphism from $G$ to $\im(w'-z\,\id_E)=F$,
and hence, for some basis $\bfB$ of $E$ that is adapted to the decomposition $E=F \oplus G$, we have
$$\Mat_\bfB(w')=\begin{bmatrix}
z I_s & I_s \\
0_s & z I_s
\end{bmatrix}.$$
Since $a'$ stabilizes $F$ and $G$, and since $a'$ commutes with $w'$,
we have, for some matrix $A \in \Mat_s(\F)$,
$$\Mat_\bfB(a')=\begin{bmatrix}
A & 0_s \\
0_s & A
\end{bmatrix} \quad \text{and} \quad p(A)=0.$$
Likewise, for some $B$ and $C$ in $\Mat_s(\F)$, we have
$$\Mat_\bfB(b')=\begin{bmatrix}
B & C \\
0_s & B
\end{bmatrix} \quad \text{with} \quad p(B)=0.$$
Denoting by $A^\star$ the $p$-conjugate of $A$ and by $B^\star$ the one of $B$, we
deduce that
$$\Mat_\bfB\bigl((a')^\star\bigr)=\begin{bmatrix}
A^\star & 0_s \\
0_s & A^\star
\end{bmatrix} \quad \text{and} \quad
\Mat_\bfB\bigl((b')^\star\bigr)=\begin{bmatrix}
B^\star & -C \\
0_s & B^\star
\end{bmatrix}.$$
As $b'(b')^\star=p(0)\,\id_E$, we obtain $BC=CB^\star$.
Now, set
$$N:=A-B,$$
and note that $A=B+N$ and $A^\star=(\tr p)I_s-A=B^\star-N$.
We shall prove the following three identities:
\begin{enumerate}[(1)]
\item $N^2=0$;
\item $NB^\star=B N$;
\item $I_s=-NC-CN$.
\end{enumerate}
Using $w'=a'(b')^\star+b' (a')^\star$, we obtain the following two identities from the upper-left and upper-right blocks:
$$AB^\star+B A^\star=z I_s \quad \text{and} \quad I_s=-AC+CA^\star.$$
The first equality yields
$$(B+N)B^\star+B(B^\star-N)=2p(0) I_s,$$
and since $BB^\star=p(0)I_s$, we obtain identity (2).
The second equality yields $I_s=-(B+N)C+C(B^\star-N)=-NC-CN$ since $CB^\star=BC$, and hence identity (3) is proved.
Finally,
$$N^2=(A-B)(B^\star-A^\star)=AB^\star+BA^\star-BB^\star-AA^\star=z I_s-2p(0)I_s=0.$$

Now, we can conclude. Identity (3) yields $2 \rk N \geq s$.
On the other hand, identity (1) yields $2\rk N \leq s$, and hence $s=2 \rk N$.
Finally, identity (2) shows that the endomorphism
$X \mapsto NX$ of the $\F$-vector space $\F^s$ is semi-linear with respect to the structure of $\F[B]$-vector space
induced by $B$ (note that the $\F$-algebra $\F[B]$ is isomorphic to the splitting field of $p$ since $p$ is irreducible and $p(B)=0$).
It follows that the rank of $N$ is even. Therefore, $s=2 \rk N$ is a multiple of $4$, which completes the proof.
\end{proof}

Now, we come back to the assumptions and notation from the last paragraph of the preceding section and we assume furthermore that $v$
has a sole eigenvalue. By renaming the roots of $p$ and $q$ in $\L$, we can assume that this eigenvalue is
$$z:=-N_{\L/\F}(x_1-y_1)=-(x_1-y_1)(x_2-y_2).$$

Remembering that
$$v=ab^\star+ba^\star-\bigl(p(0)+q(0)\bigr)\,\id_V$$
and that
$$z=-x_1x_2-y_1y_2+(x_1y_2+x_2y_1)=-p(0)-q(0)+(x_1y_2+x_2y_1),$$
the preceding lemma has the following corollary:

\begin{cor}\label{CNsamesplit}
Let $k \in \N^*$ be such that $n_k(v,z)=n_{k+1}(v,z)$.
Then, $n_{k+1}(v,z)$ is a multiple of $4$.
\end{cor}

Now, we take on a converse statement:

\begin{lemma}\label{samesplitconstructivelemma}
Let $p$ and $q$ be monic irreducible polynomials with degree $2$ over $\F$.
Assume that $p$ and $q$ have the same splitting field $\L$. Let $x$ and $y$ be respective roots of $p$ and $q$ in $\L$ such that
$x-y \not\in \F$. Set $z:=-N_{\L/\F}(x-y)$.
Then, for all $n \in \N$, the matrix $C\bigl((t^2-\delta t-z)^{n+1}\bigr) \oplus C\bigl((t^2-\delta t-z)^n\bigr)$
is a $(p,q)$-difference (and a d-exceptional one).
\end{lemma}

\begin{proof}
Set $r:=t^2-\delta t-z$ and note that $r=(t-(x-y))(t-(\delta-(x-y)))$, whence $r$ is irreducible over $\F$
and its splitting field is $\L$.
We can therefore assume that $\L$ is the subalgebra of $\Mat_2(\F)$ generated by $C(r)$.

Let $n \in \N$. In order to prove that $C(r^{n+1}) \oplus C(r^n)$ is a $(p,q)$-difference,
we distinguish between two cases.

\vskip 2mm
\noindent \textbf{Case 1. $\L$ is separable over $\F$.} \\
By Theorem \ref{theoPandQsplitSingleroots}, the direct sum
$$\bigl(C(t^{n+1})+(x-y) I_{n+1}\bigr)\oplus \bigl(C(t^{n})+(\delta-x+y) I_{n}\bigr)$$
in $\Mat_{2n+1}(\L)$ is a $(p,q)$-difference. Viewing every entry of this matrix as a $2$-by-$2$ matrix with entries in $\F$,
we gather that, for some elements $R_1$ and $R_2$ of $\Mat_2(\F)$ that are annihilated by $r$,
the matrix
$$S:=\underbrace{
\begin{bmatrix}
R_1 & 0 & \cdots & \cdots & (0) \\
I_2 & R_1 & \ddots & & \vdots \\
0 & \ddots & \ddots & \ddots & \vdots \\
\vdots & & \ddots & R_1 & 0 \\
(0) & \cdots & 0 & I_2 & R_1
\end{bmatrix}}_{\in \Mat_{2n+2}(\F)} \oplus
\underbrace{
\begin{bmatrix}
R_2 & 0 & \cdots & \cdots & (0) \\
I_2 & R_2 & \ddots & & \vdots \\
0 & \ddots & \ddots & \ddots & \vdots \\
\vdots & & \ddots & R_2 & 0 \\
(0) & \cdots & 0 & I_2 & R_2
\end{bmatrix}}_{\in \Mat_{2n}(\F)}$$
is a $(p,q)$-difference.

The polynomial $r$ is irreducible over $\F$ and splits over $\L$, whence $r$ is separable over $\F$. It then follows from
Proposition \ref{blockcyclicprop} (in the appendix) that
$S \simeq C(r^{n+1}) \oplus C(r^n)$, which proves the claimed result.

\vskip 2mm
\noindent \textbf{Case 2. $\L$ is inseparable over $\F$.} \\
Then, $\tr(p)=\tr(q)=0$ and $\delta=0$.
Over $\L$, the polynomials $p$ and $q$ are split with sole respective roots $x$ and $y$. By Theorem \ref{theoPandQsplitSoleRoot}, the matrix
$C(t^{2n+1}) + (x-y) I_{2n+1}$ of $\Mat_{2n+1}(\L)$ is a $(p,q)$-difference.
Viewing $x-y$ as a matrix $R\in \Mat_2(\F)$ annihilated by $r$, we gather that the matrix
$$T:=\begin{bmatrix}
R & 0 & \cdots & \cdots & (0) \\
I_2 & R & \ddots & & \vdots \\
0 & \ddots & \ddots & \ddots & \vdots \\
\vdots & & \ddots & R & 0 \\
(0) & \cdots & 0 & I_2 & R
\end{bmatrix}$$
of $\Mat_{4n+2}(\F)$ is a $(p,q)$-difference.
As the root of $r$ has multiplicity $2$, we deduce from Proposition \ref{blockcyclicprop} that
$T \simeq C(r^{n+1}) \oplus C(r^n)$, which proves that the latter matrix is a $(p,q)$-difference.
\end{proof}

\subsubsection{On the case when $p=q$}

\begin{lemma}\label{stablep=qlemma}
Let $p$ be a monic polynomial with degree $2$, and let $a$ and $b$ be
endomorphisms of a vector space $V$ such that $p(a)=p(b)=0$. Then, $\Ker(a-b)$ is stable under $a$ and $b$.
\end{lemma}

\begin{proof}
Let us write $p(t)=t^2-\lambda t+\alpha$. Let $x \in \Ker(a-b)$ and set $y:=a(x)=b(x)$.
Then, $a(y)=a^2(x)=(\tr p)\, a(x)-p(0)\, x$ and $b(y)=b^2(x)=(\tr p)\, b(x)-p(0)\, x$ and hence $a(y)=b(y)$, i.e.\ $a(x) \in \Ker (a-b)$.
\end{proof}

\begin{prop}\label{stablep=qeven}
Let $p$ be an irreducible monic polynomial with degree $2$, and let $a$ and $b$ be
endomorphisms of a finite-dimensional vector space $V$ such that $p(a)=p(b)=0$.
Then, for every integer $k \geq 1$, the endomorphism $a-b$ has an even number of Jordan cells of size $k$
for the eigenvalue $0$.
\end{prop}

\begin{proof}
Classically, this amounts to proving that $\Ker((a-b)^n)$ is even-dimensional for all $n \in \N$.
Yet, by the Commutation Lemma, we know that both $a$ and $b$ commute with $(a-b)^2$.
Hence, for all $n \in \N$, the subspace $\Ker\bigl((a-b)^{2n}\bigr)$ is stable under both $a$ and $b$.
Fixing $n \in \N$, the endomorphism of $\Ker\bigl((a-b)^{2n}\bigr)$
induced by $a$ is annihilated by $p$, which is irreducible with degree $2$, and hence $\dim \Ker\bigl((a-b)^{2n}\bigr)$
is a multiple of $2$. Next, $a$ and $b$ induce endomorphisms $a'$ and $b'$ of the quotient space
$\Ker \bigl((a-b)^{2n+2}\bigr)/\Ker \bigl((a-b)^{2n}\bigr)$, and the kernel of $a'-b'$ is the quotient subspace
$\Ker \bigl((a-b)^{2n+1}\bigr)/\Ker \bigl((a-b)^{2n}\bigr)$.
Hence, by Lemma \ref{stablep=qlemma}, this subspace is stable under $a'$. Again, since $p(a')=0$ we deduce that the dimension of
$\Ker \bigl((a-b)^{2n+1}\bigr)/\Ker \bigl((a-b)^{2n}\bigr)$ is even, and hence the one of $\Ker \bigl((a-b)^{2n+1}\bigr)$ is also even.
\end{proof}

Now, we prove a converse statement:

\begin{lemma}\label{CSp=q}
Let $p$ be an irreducible monic polynomial with degree $2$ over $\F$, and $n$ be a positive integer.
Then, the matrix $C(t^n) \oplus C(t^n)$ is a $(p,p)$-difference.
\end{lemma}

Note that if $n$ is even this result is a special case of the Duplication Lemma.

\begin{proof}
We use the same strategy as in the proof of Lemma \ref{samesplitconstructivelemma}.
Since $C(p)$ is annihilated by $p$, which is irreducible with degree $2$, the $\F$-algebra $\L:=\F\bigl[C(p)\bigr]$
is a splitting field of $p$.
Hence, by Theorem \ref{theoPandQsplitSoleRoot} if $p$ is inseparable, and by Theorem \ref{theoPandQsplitSingleroots} otherwise,
we know that the matrix
$$
\begin{bmatrix}
0_\L & 0_\L & \cdots & \cdots & (0) \\
1_\L & 0_\L & \ddots & & \vdots \\
0_\L & \ddots & \ddots & \ddots & \vdots \\
\vdots & & \ddots & 0_\L & 0_\L \\
(0) & \cdots & 0_\L & 1_\L & 0_\L
\end{bmatrix}$$
of $\Mat_n(\L)$ is a $(p,p)$-difference.
Hence, the matrix
$$M_n:=\begin{bmatrix}
0_2 & 0_2 & \cdots & \cdots & (0) \\
I_2 & 0_2 & \ddots & & \vdots \\
0_2 & \ddots & \ddots & \ddots & \vdots \\
\vdots & & \ddots & 0_2 & 0_2 \\
(0) & \cdots & 0_2 & I_2 & 0_2
\end{bmatrix}$$
of $\Mat_{2n}(\F)$ is a $(p,p)$-difference. By permuting the basis vectors, one sees that $M_n$ is similar to $C(t^n) \oplus C(t^n)$, and hence
this last matrix is a $(p,p)$-difference.
\end{proof}

\subsubsection{Conclusion}

We are now ready to conclude the study of the case when $p$ and $q$ are irreducible with the same splitting field.
There are five different subcases to consider:
\begin{itemize}
\item $q$ is not a translation of $p$, and $p$ is separable;
\item $q$ is not a translation of $p$, and $p$ is inseparable;
\item $q$ is a translation of $p$ and $\car(\F) \neq 2$;
\item $q$ is a translation of $p$, $p$ is separable and $\car(\F)=2$;
\item $q$ is a translation of $p$, and $p$ is inseparable.
\end{itemize}

\begin{theo}\label{theoSamesplitNottranslationSeparable}
Let $p$ and $q$ be irreducible monic polynomials with degree $2$ over $\F$, and assume that they have the same splitting field $\L$.
Assume further that $q$ is not a translation of $p$ (over $\F$) and that $p$ is separable, and set $\delta:=\tr p-\tr q$.
Denote by $x_1,x_2$ (respectively, by $y_1,y_2$) the roots of $p$ (respectively, of $q$) in $\L$.
Let $u$ be an endomorphism of a finite-dimensional vector space $V$ over $\F$. Then, the following conditions are equivalent:
\begin{enumerate}[(i)]
\item $u$ is a d-exceptional $(p,q)$-difference.
\item The minimal polynomial of $u$ is the product of a power of $t^2-\delta t+N_{\L/\F}(x_1-y_1)$
with a power of $t^2-\delta t+N_{\L/\F}(x_1-y_2)$. Moreover,
the invariant factors of $u$ read $r_1,r_2,\dots,r_{2k-1},r_{2k},\dots$ where, for every positive integer $k$,
there are integers $a$ and $b$ in $\{0,1\}$ such that
$$r_{2k-1}=r_{2k} \times \bigl(t^2-\delta t+N_{\L/\F}(x_1-y_1)\bigr)^a \bigl(t^2-\delta t+N_{\L/\F}(x_1-y_2)\bigr)^b.$$
\item In some basis of $V$, the endomorphism $u$ is represented by a block-diagonal matrix in which every diagonal block
equals
$$C\bigl((t^2-\delta t+N_{\L/\F}(x-y))^{k+\epsilon}\bigr)\oplus C\bigl((t^2-\delta t+N_{\L/\F}(x-y))^k\bigr)$$
for some root $x$ of $p$, some root $y$ of $q$, and some pair $(k,\epsilon)\in \N \times \{0,1\}$.
\end{enumerate}
\end{theo}

\begin{proof}
Since $p$ and $q$ are separable, we have $x_1-y_1 \neq x_1-y_2$ and $x_1-y_1 \neq x_2-y_1$.
Hence, with the study from the start of Section \ref{d-exceptionalsectionII}, we know that $r_1:=t^2-\delta t+N_{\L/\F}(x_1-y_1)$ and $r_2:=t^2-\delta t+N_{\L/\F}(x_1-y_2)$
are the distinct irreducible factors of $F_{p,q}$, and hence $u$ is d-exceptional with respect to $(p,q)$
if and only if the irreducible monic factors of its minimal polynomial are among $r_1$ and $r_2$.

Next, it is obvious that condition (ii) implies condition (iii).
Moreover, by Lemma \ref{samesplitconstructivelemma} and the Duplication Lemma, we know that condition (iii) implies condition (i).

In order to conclude, we assume that condition (i) holds and we prove that condition (ii) holds.
For $j \in \{1,2\}$ and $k \in \N^*$, denote by $n_k^{(j)}$ the number of diagonal blocks
equal to $C\bigl((t^2-\delta t+N_{\L/\F}(x_1-y_j))^l\bigr)$ for some $l \geq k$ in the primary canonical form of $u$.
Set $v:=u^2-\delta u$ and choose endomorphisms $a$ and $b$ of $V$ such that $u=a-b$ and $p(a)=q(b)=0$.
Since $N_{\L/\F}(x_1-y_1) \neq N_{\L/\F}(x_1-y_2)$, we have
$V=V_1 \oplus V_2$ where $V_j$ denotes the characteristic subspace of $v$ for the eigenvalue $-N_{\L/\F}(x_1-y_j)$.
As $a$ and $b$ commute with $v$, we find that they stabilize $V_1$ and $V_2$.
Fixing $j \in \{1,2\}$ and denoting by $u_j$ (respectively, by $v_j$) the endomorphism of $V_j$ induced by $u$ (respectively, by $v$), we get that
the minimal polynomial of $u_j$ is a power of $r_j$ and that $u_j$ is a $(p,q)$-difference.
Moreover, for all $k \in \N^*$ and all $j \in \{1,2\}$, we have
$$2 n_k^{(j)}=n_k\bigl(v_j,-N_{\L/\F}(x_1-y_j)\bigr);$$
On the other hand, exactly $n_k^{(j)}$ elementary invariants of $u_j$ equal $r_j^l$ for some $l \geq k$.
Hence, by Corollary \ref{CNsamesplit}, we find that if $n_{k+1}^{(j)}$ is odd then $n_{k}^{(j)}>n_{k+1}^{(j)}$.

Finally, let us consider the invariant factors $s_1,\dots,s_k,\dots,$ of $u$.
Since $u$ is d-exceptional with respect to $(p,q)$ we know that each $s_k$ is the product of a power of $r_1$ with a power of $r_2$.
Let $k \in \N^*$ and write $s_{2k-1}=r_1^{N_1} r_2^{N_2}$ and $s_{2k}=r_1^{N'_1} r_2^{N'_2}$,
so that $N_1 \geq N'_1$ and $N_2 \geq N'_2$.
If $N_1 > N'_1+1$ then $n_{N'_1+1}^{(1)}=n_{N'_1+2}^{(1)}=2k-1$ and we have a contradiction with the above result.
Hence, $N_1 \in \{N'_1+1,N'_1\}$. Likewise, we obtain $N_2 \in \{N'_2+1,N'_2\}$.
This yields condition (ii), which completes the proof.
\end{proof}

\begin{theo}\label{theoSamesplitNottranslationInseparable}
Let $p$ and $q$ be irreducible monic polynomials with degree $2$ over $\F$, and assume that they have the same splitting field $\L$.
Assume further that $q$ is not a translation of $p$ (over $\F$) and that $p$ is inseparable.
Write $p=t^2+\alpha$ and $q=t^2+\beta$ for some $(\alpha,\beta)\in \F^2$.
Let $u$ be an endomorphism of a finite-dimensional vector space over $\F$. Then, the following conditions are equivalent:
\begin{enumerate}[(i)]
\item $u$ is a d-exceptional $(p,q)$-difference.
\item The minimal polynomial of $u$ is a power of $t^2+\alpha+\beta$. Moreover,
the invariant factors of $u$ read $r_1,r_2,\dots,r_{2k-1},r_{2k},\dots$ where, for every positive integer $k$,
we have $r_{2k-1}=r_{2k}$ or $r_{2k-1}=(t^2+\alpha+\beta)r_{2k}$.
\item In some basis of $V$, the endomorphism $u$ is represented by a block-diagonal matrix in which every diagonal block
equals
$$C\bigl((t^2+\alpha+\beta)^{n+\epsilon}\bigr) \oplus C\bigl((t^2+\alpha+\beta)^n\bigr)$$
for some $n \in \N$ and some $\epsilon \in \{0,1\}$.
\end{enumerate}
\end{theo}

\begin{proof}
Here, $F_{p,q}=(t^2+\alpha+\beta)^2$, $\Lambda_{p,q}=(t+\alpha+\beta)^2$ and $t^2+\alpha+\beta$ is irreducible over $\F$ because $q$ is not a translation of $p$.
From there, the proof is essentially similar to the one of Theorem \ref{theoSamesplitNottranslationSeparable}, the only difference being that, in the proof
that (i) implies (ii), the endomorphism $v:=u^2$ has $\alpha+\beta$ as its sole eigenvalue.
\end{proof}

\begin{theo}
Let $p$ and $q$ be irreducible monic polynomials with degree $2$ over $\F$.
Assume that $q=p(t+d)$ for some $d \in \F$, and that $\car(\F) \neq 2$.
Set $r:=p\bigl(t-d+(\tr p)/2\bigr)$.
Let $u$ be an endomorphism of a finite-dimensional vector space $V$ over $\F$. Then, the following conditions are equivalent:
\begin{enumerate}[(i)]
\item $u$ is a d-exceptional $(p,q)$-difference.
\item The minimal polynomial of $u$ is the product of a power of $t-d$ with a power of $r$.
Moreover,
the invariant factors of $u$ read $s_1,s_2,\dots,s_{2k-1},s_{2k},\dots$ where, for every positive integer $k$,
we have $s_{2k-1}=s_{2k}$ or $s_{2k-1}=r\,s_{2k}$.
\item In some basis of $V$, the endomorphism $u$ is represented by a block-diagonal matrix in which every diagonal block
equals $C(r^{n+\epsilon}) \oplus C(r^n)$,
for some $n \in \N$ and some $\epsilon \in \{0,1\}$, or
$C\bigl((t-d)^n\bigr) \oplus C\bigl((t-d)^n\bigr)$ for some $n \in \N^*$.
\end{enumerate}
\end{theo}

\begin{proof}
As we have seen in the beginning of Section \ref{d-exceptionalsectionII},
$F_{p,q}(t)=(t-d)^2 r$ and $r$ is irreducible. Moreover, $r=t^2-\delta t+N_{\L/\F}(x_1-x_2+d)$, where $\L$ denotes the splitting field of
$p$, and $p(t)=(t-x_1)(t-x_2)$.

Moreover, as we can safely replace $u$ with $u-d\,\id_V$, we lose no generality in assuming that $p=q$, in which case $d=0$.
From there, the implication (ii) $\Rightarrow$ (iii) is obvious, and (iii) $\Rightarrow$ (i)
is readily obtained by applying the Duplication Lemma together with Lemmas \ref{samesplitconstructivelemma} and \ref{CSp=q}.

Assume finally that (i) holds. For $k \in \N^*$, denote by $n_k$
the number of blocks of type $C\bigl((t-d)^k\bigr)$ in the primary canonical form of $u$, and by $m_k$
the number of blocks of type $C(r^l)$, for some $l \geq k$, in the primary canonical form of $u$.

Then, by Corollary \ref{CNsamesplit}, we find that $n_k$ is even for all $k \in \N^*$.

Noting that $\Lambda_{p,q}=t\bigl(t+N_{\L/\K}(x_1-x_2)\bigr)$, we
use the same line of reasoning as in the proof of Theorem \ref{theoSamesplitNottranslationSeparable} to gather that
$m_{k}$ is odd whenever $m_{k}>m_{k+1}$. The derivation of (ii) is then done as in the proof of
Theorem \ref{theoSamesplitNottranslationSeparable}.
\end{proof}

\begin{theo}\label{theoSamesplitTranslationSeparable}
Let $p$ and $q$ be irreducible monic polynomials with degree $2$ over $\F$.
Assume further that $q=p(t+d)$ for some $d \in \F$, that $\car(\F)=2$ and that $p$ is separable.
Let $u$ be an endomorphism of a finite-dimensional vector space $V$ over $\F$. Then, the following conditions are equivalent:
\begin{enumerate}[(i)]
\item $u$ is a d-exceptional $(p,q)$-difference.
\item The minimal polynomial of $u$ is the product of a power of $t-d$ with a power of $t-d-\tr(p)$.
Moreover, the invariant factors of $u$ read $s_1,s_2,\dots,s_{2k-1},s_{2k},\dots$ where, for every positive integer $k$,
we have $s_{2k-1}=s_{2k}$.
\item In some basis of $V$, the endomorphism $u$ is represented by a block-diagonal matrix in which every diagonal block
equals $C\bigl((t-d)^n\bigr) \oplus C\bigl((t-d)^n\bigr)$ or $C\bigl((t-d-\tr(p))^n\bigr) \oplus
C\bigl((t-d-\tr(p))^n\bigr)$ for some $n \in \N^*$.
\end{enumerate}
\end{theo}

\begin{proof}
Here, we see that $F_{p,q}=(t-d)^2\bigl(t-d-\tr p\bigr)^2$ and $\Lambda_{p,q}=(t+d^2)\bigl(t+d^2+(\tr p)^2\bigr)$,
with $d^2 \neq d^2+(\tr p)^2$. From there, the proof is similar to the one of the previous three theorems.
\end{proof}

\begin{theo}
Let $p$ and $q$ be two irreducible monic polynomials with degree $2$ over $\F$.
Assume further that $q=p(t+d)$ for some $d \in \F$, that $\car(\F)=2$ and that $p$ is inseparable.
Let $u$ be an endomorphism of a finite-dimensional vector space $V$ over $\F$. Then, the following conditions are equivalent:
\begin{enumerate}[(i)]
\item $u$ is a d-exceptional $(p,q)$-difference.
\item The minimal polynomial of $u$ is a power of $t-d$.
Moreover, the invariant factors of $u$ read $s_1,s_2,\dots,s_{2k-1},s_{2k},\dots$ where, for every positive integer $k$,
we have $s_{2k-1}=s_{2k}$.
\item In some basis of $V$, the endomorphism $u$ is represented by a block-diagonal matrix in which every diagonal block
equals $C\bigl((t-d)^n\bigr) \oplus C\bigl((t-d)^n\bigr)$ for some $n \in \N^*$.
\end{enumerate}
\end{theo}

\begin{proof}
Here, we see that $F_{p,q}=(t-d)^4$ and $\Lambda_{p,q}=(t+d^2)^2$.
From there, the proof is similar to the one of the previous four theorems.
\end{proof}

Using the above five theorems, it is easy to derive the classification of
d-exceptional $(p,q)$-differences when $p$ and $q$ are both irreducible and have the same splitting field,
as given in Table \ref{dfigure7}.

\subsection{Exceptional $(p,q)$-differences (IV): When $p$ and $q$ are irreducible with distinct splitting fields}\label{d-exceptionalsectionIII}

In this final section, we tackle the case when both polynomials $p$ and $q$ are irreducible but they have distinct
splitting fields in $\overline{\F}$.
The discussion will basically be split into three subcases, whether both $p$ and $q$ are separable,
both are inseparable, or exactly one of them is separable.
In the first two cases, a basic trick will be to extend the field of scalars by using $\Lambda_{p,q}$ and $v:=u^2-\delta u$,
which will allow us to use the results from Section \ref{d-exceptionalsectionII}.

\subsubsection{Case 1. Both $p$ and $q$ are separable}\label{bothPandQseparableDistinctsplittingfields}

In that case, it is known that the splitting field $\L$ of $pq$ in $\overline{\F}$ is a Galois extension of $\F$
with degree $4$. Moreover, the Galois group of $\L$ over $\F$ contains two elements $\sigma$ and $\tau$
such that $\sigma$ exchanges the two roots of $q$ and fixes the ones of $p$,
and $\tau$ exchanges the two roots of $p$ and fixes the ones of $q$. It follows that $\Gal(\L/\F)$
acts transitively on $\Root(p)-\Root(q)$, whence $F_{p,q}$ is a power of some irreducible monic polynomial of
$\F[t]$.

Let us split $p=(t-x_1)(t-x_2)$ and $q=(t-y_1)(t-y_2)$ in $\L[t]$.
If $x_1-y_1,x_1-y_2,x_2-y_1,x_2-y_2$ are pairwise distinct, then $F_{p,q}$
is irreducible over $\F$ and separable.

Since $y_1 \neq y_2$ and $x_1 \neq x_2$, it follows that for two elements among $x_1-y_1,x_1-y_2,x_2-y_1,x_2-y_2$ to be equal,
it is necessary and sufficient that $x_1-y_1=x_2-y_2$ or $x_1-y_2=x_2-y_1$, that is
$(x_1-x_2)^2=(y_1-y_2)^2$, i.e.\ $p$ and $q$ have the same discriminant.
Yet, we have assumed that the splitting fields of $p$ and $q$ are distinct. Hence,
$p$ and $q$ have the same discriminant only if $\car(\F)=2$.
In that case, the respective discriminants of $p$ and $q$ equal $\tr(p)^2$ and $\tr(q)^2$, and hence they are equal if and only if
$\tr(p)=\tr(q)$.

Moreover, if $\car(\F)=2$ and $\tr(p)=\tr(q)$ then as $x_1-y_1 \neq x_1-y_2$ we get that
the sole irreducible monic divisor of $F_{p,q}$ over $\F$ is $p(t+y_1)=\bigl(t-(x_1-y_1)\bigr)\bigl(t-(x_2-y_1)\bigr)=t^2-\tr(p)t+p(0)+q(0)$,
which equals $p(t+y_2)$ by the way.

In any case, the polynomial $\bigl(t-(x_1-y_1)\bigr)\bigl(t-(x_2-y_2)\bigr)=t^2-\delta t+(x_1-y_1)(x_2-y_2)$ does not belong to $\F[t]$
(because $x_1-y_2$ is not one of its roots), which yields that $\Lambda_{p,q}$ is irreducible over $\F$.

\paragraph{Case 1.1. $p$ and $q$ have distinct discriminants}

\begin{theo}\label{theoFpq4roots}
Let $p$ and $q$ be monic polynomials with degree $2$ over $\F$,
and assume that they are both irreducible with distinct splitting fields
and distinct discriminants.
Let $u$ be an endomorphism of a finite-dimensional vector space $V$ over $\F$. Then, the following conditions are equivalent:
\begin{enumerate}[(i)]
\item The endomorphism $u$ is a d-exceptional $(p,q)$-difference.
\item The minimal polynomial of $u$ is a power of $F_{p,q}$ and if we denote by $r_1,\dots,r_k,\dots$
the invariant factors of $u$, then $r_{2k-1}=r_{2k}$ or $r_{2k-1}=r_{2k}\, F_{p,q}$, for all $k \in \N^*$.
\item In some basis of $V$, the endomorphism $u$ is represented by a block-diagonal matrix in which
each diagonal block equals $C\bigl(F_{p,q}^n\bigr) \oplus C\bigl(F_{p,q}^n\bigr)$
or $C\bigl(F_{p,q}^n\bigr) \oplus C\bigl(F_{p,q}^{n-1}\bigr)$ for some positive integer $n$.
\end{enumerate}
\end{theo}

\begin{proof}
It is easily seen that condition (ii) implies condition (iii).
Conversely, assume that (iii) holds but that (ii) does not.
First of all, it is obvious from condition (iii) that all the invariant factors of $u$ are powers of $F_{p,q}$.
Next, there is a least positive integer $k$ such that $F_{p,q}^2 r_{2k}$ divides $r_{2k-1}$.
Write then $r_{2k-1}=F_{p,q}^{\ell+1}$ for some integer $\ell \geq 0$. Then,
in any decomposition given by condition (iii), there is no diagonal block of the form
$C\bigl(F_{p,q}^m\bigr) \oplus C\bigl(F_{p,q}^{n}\bigr)$ in which one of $m$ and $n$ equals $\ell$.
Hence, each one of those diagonal blocks equals either $C\bigl(F_{p,q}^n\bigr) \oplus C\bigl(F_{p,q}^n\bigr)$
for some integer $n \neq \ell$, or $C\bigl(F_{p,q}^n\bigr) \oplus C\bigl(F_{p,q}^{n-1}\bigr)$
for some positive integer $n$ distinct from $\ell$ and $\ell+1$; in the latter case either both $n$ and $n-1$ are greater than $\ell$, or both
are less than $\ell$. It follows that there is an even number of invariant factors of $u$ that equal $F_{p,q}^n$ for some $n>\ell$, contradicting
the fact that there are $2k-1$ such invariants factors. Hence, condition (iii) implies condition (ii).

It only remains to prove that conditions (i) and (iii) are equivalent.
Some general work is required before we tackle this equivalence.

As we have seen in the beginning of Section \ref{bothPandQseparableDistinctsplittingfields}, our assumptions imply that both
$F_{p,q}$ and $\Lambda_{p,q}$ are irreducible over $\F$.
Moreover, the Galois group of $\L$ over $\F$ is isomorphic to $(\Z/2)^2$, which has three proper subgroups.

Denote by $\K$ the splitting field of $\Lambda_{p,q}$ in $\L$. Without loss of generality, we can consider that
$\K$ is the subalgebra of $\Mat_2(\F)$ generated by the companion matrix of $\Lambda_{p,q}$.
Obviously $\K$ is the Galois subfield of $\L$ associated with the subgroup of $\Gal(\L/\F)$
generated by the Galois automorphism that exchanges the two roots of $p$ and that exchanges the two roots of $q$.
Hence, $p$ and $q$ remain irreducible over $\K$.
Moreover, the assumptions show that $p$ and $q$ do not have the same discriminant.
In particular, over $\K$ the polynomials $p$ and $q$ satisfy the assumptions of Theorem \ref{theoSamesplitNottranslationSeparable}.

Splitting $p(t)=(t-x_1)(t-x_2)$ and $q(t)=(t-y_1)(t-y_2)$ over $\L$, we set
$$r:=\bigl(t-(x_1-y_1)\bigr)\bigl(t-(x_2-y_2)\bigr)\in \K[t] \quad \text{and} \quad s:=\bigl(t-(x_1-y_2)\bigr)\bigl(t-(x_2-y_1)\bigr)\in \K[t].$$
Let $n \in \N^*$.
Throughout the proof, $C(r^n)$ will be interpreted both as a matrix of $\Mat_{2n}(\K)$ and as a matrix of $\Mat_{4n}(\F)$, depending on the context.
Since $r$ is separable, Proposition \ref{blockcyclicprop} shows that
$C(r^n)$ is similar to
$$\begin{bmatrix}
C(r) & 0_2 & \cdots & \cdots & (0) \\
I_2 & C(r) & \ddots & & \vdots \\
0_2 & \ddots & \ddots & \ddots & \vdots \\
\vdots & & \ddots & C(r) & 0_2 \\
(0) & \cdots & 0_2 & I_2 & C(r)
\end{bmatrix}$$
in $\Mat_{2n}(\K)$. This last matrix can be interpreted as the matrix
$$M=\begin{bmatrix}
P & 0_4 & \cdots & \cdots & (0) \\
I_4 & P & \ddots & & \vdots \\
0_4 & \ddots & \ddots & \ddots & \vdots \\
\vdots & & \ddots & P & 0_4 \\
(0) & \cdots & 0_4 & I_4 & P
\end{bmatrix}$$
of $\Mat_{4n}(\F)$ for some $P \in \Mat_4(\F)$ that is annihilated by $F_{p,q}$.
Since $F_{p,q}$ is irreducible and separable, it follows once more from Proposition \ref{blockcyclicprop}
that $M$ is similar to $C(F_{p,q}^n)$ in $\Mat_{4n}(\F)$.
Hence, $C(r^n)$ is similar to $C(F_{p,q}^n)$ in $\Mat_{4n}(\F)$.
Note that this remains (trivially) true if $n=0$, and that this remains true if $r$ is replaced with $s$.

We are now ready to prove the implications (iii) $\Rightarrow$ (i) and (i) $\Rightarrow$ (ii).

In order to prove (iii) $\Rightarrow$ (i), it suffices to prove that, for all $n \in \N^*$, both matrices
$C(F_{p,q}^n) \oplus C(F_{p,q}^n)$ and  $C(F_{p,q}^n) \oplus C(F_{p,q}^{n-1})$ are $(p,q)$-differences.
For the first one it suffices to use the Duplication Lemma, since $F_{p,q}(t)=\Lambda_{p,q}(t^2-\delta t)$.
Next, let $n \in \N^*$. Since $p$ and $q$ satisfy the assumptions of Theorem \ref{theoSamesplitNottranslationSeparable} over $\K$, we
get that the matrix $C(r^n) \oplus C(s^{n-1})$ is a $(p,q)$-difference in $\Mat_{4n-2}(\K)$.
Hence, it is also a $(p,q)$-difference in $\Mat_{8n-4}(\F)$.
Yet, we have just seen that this matrix is similar to $C(F_{p,q}^n) \oplus C(F_{p,q}^{n-1})$ in $\Mat_{8n-4}(\F)$.
It follows that the latter is a $(p,q)$-difference. Hence, (iii) $\Rightarrow$ (i) is proved.

Conversely, assume that $u$ is a d-exceptional $(p,q)$-difference.
We choose endomorphisms $a$ and $b$ of the $\F$-vector space $V$ such that $u=a-b$ and $p(a)=q(b)=0$.
The endomorphism $v:=u^2-\delta u$ is annihilated by some power of the separable polynomial $\Lambda_{p,q}$.
Hence, by the Jordan-Chevalley decomposition, we have a splitting $v=S+N$ in which $S$ is a semi-simple endomorphism of $V$
that is annihilated by $\Lambda_{p,q}$ and that belongs to $\F[v]$, and $N$ is a nilpotent endomorphism of $V$.
Hence, $a$ and $b$ turn out to be endomorphisms of the $\F[S]$-vector space $V$, and so does $u$.
We denote by $V^S$ the $\F[S]$-vector space $V$ to differentiate it from the $\F$-vector space $V$.
Now, $\F[S]$ is isomorphic to the splitting field of $\Lambda_{p,q}$ over $\F$, and hence $\F[S] \simeq \K$.
Applying Theorem \ref{theoSamesplitNottranslationSeparable} to $V^S$,
we get that the endomorphism $u$ of $V^S$ is represented by a block-diagonal matrix in which each diagonal block has
one of the forms $C(r^k) \oplus C(s^{k-1})$, $C(r^k) \oplus C(s^k)$, or $C(s^k) \oplus C(r^{k-1})$ for some positive integer $k$.
It follows from our initial study that condition (iii) is satisfied by $u$, which completes the proof.
\end{proof}

\paragraph{Case 1.2. $p$ and $q$ have the same discriminant}

Here, $\car(\F)=2$, $\tr(p)=\tr(q)$,
$F_{p,q}=p(t+y_1)^2$, and $r:=p(t+y_1)=t^2+(\tr p)\,t+p(0)+q(0)$ belongs to $\F[t]$
and is irreducible over $\F$.
Hence, an endomorphism of a finite-dimensional vector space over $\F$
is d-exceptional with respect to $(p,q)$ if and only if its minimal polynomial is a power of $r$.
Note also that $\Lambda_{p,q}=t^2+(\tr p)^2\, t+(p(0)+q(0))^2$.
Remembering that $\Lambda_{p,q}$ is irreducible over $\F$, and noting that
the square of any root of $r$ is a root of $\Lambda_{p,q}$, we get that
$\Lambda_{p,q}$ and $r$ have the same splitting field in $\L$.

\begin{theo}\label{theoFpqdoubleroots}
Let $p$ and $q$ be monic polynomials with degree $2$ over $\F$,
and assume that they are both irreducible with distinct splitting fields and the same non-zero discriminant.
Set $r:=t^2+(\tr p)\,t+p(0)+q(0)$.
Let $u$ be an endomorphism of a finite-dimensional vector space $V$ over $\F$. Then, the following conditions are equivalent:
\begin{enumerate}[(i)]
\item The endomorphism $u$ is a d-exceptional $(p,q)$-difference.
\item The minimal polynomial of $u$ is a power of $r$, and if we denote by $r_1,\dots,r_k,\dots$
the invariant factors of $u$, then $r_{2k-1}=r_{2k}$ for all $k \in \N^*$.
\item In some basis of $V$, the endomorphism $u$ is represented by a block-diagonal matrix in which
each diagonal block equals $C(r^n) \oplus C(r^n)$ for some $n \in \N^*$.
\end{enumerate}
\end{theo}

\begin{proof}
It is obvious that conditions (ii) and (iii) are equivalent.

Next, we prove that conditions (i) and (iii) are equivalent.
As in the proof of Theorem \ref{theoFpq4roots}, some prior work is required.

Remember that $\car(\F)=2$ and $\tr p=\tr q$.
The splitting field $\K$ of $\Lambda_{p,q}$ over $\L$ can be identified with
a subalgebra of $\Mat_2(\F)$. We have seen that $\K$ is also the splitting field of $r$ over $\F$.
Let us split $p(t)=(t-x_1)(t-x_2)$ and $q(t)=(t-y_1)(t-y_2)$ in $\L[t]$.
Note that $x_1-y_1 \in \K$, $x_1-y_2 \in \K$, and $q$ is a translation of $p$ over $\K$.

Next, fix $n \in \N^*$. Set
$$A_n:=\begin{bmatrix}
x_1-y_1 & 0 & \cdots & \cdots & (0) \\
1 & x_1-y_1 & \ddots & & \vdots \\
0 & \ddots & \ddots & \ddots & \vdots \\
\vdots & & \ddots & x_1-y_1 & 0 \\
(0) & \cdots & 0 & 1 & x_1-y_1
\end{bmatrix}$$ and
$$B_n:=\begin{bmatrix}
x_1-y_2 & 0 & \cdots & \cdots & (0) \\
1 & x_1-y_2 & \ddots & & \vdots \\
0 & \ddots & \ddots & \ddots & \vdots \\
\vdots & & \ddots & x_1-y_2 & 0 \\
(0) & \cdots & 0 & 1 & x_1-y_2
\end{bmatrix}$$
in $\Mat_n(\K)$, and set further
$$M_n:=A_n \oplus A_n \quad \text{and} \quad N_n:=B_n \oplus B_n,$$
which we see as matrices of $\Mat_{2n}(\K)$.
Here, $x_1-y_1$ is a root of $r$. Seeing $\K$ as a subalgebra of $\Mat_2(\F)$, we have
$$A_n=\begin{bmatrix}
x_1-y_1 & 0_2 & \cdots & \cdots & (0) \\
I_2 & x_1-y_1 & \ddots & & \vdots \\
0_2 & \ddots & \ddots & \ddots & \vdots \\
\vdots & & \ddots & x_1-y_1 & 0_2 \\
(0) & \cdots & 0_2 & I_2 & x_1-y_1
\end{bmatrix} \in \Mat_{2n}(\F).$$
Since $x_1-y_1$ is annihilated by $r$, which is separable with degree $2$,
Proposition \ref{blockcyclicprop} shows that
$A_n$ is similar to $C(r^n)$ in $\Mat_{2n}(\F)$. Likewise, $B_n$ is similar to $C(r^n)$ in $\Mat_{2n}(\F)$,
and we conclude that both $M_n$ and $N_n$ are similar to $C(r^n) \oplus C(r^n)$ in $\Mat_{4n}(\F)$.

We are now ready to conclude. For all $n \in \N^*$, we know from Theorem \ref{theoSamesplitTranslationSeparable}
that $M_n$ is a $(p,q)$-difference in $\Mat_{2n}(\K)$, and hence it is also a $(p,q)$-difference in $\Mat_{4n}(\F)$.
Hence, condition (iii) implies condition (i).
Conversely, assume that $u$ is a d-exceptional $(p,q)$-difference.
Let $a,b$ be endomorphisms of $V$ such that $u=a-b$ and $p(a)=q(b)=0$.
Setting $v:=u^2$, we see that $v$ is annihilated by some power of $\Lambda_{p,q}$. Since $\Lambda_{p,q}$ is separable, we can use the Jordan-Chevalley
decomposition $v=s+n$ in which $s$ is semi-simple, $n$ is nilpotent and $s$ is a polynomial in $v$.
Note that $\Lambda_{p,q}(s)=0$.
By the Commutation Lemma, both $a$ and $b$ commute with $s$, and hence $a$ and $b$ are endomorphisms of the $\F[s]$-vector space
$V$, which we denote by $V^s$. Hence, $u$ is a d-exceptional $(p,q)$-difference in the algebra of all endomorphisms of $V^s$.
Yet, $\F[s] \simeq \K$, and hence, by Theorem \ref{theoSamesplitTranslationSeparable}, in some basis of $V^s$ the endomorphism
$u$ is represented by a block-diagonal matrix in which every diagonal block equals either $C\Bigl(\bigl(t-(x_1-y_1)\bigr)^k\Bigr) \oplus
C\Bigl(\bigl(t-(x_1-y_1)\bigr)^k\Bigr)$ or $C\Bigl(\bigl(t-(x_1-y_2)\bigr)^k\Bigr) \oplus
C\Bigl(\bigl(t-(x_1-y_2)\bigr)^k\Bigr)$ for some $k \in \N^*$.
Hence, there is a basis of the $\F$-vector space $V$ in which $u$ is represented by a block-diagonal
matrix in which every diagonal block equals $M_k$ or $N_k$ for some $k \in \N^*$. Using the preliminary work on block matrices, we conclude that condition (iii) holds.
\end{proof}

\subsubsection{Case 2. Both $p$ and $q$ are inseparable}

Here, we will need another lemma:

\begin{lemma}\label{LemmaPandQinseparable}
Assume that $\car(\F)=2$. Let $p$ and $q$ be monic polynomials with degree $2$ of $\F[t]$, both irreducible and inseparable,
and assume that $p$ and $q$ have distinct splitting fields in $\overline{\F}$.
Let $a$ and $b$ be endomorphisms of a finite-dimensional vector space $V$ such that $p(a)=q(b)=0$.
Set $w:=ab+ba$. Then, for all $k \in \N^*$, the integer $n_k(w,0)$ is a multiple of $4$.
\end{lemma}

\begin{proof}
Denoting by $a^\star$ the $p$-conjugate of $a$ and by $b^\star$ the $q$-conjugate of $b$, we have $a^\star=a$ and $b^\star=b$, whence
$w=ab^\star+b a^\star$, and the Basic Commutation Lemma yields that both $a$ and $b$ commute with $w$.
Let $k \in \N^*$. Then, we know that $a$ and $b$ induce endomorphisms $a'$ and $b'$ of the
quotient vector space
$$E:=\Ker w^k/\Ker w^{k-1}$$
such that $p(a')=0=q(b')$, and $a'b'+b'a'=0$. Here, $\F$ has characteristic $2$, whence
$a'$ and $b'$ commute. Using $p(a')=0$ and $q(b')=0$, it follows that the subalgebra $\L:=\F[a',b']$ generated by $a'$ and $b'$ is actually isomorphic to the splitting field of $pq$ in $\overline{\F}$, and in particular it is a field with degree $4$ over $\F$. Hence, $E$ can be seen as an $\L$-vector space, leading to
$\dim_\F E=4 \dim_\L E$. We conclude that $n_k(w,0)=\dim_\F E$ is a multiple of $4$, as claimed.
\end{proof}

\begin{theo}\label{theoPandQinseparable}
Assume that $\car(\F)=2$.
Let $\alpha$ and $\beta$ be elements of $\F$, set $p(t):=t^2-\alpha$ and $q(t):=t^2-\beta$
and assume that both $p$ and $q$ are irreducible over $\F$ and that they have distinct splitting fields in $\overline{\F}$.
Let $u$ be an endomorphism of a finite-dimensional vector space $V$. Then, the following conditions are equivalent:
\begin{enumerate}[(i)]
\item $u$ is a d-exceptional $(p,q)$-difference.
\item Every invariant factor of $u$ is a power of $t^2-\alpha-\beta$, and if we denote by $r_1,\dots,r_k,\dots$
those invariant factors we have $r_{2k-1}=r_{2k}$ for all $k \in \N^*$.
\item In some basis of $V$, the endomorphism $u$ is represented by a block-diagonal matrix in which every diagonal block
equals $C\bigl((t^2-\alpha-\beta)^n\bigr) \oplus C\bigl((t^2-\alpha-\beta)^n\bigr)$ for some $n \in \N^*$.
\end{enumerate}
\end{theo}

\begin{proof}
The equivalence between conditions (ii) and (iii) is obvious.

Before we prove that conditions (i) and (ii) are equivalent, some preliminary work is required.
First of all, here we have $F_{p,q}=(t^2-\alpha-\beta)^2$. Set $r:=t^2-\alpha-\beta$.
Since $p$ and $q$ have distinct splitting fields, the roots $\sqrt{\alpha}$ and $\sqrt{\beta}$ in $\L$
are linearly independent over $\F$, and hence $r$ is irreducible over $\F$.
It follows that $u$ is d-exceptional with respect to $(p,q)$ if and only if it is annihilated by some power of $r$.

From there, the Duplication Lemma yields that condition (iii) implies condition (i).
In order to conclude, we prove that condition (i) implies condition (ii).
Assume that condition (i) holds. First of all, we know that each invariant factor of $u$
is a power of $r$. Let $a$ and $b$ be endomorphisms of $V$ such that $p(a)=0=q(b)$ and $u=a-b$.
Set $v:=u^2$ and note that $ab+ba=v-(\alpha+\beta)\id_V$.
Let $n \in \N^*$ and denote by $N_n$ the number of invariant factors of $u$ that equal $(t^2-\alpha-\beta)^k$ for some $k \geq n$.
Then, $2N_n=\dim \Ker\bigl(v-(\alpha+\beta)\,\id_V\bigr)^n-\dim \Ker\bigl(v-(\alpha+\beta)\,\id_V\bigr)^{n-1}$ which, by Lemma \ref{LemmaPandQinseparable},
is a multiple of $4$. Hence, $N_n$ is even. Finally, for all $n \in \N^*$, the number of invariant factors of $u$ that equal $r^n$
is $N_n-N_{n+1}$, and hence it is even. It follows that condition (ii) holds.
\end{proof}

Combining Theorems \ref{theoFpqdoubleroots} and \ref{theoPandQinseparable}, we
deduce the classification of indecomposable d-exceptional $(p,q)$-differences
in the case when $p$ and $q$ have distinct splitting fields and the same discriminant,
as given in Table \ref{dfigure9}.

\subsubsection{Case 3. $p$ is separable and $q$ is not}

This is the last remaining case, and by far the most difficult one.

Here, $\car(\F)=2$. The splitting field $\L$ of $pq$ is not a Galois extension of $\F$.
Yet, it is not a radicial extension either because $p$ is irreducible and separable.
Hence, we have a decomposition $\F - R - \L$ where $R$ is a radicial quadratic extension of $\F$
and $\L$ is a separable extension of $R$. Explicitly, $R$ is the set of all $x \in \L$ such that $x^2 \in \F$.
Moreover, $\Gal(\L/\F)=\Gal(\L/R)$ has cardinality $2$.

Since $q$ is inseparable, $R$ is its splitting field in $\L$.

Let us split $p(t)=(t-x_1)(t-x_2)$ and $q(t)=(t-y)^2$ in $\L[t]$.
Let $u$ be an endomorphism of a finite-dimensional vector space $V$, and
set $v:=u^2-\tr(p) u=u^2-\delta u$.
Here,
$$F_{p,q}(t)=\bigl(t^2-\delta t+(x_1-y)(x_2-y)\bigr)^2 \quad \text{and} \quad \Lambda_{p,q}(t)=\bigl(t-p(y)\bigr)^2.$$
Yet, $p(y)=y^2-(\tr p)y+p(0)$ belongs to $R \setminus \F$ since $\tr p \neq 0$ and $y \in R \setminus \F$.
It follows that $\Lambda_{p,q}$ is irreducible over $\F$.
Next, $F_{p,q}$ is also irreducible over $\F$: indeed, it is split over $\L$, the Galois group of $\L$ over $\F$ acts transitively on the set $\{x_1-y,x_2-y\}$ of its roots in $\L$, and hence the only possible monic irreducible proper divisor of $F_{p,q}$ would be $(t-(x_1-y))(t-(x_2-y))=t^2-\delta t+p(y)$.
Yet, this last polynomial does not belong to $\F[t]$ since $p(y)\not\in \F$.

Hence, given an endomorphism $u$ of a finite-dimensional vector space over $\F$,
the minimal polynomial of $u$ is a power of $F_{p,q}$ if and only if $u$ is d-exceptional with respect to $(p,q)$.

Before we can obtain the classification of d-exceptional $(p,q)$-differences, we need a technical result
that is quite similar to Lemma \ref{CNsamesplitbasicLemma}:

\begin{prop}\label{lastprop}
Let $p$ and $q$ be irreducible monic polynomials with degree $2$ over $\F$.
Assume that $\tr p=1$ and that $q$ is inseparable.
Let $a,b$ be endomorphisms of a finite-dimensional vector space $V$ such that $p(a)=q(b)=0$, and denote by $a^\star$ the $p$-conjugate of $a$
and by $b^\star$ the $q$-conjugate of $b$. Set $w:=ab^\star+ba^\star$.
For $k \in \N^*$, set $n_k(w):=\dim \Ker (q(w)^{k})-\dim \Ker (q(w)^{k-1})$.

Let $k \in \N^*$. If $n_k(w)=n_{k+1}(w)$ then $n_k(w)$ is a multiple of $8$.
\end{prop}

The proof is split into two consecutive lemmas, in which we keep the same notation and assumptions as in Proposition \ref{lastprop} with regards to the polynomials $p$ and $q$.

\begin{lemma}\label{lastlemma1}
Let $a$ and $b$ be elements of an $\F$-algebra $\calA$, and assume that $p(a)=q(b)=0$.
Set $w:=ab^\star+ba^\star$ (where $a^\star$ denotes the $p$-conjugate of $a$, and $b^\star$ the $q$-conjugate of $b$).
Assume that $q(w)=0$. Then, the element $n:=b-w$ of $\calA$ satisfies the following conditions:
\begin{enumerate}[(i)]
\item $n^2=0$;
\item $n$ commutes with $w$;
\item $na=a^\star n$ and $na^\star=an$.
\end{enumerate}
\end{lemma}

\begin{proof}[Proof of Lemma \ref{lastlemma1}]
By the Basic Commutation Lemma, $a$, $b$ and $n$ commute with $w$.
Note that $\car(\F)=2$, due to the assumptions on $q$. Since $b$ commutes with $w$,
we deduce that $n^2=b^2+w^2=-q(0)1_\calA-q(0)1_\calA=0$. Finally, since $\tr q=0$ and $\tr p=1$,
$$b-n=w=ab^\star+ba^\star=-ab+b-ba,$$
whence $n=ab+ba$. Since $w$ commutes with $a$, this yields $n=a(b-w)+(b-w)a=an+na$, and condition (iii) follows.
\end{proof}

\begin{lemma}\label{lastlemma2}
With the assumptions and notation from Proposition \ref{lastprop},
assume that $q(w)^2=0$ and $\dim V=2\dim \Ker q(w)$. Then, the dimension of $\Ker q(w)$ is a multiple of $8$.
\end{lemma}

\begin{proof}[Proof of Lemma \ref{lastlemma2}]
Set $V_1:=\Ker q(w)$ and $u:=a-b$.
By the Basic Commutation Lemma, $V_1$ is stable under both $a$ and $b$, and hence under $u$.
Here
\begin{align*}
F_{p,q}(u) & =(u^2-u+p(y)\,\id_V)^2 \\
& =\bigl(w-(p(0)+q(0)+p(y))\,\id_V\bigr)^2 \\
& =\bigl(w-(y^2+y+q(0))\,\id_V\bigr)^2=(w-y)^2=q(w).
\end{align*}
Hence, the irreducible polynomial $F_{p,q}$ (with degree $4$) annihilates $u_{|V_1}$, and it follows that
$\dim V_1=4k$ for some non-negative integer $k$. We shall prove that $k$ is even.

Since $V_1$ is stable under $a$, which is annihilated by the irreducible polynomial $p$,
there exists a direct factor $V_2$ of $V_1$ in the $\F$-vector space $V$ that is stable under $a$.
Note that $\dim V_2=\dim (V/V_1)=\dim V_1$.
Hence, in some basis $\bfB$ of the vector space $V$ that is adapted to the decomposition $V=V_1\oplus V_2$,
we have, for some matrices $A$, $B_1$, $B_2$, $Q_1$, $Q_2$ and $C$ of $\Mat_{4k}(\F)$,
$$\Mat_\bfB(a)=\begin{bmatrix}
A & 0 \\
0 & A
\end{bmatrix}, \; \Mat_\bfB(b)=\begin{bmatrix}
B_1 & C  \\
0 & B_2
\end{bmatrix} \quad \text{and} \quad \Mat_\bfB(w)=\begin{bmatrix}
Q_1 & ?  \\
0 & Q_2
\end{bmatrix}.$$
More precisely, $A$ can be taken as the direct sum of $2k$ copies of the companion matrix $C(p)$.

Using $p(a)=q(b)=0$ and $q(w)=0$, we get that $p(A)=q(B_1)=q(B_2)=q(Q_1)=q(Q_2)=0$.
Denoting by $A^\star,B_1^\star,B_2^\star$ the $p$-conjugate of $A$, the $q$-conjugate of $B_1$, and the $q$-conjugate of $B_2$, respectively,
we find that $Q_i=AB_i^\star+B_iA^\star$ for all $i \in \{1,2\}$. Hence, Lemma \ref{lastlemma1} yields that the matrices
$$N_1:=B_1-Q_1 \quad \text{and} \quad N_2:=B_2-Q_2$$
have square zero and satisfy the following equalities:
\begin{equation}\label{commuteequations}
\forall i \in \{1,2\}, \quad N_iA=A^\star N_i, \quad AN_i=N_i A^\star\quad \text{and} \quad N_iQ_i=Q_iN_i.
\end{equation}
In the next key step, we shall prove that $N_1$ has rank $2k$.

Note first, by the definition of $w$, that
$$\Mat_\bfB(w)=\begin{bmatrix}
Q_1 & AC+CA^\star  \\
0 & Q_2
\end{bmatrix},$$
and it follows, since $\tr q=0$, that
\begin{align*}
\Mat_\bfB(q(w)) & =\begin{bmatrix}
q(Q_1) & Q_1(AC+CA^\star)+(AC+CA^\star)Q_2  \\
0 & q(Q_2)
\end{bmatrix} \\
& =
\begin{bmatrix}
0 & Q_1(AC+CA^\star)+(AC+CA^\star)Q_2  \\
0 & 0
\end{bmatrix}.
\end{align*}
Using the definition of $V_1$, we conclude that
$$Q_1(AC+CA^\star)+(AC+CA^\star)Q_2 \quad \text{is invertible.}$$
We shall see that this leads to the equality $\rk N_1=2k$.

First of all, since $q$ is irreducible and both matrices $\begin{bmatrix}
B_1 & 0 \\
0 & B_2
\end{bmatrix}$ and $\begin{bmatrix}
B_1 & C \\
0 & B_2
\end{bmatrix}$ are annihilated by $q$, they are similar (as both are similar to the direct sum of $4k$ copies of $C(q)$).
By Roth's theorem \cite{Roth}, this yields a matrix $X \in \Mat_{4k}(\F)$
such that
$$C=B_1X-X B_2.$$

Next, consider the endomorphism
$$\varphi : M \in \Mat_{4k}(\F) \mapsto Q_1 M+MQ_2 \in \Mat_{4k}(\F).$$
Since $Q_1^2$ and $Q_2^2$ are equal scalar multiples of $I_{4k}$, we find that
$$\forall M \in \Mat_{4k}(\F), \quad Q_1 \varphi(M)=\varphi(M) Q_2,$$
i.e.\
\begin{equation}\label{lastidentity}
\forall M \in \Mat_{4k}(\F), \quad
(B_1-N_1) \varphi(M)=\varphi(M)(B_2-N_2).
\end{equation}

On the other hand, by the Basic Commutation Lemma, we see that $A$ and $A^\star$ commute with $Q_1$ and $Q_2$, whence
$$\forall M \in \Mat_{4k}(\F), \; \varphi(AM)=A\varphi(M) \; \text{and} \; \varphi(MA^\star)=\varphi(M)A^\star,$$
whereas $B_i$ commutes with $Q_i$ for all $i \in \{1,2\}$, whence
$$\forall M \in \Mat_{4k}(\F), \; \varphi(B_1M)=B_1 \varphi(M) \; \text{and} \; \varphi(MB_2)=\varphi(M)B_2.$$
Hence,
\begin{align*}
\varphi(AC+CA^\star) & =A \varphi(C)+\varphi(C)A^\star \\
& =A(B_1 \varphi(X)-\varphi(X) B_2)+(B_1 \varphi(X)-\varphi(X) B_2) A^\star \\
& =A(N_1 \varphi(X)-\varphi(X) N_2)+(N_1 \varphi(X)-\varphi(X) N_2) A^\star \quad \text{(by \eqref{lastidentity})} \\
& =N_1 A^\star \varphi(X)-A \varphi(X) N_2+N_1 \varphi(X) A^\star-\varphi(X) A N_2 \quad \text{(by \eqref{commuteequations})} \\
& =N_1\bigl(A^\star \varphi(X)+\varphi(X)A^\star\bigr)-\bigl(A\varphi(X)+\varphi(X)A\bigr)N_2.
\end{align*}
It follows that
$$\rk \bigl(\varphi(AC+CA^\star)\bigr) \leq \rk N_1+\rk N_2.$$
Yet, $N_1$ and $N_2$ are square-zero matrices of $\Mat_{4k}(\F)$ and hence $\rk N_1 \leq 2k$ and $\rk N_2 \leq 2k$.
On the other hand $\varphi(AC+CA^\star)$ is invertible. Hence $\rk N_1=2k$.

We are now ready to conclude. Remembering that $A$ and $Q_1$ commute with one another and are respectively
annihilated by $p$ and $q$, we find that the subalgebra $\F[A,Q_1]$ of $\Mat_{4k}(\F)$ is
isomorphic to $\L$, and in particular it is a field.
Without loss of generality, we can simply assume that $\mathbb{L}=\F[A,Q_1]$. In that identification,
the non-identity automorphism $\sigma$ in $\Gal(\L/\F)$ leaves $Q_1$ invariant and maps $A$ to $A^\star$.
The equalities $N_1Q_1=Q_1N_1$ and $N_1A=A^\star N_1$ then yield that the endomorphism $m : X \mapsto N_1 X$
of the $\F$-vector space $\F^{4k}$ satisfies
$$\forall \lambda \in \L, \; \forall x \in \F^{4k}, \; m(\lambda x)=\sigma(\lambda)\,m(x).$$
In other words, $m$ is semilinear with respect to $\sigma$.
It follows that the range of $m$ is an $\L$-linear subspace of $\F^{4k}$, and hence its dimension over $\F$ is a multiple of $4$.
Hence, $2k$ is a multiple of $4$, and we finally obtain the desired conclusion that $k$ is even.
This completes the proof.
\end{proof}

Now, we are ready to conclude the proof of Proposition \ref{lastprop}:

\begin{proof}[Proof of Proposition \ref{lastprop}]
Let $k \in \N^*$. Assume that $n_k(w)=n_{k+1}(w)$. For all $i \in \N^*$, we know from the Basic Commutation Lemma that
$a$ and $b$ stabilize $\Ker q(w)^i$. Setting $E:=\Ker q(w)^{k+1}/\Ker q(w)^{k-1}$, we deduce that
$a$ and $b$ induce endomorphisms $a'$ and $b'$ of $E$ such that $p(a')=q(b')=0$.
Moreover, the endomorphism $w'$ of $E$ induced by $w$ satisfies $q(w')^2=0$, and
$\Ker q(w')=\Ker q(w)^{k}/\Ker q(w)^{k-1}$ has dimension $n_k(w)$, while $\dim E=n_k(w)+n_{k+1}(w)=2n_k(w)$.
Finally, $a'(b')^\star+b'(a')^\star=w'$, where $(a')^\star$ is the $p$-conjugate of $a'$ and $(b')^\star$ is the $q$-conjugate of $b'$.
Thus, Lemma \ref{lastlemma2} applies to the pair $(a',b')$, whence $n_k(w)$ is a multiple of $8$.
\end{proof}

Having done that preliminary work, we are now ready to classify the d-exceptional $(p,q)$-differences.

\begin{theo}\label{theolastcase}
Let $p$ and $q$ be irreducible monic polynomials with degree $2$ over $\F$.
Assume that $p$ is separable and that $q$ is inseparable.
Let $u$ be an endomorphism of a finite-dimensional vector space $V$ over $\F$.
Then, the following conditions are equivalent:
\begin{enumerate}[(i)]
\item $u$ is a d-exceptional $(p,q)$-difference.
\item The minimal polynomial of $u$ is a power of $F_{p,q}$ and
the invariant factors of $u$ read $r_1,\dots,r_k,\dots$, where, for each positive integer $k$,
either $r_{2k}=r_{2k-1}$ or $r_{2k-1}=r_{2k}\,F_{p,q}$.
\item In some basis of $V$, the endomorphism $u$ is represented by a block-diagonal matrix
in which every diagonal block equals $C(F_{p,q}^{n+\epsilon}) \oplus C(F_{p,q}^n)$ for some non-negative integer
$n$ and some $\epsilon \in \{0,1\}$.
\end{enumerate}
\end{theo}

\begin{proof}
Let us write $q=(t-y)^2=t^2-y^2$.

The equivalence between conditions (ii) and (iii) is obtained in exactly the same way as in the proof of Theorem \ref{theoFpq4roots}.

Next, we prove that condition (iii) implies condition (i).
Obviously, it suffices to fix a positive integer $n$ and to prove that both matrices
$C(F_{p,q}^{n-1}) \oplus C(F_{p,q}^n)$ and $C(F_{p,q}^{n}) \oplus C(F_{p,q}^n)$ are $(p,q)$-differences
(for the second one, this can be directly obtained as a consequence of the Duplication Lemma).
First, fix a positive integer $k$. Without loss of generality, we can assume that $R=\F[C(q)]$.
Over $R$, the polynomial $q$ splits with a double root, whereas $p$ remains irreducible and separable.
Hence, by Theorem \ref{theoQsplitswithdouble} the matrix $C\bigl(p(t+y)^k\bigr)$ of $\Mat_{2k}(R)$ is a $(p,q)$-difference.
Since $p(t+y)$ is separable over $R$, Proposition \ref{blockcyclicprop} yields that the
matrix
$$M_k=\begin{bmatrix}
C\bigl(p(t+y)\bigr) & 0_2 & \cdots & \cdots & (0) \\
I_2 & C\bigl(p(t+y)\bigr) & \ddots & & \vdots \\
0_2 & \ddots & \ddots & \ddots & \vdots \\
\vdots & & \ddots & C\bigl(p(t+y)\bigr) & 0_2 \\
(0) & \cdots & 0_2 & I_2 & C\bigl(p(t+y)\bigr)
\end{bmatrix}$$
of $\Mat_{2k}(R)$ is similar to $C(p(t+y)^k)$, and is therefore a $(p,q)$-difference.

Viewing $C\bigl(p(t+y)\bigr)$ as a matrix $P$ of $\Mat_4(\F)$, we deduce that $F_{p,q}$ annihilates $P$ and that the matrix
$$\begin{bmatrix}
P & 0_4 & \cdots & \cdots & (0) \\
I_4 & P & \ddots & & \vdots \\
0_4 & \ddots & \ddots & \ddots & \vdots \\
\vdots & & \ddots & P & 0_4 \\
(0) & \cdots & 0_4 & I_4 & P
\end{bmatrix}$$
of $\Mat_{4k}(\F)$ is a $(p,q)$-difference.
Here $F_{p,q}$ is irreducible with double roots. Therefore, Proposition \ref{blockcyclicprop}
yields that the above matrix of $\Mat_{4k}(\F)$ is similar to $C(F_{p,q}^{s+\epsilon}) \oplus C(F_{p,q}^s)$,
where $s$ and $\epsilon$ respectively denote the quotient and the remainder of $k$ mod $2$.
Varying $k$ then yields the claimed result, and we conclude that condition (iii) implies condition (i).

It remains to prove that condition (i) implies condition (ii).
Assume therefore that $u$ is a d-exceptional $(p,q)$-difference. Let $a$ and $b$ be endomorphisms of $V$ such that $u=a-b$ and $p(a)=q(b)=0$.
Since $F_{p,q}$ is monic and irreducible, the minimal polynomial of $u$ must be a power of it.
From there, we set $v:=u^2-\delta u$, with $\delta:=\tr p-\tr q$, and we proceed as in the proof of Theorem
\ref{theoSamesplitNottranslationSeparable}.
It suffices to prove the following statement:
for all $k \in \N$, if $\dim \bigl(\Ker F_{p,q}(u)^{k+1}/\Ker F_{p,q}(u)^k\bigr)=\dim \bigl(\Ker F_{p,q}(u)^{k+2}/\Ker F_{p,q}(u)^{k+1}\bigr)$,
then $\dim \bigl(\Ker F_{p,q}(u)^{k+1}/\Ker F_{p,q}(u)^k\bigr)$ is a multiple of $8$.
Set $a':=(\tr p)^{-1}\,a$ and note that $a'$ is annihilated by $p_1:=t^2+t+\frac{p(0)}{(\tr p)^2}\cdot$
Denoting by $(a')^\star$ the $p_1$-conjugate of $a$, one sees that $a b^\star+b a^\star=(\tr p) (a' b^\star+b (a')^\star)$,
and hence by setting $w:=a' b^\star+b (a')^\star$, we have
$$v=a b^\star+b a^\star-(p(0)+q(0))\id_V=(\tr p)\Bigl(w-\frac{p(0)+q(0)}{\tr p}\,\id_V\Bigr).$$
Since
$$p(0)+q(0)+p(y)=p(0)+y^2+p(y)=(\tr p)\, y,$$
we find $(p(0)+q(0))^2+p(y)^2=(\tr p)^2\, y^2$ and it follows that
$$F_{p,q}(u)=\Lambda_{p,q}(v)=v^2-p(y)^2\,\id_V=(\tr p)^2(w^2-y^2\,\id_V)=(\tr p)^2 q(w).$$
Hence, for all $k \in \N$, we find $\Ker F_{p,q}(u)^k=\Ker q(w)^k$.
The claimed result then follows directly from Proposition \ref{lastprop} applied to the pair $(a',b)$ and to the polynomials $p_1$ and $q$.
\end{proof}

Combining Theorem \ref{theoFpq4roots} with Theorem \ref{theolastcase}, we obtain the classification of indecomposable d-exceptional $(p,q)$-differences when
$p$ and $q$ are irreducible with distinct splitting fields and distinct discriminants, as given in Table \ref{dfigure8}.

\newpage

\section{The quotient of two invertible quadratic matrices}\label{QuotSection}

\subsection{The basic splitting}

Let $u$ be an automorphism of a finite-dimensional vector space $V$ over $\F$.
Let $p$ and $q$ be monic polynomials with degree $2$ over $\F$, with $p(0)q(0) \neq 0$.

The \textbf{q-fundamental polynomial} of the pair $(p,q)$ is defined as the resultant
$$G_{p,q}(t):=\res\bigl(p(x),q(0)^{-1} t^2q(x/t)\bigr) \in \F[t],$$
which is a monic polynomial of degree $4$.
More explicitly, if we split $p(t)=(t-x_1)(t-x_2)$ and $q(t)=(t-y_1)(t-y_2)$ in $\overline{\F}[t]$,
then
$$G_{p,q}(t)=\prod_{1 \leq i,j \leq 2}\bigl(t-x_iy_j^{-1}\bigr)=q(0)^{-2} p(y_1t)\,p(y_2t)=q(0)^{-2}\,t^4\, q(x_1 t^{-1})\, q(x_2 t^{-1}).$$
We set
$$E'_{p,q}(u):=\underset{n \in \N}{\bigcup} \Ker G_{p,q}(u)^n \quad \text{and} \quad R'_{p,q}(u):=\underset{n \in \N}{\bigcap} \im G_{p,q}(u)^n.$$
Hence, $V=E'_{p,q}(u) \oplus R'_{p,q}(u)$, and the endomorphism $u$ stabilizes both linear subspaces
$E'_{p,q}(u)$ and $R'_{p,q}(u)$.
The endomorphism $u$ is called \textbf{q-exceptional} with respect to $(p,q)$ (respectively, \textbf{q-regular} with respect to $(p,q)$)
whenever $E'_{p,q}(u)=V$ (respectively, $R'_{p,q}(u)=V$).
In other words, $u$ is q-exceptional (respectively, q-regular) with respect to $(p,q)$
if and only if all the eigenvalues of $u$ in $\overline{\F}$ belong to $\Root(p)\Root(q)^{-1}$
(respectively, no eigenvalue of $u$ in $\overline{\F}$ belongs to $\Root(p)\Root(q)^{-1}$).

The endomorphism of $E'_{p,q}(u)$ (respectively, of $R'_{p,q}(u)$)
induced by $u$ is always q-exceptional (respectively, always regular) with respect to $(p,q)$
and we call it the \textbf{q-exceptional part} (respectively, the \textbf{q-regular part}) of $u$ with respect to $(p,q)$.

Finally, we set
\begin{align*}
\Theta_{p,q}& :=t^2-(\tr p)(\tr q)t+\bigl(q(0)(\tr p)^2+p(0)(\tr q)^2-4\,p(0)q(0)\bigr) \\
& =\bigl(t-(x_1y_1+x_2y_2)\bigr)\bigl(t-(x_1y_2+x_2y_1)\bigr).
\end{align*}
One sees that, for
$$v:=q(0)\, u+p(0)\, u^{-1},$$
we have
$$\bigl(u-x_1y_1^{-1}\id_V\bigr)\bigl(u-x_2y_2^{-1}\id_V\bigr)=q(0)^{-1} u\,\bigl(v-(x_1y_2+x_2y_1)\,\id_V\bigr)$$
and likewise
$$\bigl(u-x_1y_2^{-1}\id_V\bigr)\bigl(u-x_2y_1^{-1}\id_V\bigr)=q(0)^{-1} u\,\bigl(v-(x_1y_1+x_2y_2)\,\id_V\bigr).$$
Hence,
$$G_{p,q}(u)=q(0)^{-2} u^2 \Theta_{p,q}(v).$$

\begin{Rem}\label{qconjugateremark}
Let $\calA$ be an $\F$-algebra, and let $a,b$ be elements of $\calA$ such that $p(a)=q(b)=0$.
Denote by $a^\star$ the $p$-conjugate of $a$ and by $b^\star$ the $q$-conjugate of $b$. Then, $b^\star=q(0)\,b^{-1}$ and $a^\star=p(0)\,a^{-1}$, whence
$$q(0)\,ab^{-1}+p(0)\,(ab^{-1})^{-1}=
q(0)\,ab^{-1}+p(0)\,ba^{-1}=ab^\star+ba^\star.$$
\end{Rem}

Our first basic result follows:

\begin{prop}\label{qseparEetR}
The endomorphism $u$ is a $(p,q)$-quotient if and only if
both its q-exceptional part and its q-regular part are $(p,q)$-quotients.
\end{prop}

The proof of this result will use the following basic lemma, which is a straightforward corollary to the 
Basic Commutation Lemma:

\begin{lemma}[Commutation Lemma]\label{qcommutelemma}
Let $p$ and $q$ be monic polynomials of $\F[t]$ with degree $2$ such $p(0)q(0) \neq 0$, and let $a$ and $b$ be endomorphisms of a vector space $V$
such that $p(a)=q(b)=0$. Then, both $a$ and $b$ commute with $q(0)\,(ab^{-1})+p(0)\,(ab^{-1})^{-1}$.
\end{lemma}

Now, we are ready to prove Proposition \ref{qseparEetR}.

\begin{proof}[Proof of Proposition \ref{qseparEetR}]
The ``if" part is obvious. Conversely, assume that $u$ is a $(p,q)$-quotient,
and split $u=ab^{-1}$ where $a$ and $b$ are automorphisms of $V$ such that $p(a)= 0$ and $q(b)=0$.
By the Commutation Lemma, both $a$ and $b$ commute with $v:=q(0) u+p(0) u^{-1}$.
Hence, $a$ and $b$ commute with $\Theta_{p,q}(v)$.
Since $u$ commutes with $v$ and is an automorphism, we see that
$G_{p,q}(u)^n=q(0)^{-2n} u^{2n} \Theta_{p,q}(v)^n=q(0)^{-2n} \Theta_{p,q}(v)^n u^{2n}$ for every positive integer $n$,
and it follows that $\Ker G_{p,q}(u)^n=\Ker \Theta_{p,q}(v)^n$ and $\im G_{p,q}(u)^n=\im \Theta_{p,q}(v)^n$ for every such integer $n$.
Hence, as $a$ and $b$ commute with $v$ we deduce that both stabilize $E_{p,q}(u)$ and $R_{p,q}(u)$
(and of course they induce automorphisms of those vector spaces).
Denote by $a'$ and $b'$ (respectively, by $a''$ and $b''$) the automorphisms of $E_{p,q}(u)$ (respectively, of $R_{p,q}(u)$)
induced by $a$ and $b$. Then, the q-exceptional part of $u$ is $a'(b')^{-1}$, and the q-regular part of $u$ is $a''(b'')^{-1}$.
As $p$ annihilates $a'$ and $a''$, and $q$ annihilates $b'$ and $b''$, both the q-exceptional and the q-regular part of $u$ are $(p,q)$-quotients.
\end{proof}

From there, it is clear that classifying $(p,q)$-quotients amounts to classifying the q-exceptional ones and the
q-regular ones. The easier classification is the latter: as we shall see, it involves little discussion on the specific polynomials
$p$ and $q$ under consideration (whether they are split or not over $\F$, separable or not over $\F$, etc).
In contrast, the classification of q-exceptional $(p,q)$-quotients involves a tedious case-by-case study.

\subsection{The $R_\delta$ transformation}\label{RdeltaSection}

\begin{Not}
Let $r$ be a monic polynomial with degree $d$, and let $\delta$ be a nonzero scalar.
We set $$R_\delta(r):=t^d r(t+\delta t^{-1}),$$
which is a monic polynomial with degree $2d$ and valuation $0$. \\
\end{Not}

Some basic facts will be useful on the $R_\delta$ transformation:

\begin{itemize}
\item For all monic polynomials $r$ and $s$, we have $R_\delta(r)R_\delta(s)=R_\delta(rs)$.
\item Followingly, if $r$ and $s$ are monic polynomials such that $r$ divides $s$, then $R_\delta(r)$ divides $R_\delta(s)$.
\item Let $r$ and $s$ be coprime monic polynomials. Then, $R_\delta(r)$ and $R_\delta(s)$ are coprime: indeed
if in some algebraic (field) extension of $\F$ those polynomials had a common root $z$ (necessarily nonzero)
then $z+\delta z^{-1}$ would be a common root of $r$ and $s$.
\end{itemize}

\subsection{Statement of the results}\label{resultssection}

We are now ready to state our results.
We shall frame them in terms of direct-sum decomposability.

Let $u$ be an endomorphism of a nonzero finite-dimensional vector space $V$.
Assume that $V$ splits into $V_1 \oplus V_2$, and that each linear subspace $V_1$ and $V_2$ is stable under $u$ and nonzero,
and both induced endomorphisms $u_{|V_1}$ and $u_{|V_2}$ are $(p,q)$-quotients.
Then, $u$ is obviously a $(p,q)$-quotient.
In the event when such a decomposition exists we shall say that $u$ is
a \textbf{decomposable} $(p,q)$-quotient, otherwise and if $u$ is a $(p,q)$-quotient, we shall say that $u$ is
an \textbf{indecomposable} $(p,q)$-quotient. Obviously, every $(p,q)$-quotient in $\End(V)$
is the direct sum of indecomposable ones. Hence, it suffices to describe the indecomposable $(p,q)$-quotients.

Moreover, if a $(p,q)$-quotient is indecomposable then it is either q-regular or q-exceptional, owing to Proposition \ref{qseparEetR}.

In each one of the following tables, we give a set of matrices. Each matrix represents an indecomposable $(p,q)$-quotient,
and every indecomposable $(p,q)$-quotient in $\End(V)$ is represented by one of those matrices, in some basis.
It is convenient to set
$$\delta:=p(0)q(0)^{-1}.$$

We start with q-regular $(p,q)$-quotients. In that situation the classification is rather simple:

\begin{table}[H]
\begin{center}
\caption{The classification of indecomposable q-regular $(p,q)$-quotients.}
\label{qfigure1}
\begin{tabular}{| c | c |}
\hline
Representing matrix & Associated data  \\
\hline
\hline
 & $n\in \N^*$, $r \in \F[t]$ irreducible and monic, \\
$C\bigl(R_\delta(r)^n\bigr)$ &  $R_\delta(r)$ has no root in $\Root(p)\Root(q)^{-1}$ \\
& $N_{\calW(p,q,q(0)y)_\L}$ is isotropic over $\L:=\F[t]/(r)$ \\
& for some root $y$ of $r$ in $\L$ \\
\hline
$C\bigl(R_\delta(r)^n\bigr)$ & $n\in \N^*$, $r \in \F[t]$ irreducible and monic, \\
$\oplus$ & $R_\delta(r)$ has no root in $\Root(p)\Root(q)^{-1}$ \\
$C\bigl(R_\delta(r)^n\bigr)$ & $N_{\calW(p,q,q(0)y)_\L}$ is non-isotropic over $\L:=\F[t]/(r)$ \\
& for some root $y$ of $r$ in $\L$ \\
\hline
\end{tabular}
\end{center}
\end{table}

Remember (see Remark \ref{isotropicnormremark}), that the norm of $\calW(p,q,x)$
is isotropic whenever one of $p$ and $q$ splits in $\F[t]$.

\vskip 3mm
Next, we tackle the q-exceptional indecomposable $(p,q)$-quotients. Here, there are many cases to consider.
We start with the one when both $p$ and $q$ are split.

\begin{table}[H]
\begin{center}
\caption{The classification of indecomposable q-exceptional $(p,q)$-quotients: When both $p$ and $q$ are split with a double root.}
\label{qfigure2}
\begin{tabular}{| c | c |}
\hline
Representing matrix & Associated data  \\
\hline
\hline
$C\bigl((t-x)^{n}\bigr)$ & $x \in \Root(p)\Root(q)^{-1}$, $n \in \N^*$ \\
\hline
\end{tabular}
\end{center}
\end{table}

\begin{table}[H]
\begin{center}
\caption{The classification of indecomposable q-exceptional $(p,q)$-quotients: When both $p$ and $q$ are split with simple roots.}
\label{qfigure3}
\begin{tabular}{| c | c |}
\hline
Representing matrix & Associated data  \\
\hline
\hline
 & $n \in \N^*$, \\
$C\bigl((t-x)^n\bigr) \oplus C\bigl((t-\delta x^{-1})^n\bigr)$ & $x \in \Root(p)\Root(q)^{-1}$  \\
& such that $x \neq \delta x^{-1}$ \\
\hline
 & $n \in \N$, \\
$C\bigl((t-x)^{n+1}\bigr) \oplus C\bigl((t-\delta x^{-1})^{n}\bigr)$ & $x \in \Root(p)\Root(q)^{-1}$  \\
& such that $x \neq \delta x^{-1}$ \\
\hline
 & $n\in \N^*$,  \\
$C\bigl((t-x)^{n}\bigr)$ & $x \in \Root(p)\Root(q)^{-1}$ \\
& such that $x=\delta x^{-1}$ \\
\hline
\end{tabular}
\end{center}
\end{table}

\begin{table}[H]
\begin{center}
\caption{The classification of indecomposable q-exceptional $(p,q)$-quotients: When both $p$ and $q$ are split, $p$ has simple roots
 and $q$ has a double root.}
 \label{qfigure4}
\begin{tabular}{| c | c |}
\hline
Representing matrix & Associated data  \\
\hline
\hline
$C\bigl((t-x)^n\bigr) \oplus C\bigl((t-\delta x^{-1})^n\bigr)$ & $x \in \Root(p)\Root(q)^{-1}$, $n \in \N^*$ \\
\hline
$C\bigl((t-x)^{n+1}\bigr) \oplus C\bigl((t-\delta x^{-1})^{n}\bigr)$ & $x \in \Root(p)\Root(q)^{-1}$, $n \in \N$ \\
\hline
$C\bigl((t-x)^{n+2}\bigr) \oplus C\bigl((t-\delta x^{-1})^{n}\bigr)$ & $x \in \Root(p)\Root(q)^{-1}$, $n \in \N$ \\
\hline
\end{tabular}
\end{center}
\end{table}

We now turn to the case when $p$ is irreducible but $q$ splits.

There are two subcases to consider, whether the two polynomials deduced from $p$ by using the homotheties
with ratio among the roots of $q$ are equal or not.

\begin{table}[H]
\begin{center}
\caption{The classification of indecomposable q-exceptional $(p,q)$-quotients: When $p$ is irreducible, $q=(t-y_1)(t-y_2)$
for some $y_1,y_2$ in $\F$, and $H_{y_1}(p)=H_{y_2}(p)$.}
\label{qfigure5}
\begin{tabular}{| c | c |}
\hline
Representing matrix & Associated data  \\
\hline
\hline
$C\bigl(H_y(p)^n\bigr)$ &  $n\in \N^*$, $y \in \Root(q)$  \\
\hline
\end{tabular}
 \end{center}
\end{table}

\begin{table}[H]
\begin{center}
\caption{The classification of indecomposable q-exceptional $(p,q)$-quotients: When $p$ is irreducible, $q=(t-y_1)(t-y_2)$
for some $y_1,y_2$ in $\F$, and $H_{y_1}(p)\neq H_{y_2}(p)$.}
\label{qfigure6}
\begin{tabular}{| c | c |}
\hline
Representing matrix & Associated data  \\
\hline
\hline
$C\bigl(H_{y_1}(p)^n\bigr) \oplus C\bigl(H_{y_2}(p)^n\bigr)$ & $n\in \N^*$ \\
\hline
$C\bigl(H_{y_1}(p)^{n+1}\bigr) \oplus C\bigl(H_{y_2}(p)^{n}\bigr)$ & $n\in \N$ \\
\hline
$C\bigl(H_{y_1}(p)^{n+2}\bigr) \oplus C\bigl(H_{y_2}(p)^{n}\bigr)$ & $n\in \N$ \\
\hline
\end{tabular}
 \end{center}
\end{table}

Next, we consider the situation where both $p$ and $q$ are irreducible in $\F[t]$, with the same
splitting field.

\begin{table}[H]
\begin{center}
\caption{The classification of indecomposable q-exceptional $(p,q)$-quotients: When $p$ and $q$ are irreducible with the same splitting field $\L$.}
\label{qfigure7}
\begin{tabular}{| c | c |}
\hline
Representing matrix & Associated data  \\
\hline
\hline
$C\Bigl(\bigl(t^2-\Tr_{\L/\F}(xy^{-1})t+\delta\bigr)^{n}\Bigr)$
&  $n\in \N^*$,  \\
$\oplus$ &  $x \in \Root(p)$, $y \in \Root(q)$ \\
$C\Bigl(\bigl(t^2-\Tr_{\L/\F}(xy^{-1})t+\delta\bigr)^{n}\Bigr)$ & with $xy^{-1} \not\in \F$  \\
\hline
$C\Bigl(\bigl(t^2-\Tr_{\L/\F}(xy^{-1})t+\delta\bigr)^{n+1}\Bigr)$
&  $n\in \N$,  \\
$\oplus$ &  $x \in \Root(p)$, $y \in \Root(q)$ \\
$C\Bigl(\bigl(t^2-\Tr_{\L/\F}(xy^{-1})t+\delta\bigr)^{n}\Bigr)$ & with $xy^{-1} \not\in \F$  \\
\hline
 & $n \in \N^*$, \\
$C\bigl((t-xy^{-1})^n\bigr) \oplus C\bigl((t-xy^{-1})^n\bigr)$ &  $x \in \Root(p)$, $y \in \Root(q)$ \\
& with $xy^{-1} \in \F$ \\
\hline
\end{tabular}
 \end{center}
\end{table}

We finish with the case when $p$ and $q$ are both irreducible, with distinct splitting fields.
There are two subcases to consider, whether both $\tr p$ and $\tr q$ equal zero or not.

\begin{table}[H]
\begin{center}
\caption{The classification of indecomposable q-exceptional $(p,q)$-quotients: When $p$ and $q$ are irreducible with distinct splitting fields and
$(\tr p,\tr q) \neq (0,0)$.}
\label{qfigure8}
\begin{tabular}{| c | c |}
\hline
Representing matrix & Associated data  \\
\hline
\hline
$C(G_{p,q}^n) \oplus C(G_{p,q}^n)$ & $n \in \N^*$ \\
\hline
$C(G_{p,q}^{n+1}) \oplus C(G_{p,q}^{n})$ & $n \in \N$ \\
\hline
\end{tabular}
 \end{center}
\end{table}

\begin{table}[H]
\begin{center}
\caption{The classification of indecomposable q-exceptional $(p,q)$-quotients: When $p$ and $q$ are irreducible with distinct splitting fields and
$\tr p=\tr q=0$.}
\label{qfigure9}
\begin{tabular}{| c | c |}
\hline
Representing matrix & Associated data  \\
\hline
\hline
$C\bigl((t^2-\delta)^n\bigr)$ & \\
$\oplus$ & $n \in \N^*$ \\
$C\bigl((t^2-\delta)^n\bigr)$ &  \\
\hline
\end{tabular}
 \end{center}
\end{table}

\subsection{An example}

Here, we consider the case when $\F$ is the field $\R$ of real numbers and $p=q=t^2+1$.
In other words, we determine the automorphisms of a finite-dimensional real vector space $V$ that split
into $ab^{-1}$ for some automorphisms $a$ and $b$ such that $a^2=b^2=-\id_V$
(note that these automorphisms are also the $(t^2+1,t^2+1)$-products).
Here, $\Root(p)\Root(q)^{-1}=\{1,-1\}$ and $\delta:=p(0)q(0)^{-1}$ equals $1$.

Let us investigate the indecomposable $(p,q)$-quotients.
Let $r \in \R[t]$ be an irreducible monic polynomial. The fraction $r(t+\delta t^{-1})=r(t+t^{-1})$ has no root in $\Root(p)\Root(q)^{-1}$
if and only if $r(2) \neq 0$ and $r(-2) \neq 0$.

From now on, we assume that $r \neq t-2$ and $r \neq t+2$.
We set $\L:=\R[t]/(r)$ and we denote by $\overline{t}$ the class of $t$ in it.
If $r$ has degree $2$, then $\L$ is isomorphic to $\C$, which is algebraically closed, and it follows that
the norm of $\calW(p,q,\overline{t})_{\L}$ is isotropic.
Note that, for all $(\alpha,\beta)\in \R^2$ such that $\alpha^2<4\beta$,
we have
$$R_\delta(t^2+\alpha t+\beta)=t^2\bigl((t+t^{-1})^2+\alpha (t+t^{-1})+\beta\bigr)=t^4+\alpha t^3+(\beta+2)t^2+\alpha t+1.$$

Assume now that $r$ has degree $1$, and denote by $x$ its root (so that $x \neq \pm 2$).
The norm of $\calW(p,q,x)_{\R}$ reads
$$a I_4+b A+c B+dC \mapsto a^2+b^2+c^2+d^2+x bc-x ad$$
which is equivalent to the orthogonal direct sum of two copies of the quadratic form
$$Q : (a,b) \mapsto a^2+x ab+b^2.$$
We have $Q(1,0)>0$, and the discriminant of $Q$ equals $\frac{x^2-4}{4}$. Hence,
either $|x|<2$ and $Q$ is positive definite,
or $|x|>2$ and $Q$ is isotropic.
It follows that if $x \in (-2,2)$, then the norm of $\calW(p,q,x)_{\R}$ is non-isotropic, otherwise it is isotropic.

The following table thus gives a complete list of indecomposable $(t^2+1,t^2+1)$-quotients, where
the q-exceptional ones -- given in the last two rows -- are obtained thanks to Table \ref{qfigure7}:

\begin{table}[ht]
\caption{The classification of indecomposable $(t^2+1,t^2+1)$-quotients over $\R$.}
\begin{center}
\begin{tabular}{| c | c |}
\hline
Representing matrix & Associated data  \\
\hline
\hline
$C\bigl((t^2-xt+1)^n\bigr) \oplus C\bigl((t^2-xt+1)^n\bigr)$ & $n\in \N^*$, $x \in \left(-2,2\right)$ \\
\hline
$C\bigl((t^2-xt+1)^n\bigr)$ & $n\in \N^*$, $x \in \left(-\infty,-2\right) \cup \left(2,+\infty\right)$ \\
\hline
$C\bigl((t^4+\alpha t^3+(\beta+2)t^2+\alpha t+1)^n\bigr)$ & $n \in \N^*$, $(\alpha,\beta)\in \R^2$ with $\alpha^2<4\beta$ \\
\hline
$C((t-1)^n) \oplus C((t-1)^n)$ & $n \in \N^*$ \\
\hline
$C((t+1)^n) \oplus C((t+1)^n)$ & $n \in \N^*$  \\
\hline
\end{tabular}
 \end{center}
\end{table}

\subsection{Strategy, and structure of the remainder of the section}

As in the previous section, the study has two very distinct parts, just like in the study of $(p,q)$-differences. 
The structure of the algebra $\calW(p,q,x)_R$ is used to obtain the classification of regular $(p,q)$-quotients
(see the next section).   The study of q-exceptional $(p,q)$-quotients is carried out in the last four sections:
\begin{itemize}
\item In Section \ref{q-exceptionalsectionI}, we consider the case when $p$ and $q$ are both split
(contrary to the case of $(p,q)$-sums, this problem was still partly open to this day, even over algebraically closed fields).
\item In Section \ref{q-exceptionalsectionII}, we tackle the case when $p$ is irreducible and $q$ is split.
\item In Section \ref{q-exceptionalsectionIII}, we tackle the case when both $p$ and $q$ are irreducible with the same splitting field
in $\overline{\F}$.
\item Finally, in Section  \ref{q-exceptionalsectionIV}, we complete the study by considering the case when $p$ and $q$ are both irreducible
with distinct splitting fields in $\overline{\F}$.
\end{itemize}
In all those sections, several subcases need to be considered. 
In several instances, we will take advantage of some technical results that have already been proved in the
study of $(p,q)$-differences.

\subsection{Regular $(p,q)$-quotients}\label{q-regularsection}

\subsubsection{The initial reduction}

\begin{prop}\label{basicq-regularreduction}
Let $p$ and $q$ be monic polynomials with degree $2$ over $\F$ such that $p(0)q(0) \neq 0$, and set $\delta:=p(0)q(0)^{-1}$.
Let $u$ be an endomorphism of a finite-dimensional vector space $V$ and assume that $u$ is a q-regular $(p,q)$-quotient
Then:
\begin{enumerate}[(a)]
\item Each invariant factor of $u$ has the form $R_\delta(r)$ for some monic polynomial~$r$.
\item In some basis of $V$, the endomorphism $u$ is represented by
a block-diagonal matrix in which every diagonal block has the form $R_\delta(r)^n$ for some irreducible monic polynomial~$r$
and some positive integer $n$. We call such a matrix a \textbf{$(p,q)$-reduced canonical form} of $u$.
\end{enumerate}
\end{prop}

It is easily seen that a $(p,q)$-reduced canonical form is unique up to a permutation of the diagonal blocks.

Before we prove Proposition \ref{basicq-regularreduction}, we need the corresponding special case when both polynomials $p$ and $q$ are split over $\F$:
this result will be obtained by following a similar method as for the study of $(p,q)$-differences.

\begin{prop}\label{basicq-regularinvariants}
Let $p$ and $q$ be split monic polynomials with degree $2$ over $\F$ such that $p(0)q(0) \neq 0$, and set $\delta:=p(0)q(0)^{-1}$.
Let $u$ be an endomorphism of a finite-dimensional vector space $V$ and assume that $u$ is a q-regular $(p,q)$-quotient.
Then, each invariant factor of $u$ has the form $R_\delta(r)$ for some monic polynomial~$r$.
\end{prop}

The proof requires the following basic lemma, which will be reused later in this article:

\begin{lemma}\label{qblockmatrixlemma}
Let $r\in \F[t]$ be a monic polynomial with degree $n>0$, and $\delta$ be a nonzero scalar.
Then,
$$\begin{bmatrix}
0_n & -\delta I_n \\
I_n & C(r)
\end{bmatrix}\simeq C\bigl(R_\delta(r)\bigr).$$
\end{lemma}

Before we give the proof, we recall some known results on palindromials. Let $\delta \in \F \setminus \{0\}$.
Given a non-negative integer $m$, a $(2m,\delta)$-\textbf{palindromial} is a polynomial
$R(t)=\underset{k=0}{\overset{2m}{\sum}} a_k t^k$ in $\F_{2m}[t]$ such that $R(t)=t^{2m} \delta^{-m} R(\delta/t)$ or, in other words,
$a_{2m-k}=\delta^{k-m} a_k$ for all $k \in \lcro 0,2m\rcro$.
The $(2m,\delta)$-palindromials obviously constitute a linear subspace $P_{2m,\delta}(\F)$ of $\F_{2m}[t]$
with dimension $m+1$, and the mapping
$$U \mapsto t^m U(t+\delta t^{-1})$$
is a linear injection from $\F_m[t]$ into $P_{2m,\delta}(\F)$. Hence, because of the dimension of the source and target spaces we get
that this map is a linear isomorphism.

Finally, given a positive integer $m$, every polynomial $R \in \F_{2m}[t]$ splits (uniquely)
into $U+V$ for some $(2m,\delta)$-palindromial $U$ and some $(2m-2,\delta)$-palindromial $V$.
Indeed:
\begin{itemize}
\item We see that $\dim P_{2m,\delta}(\F)+\dim P_{2m-2,\delta}(\F)=(m+1)+m=\dim \F_{2m}[t]$.
\item On the other hand we have $P_{2m,\delta}(\F) \cap P_{2m-2,\delta}(\F)=\{0\}$: indeed,
if $\underset{k=0}{\overset{2m}{\sum}} a_k t^k$ is both a $(2m,\delta)$-palindromial and a $(2m-2,\delta)$-palindromial, then,
with the convention that $a_k=0$ for every integer $k \in \Z \setminus \{0,\dots,2m\}$ we see that
$a_k=\delta^{m-k} a_{2m-k}$ and $a_k=\delta^{m-1-k} a_{2(m-1)-k}$ for all $k \in \Z$, which shows that
$a_k=\delta^{-1}a_{k-2}$ for all $k \in \Z$. Since $(a_k)_{k \in \Z}$ ultimately vanishes, we deduce that  $a_k=0$ for all $k \in \Z$.
\end{itemize}

It follows that every polynomial of $\F_{2m}[t]$ has a (unique) splitting into
$$t^m P(t+\delta t^{-1})+t^{m-1} Q(t+\delta t^{-1})$$
for some polynomials $P$ and $Q$ with $\deg P \leq m$ and $\deg Q \leq m-1$.

With this result in mind, we can now prove the above lemma.

\begin{proof}[Proof of Lemma \ref{qblockmatrixlemma}]
Set $$N:=\begin{bmatrix}
0_n & -\delta I_n \\
I_n & C(r)
\end{bmatrix}.$$
By a straightforward computation, one checks that
$N$ is invertible and that
$$\delta N^{-1}=\begin{bmatrix}
C(r) & \delta I_n \\
-I_n & 0_n
\end{bmatrix},$$
whence
$$N+\delta N^{-1}=\begin{bmatrix}
C(r) & 0_n \\
0_n & C(r)
\end{bmatrix}.$$
Hence, for every polynomial $P \in \F[t]$, we see that
$$P(N+\delta N^{-1})=\begin{bmatrix}
P\bigl(C(r)\bigr) & 0_n \\
0_n & P\bigl(C(r)\bigr)
\end{bmatrix}$$
and
$$N P(N+\delta N^{-1})=\begin{bmatrix}
0_n & -\delta P\bigl(C(r)\bigr) \\
P\bigl(C(r)\bigr) & C(r)P\bigl(C(r)\bigr)
\end{bmatrix}.$$
In particular, $R_\delta(r)$
annihilates $N$; note that this polynomial is monic with degree $2n$.

Let $u(t) \in \F[t]$ annihilate $N$ with $\deg u(t) <2n$. Then, as we have seen before the start of the proof,
$u(t)=t^n P(t+\delta t^{-1})+t^{n-1}Q(t+\delta t^{-1})$ for some pair $(P,Q)$ of polynomials such that $\deg P \leq n$ and $\deg Q \leq n-1$.
Since $\deg u(t)<2n$ we must have $\deg P<n$.
Since $N$ is invertible, we have
$$N P(N+\delta N^{-1})+Q(N+\delta N^{-1})=0.$$
By looking at the upper-left and lower-left $n$-by-$n$ blocks in this identity, we get $P\bigl(C(r)\bigr)=0$ and
$Q\bigl(C(r)\bigr)=0$, and hence $r$ divides $P$ and $Q$. Since $\deg P<n$ and $\deg Q<n$ we obtain $P=0=Q$, whence $u=0$.

We conclude that $R_\delta(r)$ is the minimal polynomial of $N$, and since this polynomial is of degree $2n$ we conclude that
$N$ is similar to the companion matrix of $R_\delta(r)$.
\end{proof}

\begin{cor}\label{qcorblockinvariants}
Let $N$ be an arbitrary matrix of $\Mat_n(\F)$, and $\delta$ be a nonzero scalar.
Denote by $r_1,\dots,r_a$ the invariant factors of $N$.
Then, the invariant factors of
$$K(N):=\begin{bmatrix}
0_n & -\delta I_n \\
I_n & N
\end{bmatrix}$$
are $R_\delta(r_1),\dots,R_\delta(r_a)$.
\end{cor}

\begin{proof}
We note that the similarity class of $K(N)$ depends only on that of $N$: Indeed, for all
$P \in \GL_n(\K)$, the invertible matrix $Q:=P \oplus P$ satisfies
$Q K(N)Q^{-1}=K(PNP^{-1})$.
Hence,
$$K(N) \simeq K\bigl(C(r_1) \oplus \cdots \oplus C(r_a)\bigr).$$
By permuting the basis vectors, we find
$$K\bigl(C(r_1) \oplus \cdots \oplus C(r_a)\bigr) \simeq K\bigl(C(r_1)\bigr) \oplus \cdots \oplus K\bigl(C(r_a)\bigr).$$
Hence, by Lemma \ref{qblockmatrixlemma}, we conclude that
$$K(N) \simeq C\bigl(R_\delta(r_1)\bigr) \oplus \cdots \oplus C\bigl(R_\delta(r_a)\bigr).$$

Finally, by the results of Section \ref{RdeltaSection}, we see that $R_\delta(r_{i+1})$ divides $R_\delta(r_i)$ for all $i \in \lcro 1,a-1\rcro$.
Therefore, the monic polynomials $R_\delta(r_1),\dots,R_\delta(r_a)$ are the invariant factors of $K(N)$.
\end{proof}

\begin{proof}[Proof of Proposition \ref{basicq-regularinvariants}]
Let $a$ and $b$ be automorphisms of $V$ such that $p(a)=q(b)=0$ and $u=ab^{-1}$.
Denote by $x$ (respectively, by $y$) an eigenvalue of $a$ (respectively, of $b$)
with maximal geometric multiplicity, and split $p(t)=(t-x)(t-x')$ and $q(t)=(t-y)(t-y')$.

We claim that
$$\dim \Ker(a-x\,\id_V) \geq \frac{n}{2}\cdot$$
Indeed, since $p(a)=0$ we have $\im(a- x'\, \id_V) \subset \Ker(a-x\,\id_V)$, which yields
$\dim \Ker(a-x\,\id_V)+\dim \Ker(a-x'\,\id_V) \geq n$. Since $\dim \Ker(a-x\,\id_V) \geq \dim \Ker(a-x'\,\id_V)$,
the claimed inequality follows. Likewise, $\dim \Ker(b-y\,\id_V) \geq \frac{n}{2}\cdot$

Since $u$ is $(p,q)$-q-regular, any eigenspace of $a$ is linearly disjoint from any eigenspace of $b$: indeed
if we had a common eigenvector of $a$ and $b$, with corresponding eigenvalues $x_i$ and $y_j$, then this vector would be an eigenvector of $u$ with corresponding eigenvalue $x_i y_j^{-1}$, thereby contradicting the assumption that
$u$ has no eigenvalue in $\Root(p)\Root(q)^{-1}$.
In particular, $\Ker(a-x \id_V) \cap \Ker(b-y\,\id_V)=\{0\}$.
It follows that $\dim \Ker(a-x\,\id_V)=\frac{n}{2}=\dim \Ker(b-y\,\id_V)$, that $n$ is even and that
$V=\Ker(a-x \id_V) \oplus \Ker(b-y\,\id_V)$.
Next, we deduce that $\frac{n}{2}=\dim \im (a-x \id_V)$ and $\dim \Ker(a-x' \id_V) \leq \frac{n}{2}$ by choice of $x$.
However, $\im(a-x\id_V) \subset \Ker(a-x' \id_V)$, and hence it follows that
$\im(a-x\id_V) = \Ker(a-x' \id_V)$. Likewise, $\im(b-y\id_V) = \Ker(b-y' \id_V)$,
and it follows that $x'$ has geometric multiplicity $\frac{n}{2}$ with respect to $a$,
and ditto for $y'$ with respect to $b$.
In turn, this shows that $\im(a-x'\id_V)=\Ker(a-x\id_V)$ and $\im(b-y'\id_V)=\Ker(b-y\id_V)$,
and any eigenspace of $a$ is a complementary subspace of any eigenspace of $b$.

Let us write $s:=\frac{n}{2}$ and choose a basis $(e_1,\dots,e_s)$ of
$\Ker(b-y \id_V)$. Then, we have $V=\Ker(b-y \id_V) \oplus \Ker(a-x \id_V)$, whence
$(e_{s+1},\dots,e_n):=\bigl(y^{-1}(a-x\id_V)(e_1),\dots,y^{-1}(a-x\id_V)(e_s)\bigr)$ is a basis of
$\im(a-x \id_V)=\Ker (a-x' \id_V)$.
Since $\Ker(b-y \id_V) \oplus \Ker (a-x' \id_V)=V$, we deduce that
$\bfB:=(e_1,\dots,e_n)$ is a basis of $V$. Obviously
$$\Mat_\bfB(a)=\begin{bmatrix}
x I_s & 0 \\
y I_s & x' I_s
\end{bmatrix}.$$
On the other hand, since $\Ker(b-y \id_V)=\im(b-y' \id_V)$, we find
$$\Mat_\bfB(b^{-1})=\begin{bmatrix}
y^{-1} I_s & N \\
0 & (y')^{-1} I_s
\end{bmatrix} \quad \text{for some matrix $N \in \Mat_s(\F)$.}$$
Hence,
$$\Mat_\bfB(u)=\begin{bmatrix}
xy^{-1} I_s & xN  \\
I_s & y N+x' (y')^{-1} I_s
\end{bmatrix}.$$
Setting
$$P:=\begin{bmatrix}
I_s & -xy^{-1} I_s \\
0 & I_s
\end{bmatrix},$$
we obtain
$$P \Mat_\bfB(u) P^{-1}=
\begin{bmatrix}
0 & -(xx')(yy')^{-1} I_s  \\
I_s & N'
\end{bmatrix}$$
for some matrix $N' \in \Mat_s(\F)$. Since $\delta=p(0)q(0)^{-1}=(xx')(yy')^{-1}$, the claimed result is then
readily deduced from Corollary \ref{qcorblockinvariants}.
\end{proof}

\begin{proof}[Proof of Proposition \ref{basicq-regularreduction}]
We start with point (a). Let us extend the scalar field to $\overline{\F}$.
The corresponding extension $\overline{u}$ of $u$ is still a $(p,q)$-quotient.
Hence, by Corollary \ref{qcorblockinvariants} its invariant factors are $R_\delta(p_1),\dots,
R_\delta(p_r)$ for some monic polynomials $p_1,\dots,p_r$ of $\overline{\F}[t]$ such that $p_{i+1}$ divides $p_i$ for all $i \in \lcro 1,r-1\rcro$.

Yet, the invariant factors of $\overline{u}$ are known to be the ones of $u$.
Finally, given a monic polynomial $h \in \overline{\F}[t]$ (with degree $N$) such that $R_\delta(h) \in \F[t]$,
we obtain by downward induction that all the coefficients of $h$ belong to $\F$:
Indeed, if we write $h(t)=t^N-\underset{i=0}{\overset{N-1}{\sum}} \alpha_i\,t^i$ and we know that $\alpha_{N-1},\dots,\alpha_{k+1}$ all belong to $\F$ for some $k \in \lcro 0,N-1\rcro$, then $\underset{i=0}{\overset{k}{\sum}} \alpha_i\,t^N (t+\delta t^{-1})^i=t^N(t+\delta t^{-1})^N -\underset{i=k+1}{\overset{N-1}{\sum}} \alpha_i\,t^N (t+\delta t^{-1})^i$
belongs to $\F[t]$, and by considering the coefficient on $t^{N+k}$, we gather that $\alpha_k \in \F$.
It follows that $p_1,\dots,p_r$ all belong to $\F[t]$, which completes the proof of statement (a).

From point (a), we easily derive point (b): indeed, consider an invariant factor
$R_\delta(r)$ of $u$ with some monic polynomial $r \in \F[t]$. Then, we split $r=r_1^{n_1}\cdots r_k^{n_k}$ where $r_1,\dots,r_k$ are pairwise distinct
irreducible monic polynomials of $\F[t]$, and $n_1,\dots,n_k$ are positive integers.
By a previous remark (see Section \ref{RdeltaSection}), the monic polynomials $R_\delta(r_1^{n_1}),\dots,R_\delta(r_k^{n_k})$
are pairwise coprime and their product equals $R_\delta(r)$, whence
$$C\bigl(R_\delta(r)\bigr) \simeq C\bigl(R_\delta(r_1^{n_1})\bigr) \oplus \cdots \oplus C\bigl(R_\delta(r_k^{n_k})\bigr).$$
Using point (a), we deduce that statement (b) holds true.
\end{proof}

\subsubsection{Application of the structural results on $\calW(p,q,x)$ to the characterization of q-regular $(p,q)$-quotients}

An additional definition will be useful:

\begin{Def}
Let $p$ and $q$ be monic polynomials with degree $2$ in $\F[t]$ such that $p(0)q(0) \neq 0$.
Set $\delta:=p(0)q(0)^{-1}$.
Let $r$ be an irreducible monic polynomial of $\F[t]$, and set
$\L:=\F[t]/(r)$. Denote by $x$ the class of $t$ in $\L$.
We say that $r$ has:
\begin{itemize}
\item \textbf{Type $1$ with respect to $(p,q)$} if $R_\delta(r)$ has no root in $\Root(p)\Root(q)^{-1}$ and
the norm of $\calW\bigl(p,q,q(0)x\bigr)_{\L}$ is isotropic.
\item \textbf{Type $2$ with respect to $(p,q)$} if $R_\delta(r)$ has no root in $\Root(p)\Root(q)^{-1}$ and
the norm of $\calW\bigl(p,q,q(0)x\bigr)_{\L}$ is non-isotropic.
\end{itemize}
\end{Def}

First of all, we use the structural results on $\calW(p,q,x)_R$ to obtain various $(p,q)$-quotients.
Our first result is actually not restricted to q-regular $(p,q)$-quotients and will be used later in our study.

\begin{lemma}[Duplication Lemma]\label{qduplicationlemma}
Let $p$ and $q$ be monic polynomials of $\F[t]$ with degree $2$ such that $p(0)q(0) \neq 0$, and set $\delta:=p(0)q(0)^{-1}$.
Let $r$ be a nonconstant monic polynomial of $\F[t]$.
Then, $C\bigl(R_\delta(r)\bigr) \oplus C\bigl(R_\delta(r)\bigr)$ is a $(p,q)$-quotient.
\end{lemma}

\begin{proof}
Denote by $d$ the degree of $r$.
We work with the commutative $\F$-algebra $R:=\F[C(r)]$, which is isomorphic to the quotient ring $\F[t]/(r)$,
and with the element $x:=C(r)$.
Then, we consider the endomorphisms $a : X \mapsto AX$ and $b:X \mapsto BX$ of $\calW(p,q,q(0)x)_R$.
Since $p(A)=0$ and $q(B)=0$, we get $p(a)=0$ and $q(b)=0$.
Moreover, since $B^{-1}=(\tr q) I_4-q(0) B$, we have
$AB^{-1}=(\tr q)A-q(0)\,AB$ and as $q(0) \neq 0$ we deduce that
$(B,A,I_4,AB^{-1})$ is a basis $\mathbf{B}$ of the free $R$-module $\calW(p,q,q(0)x)_R$.

Denote by $A'$ and $B'$ the respective matrices of $a$ and $b$ in $\bfB$.
Using identity \eqref{basicrel2}, we find
$(AB^{-1})^2=-\delta I_4+x \,(AB^{-1})$.
It easily follows that
$$A'(B')^{-1}=\begin{bmatrix}
0 & -\delta.1_R & 0 & 0  \\
1_R & x & 0 & 0 \\
0 & 0 & 0 & -\delta.1_R \\
0 & 0 & 1_R & x
\end{bmatrix}.$$
Therefore, $A'(B')^{-1}$, seen as a matrix of $\Mat_{4d}(\F)$,
is similar to $C\bigl(R_\delta(r)\bigr) \oplus C\bigl(R_\delta(r)\bigr)$ by Lemma \ref{qblockmatrixlemma}.
Since $p(A')=0$ and $q(B')=0$, the conclusion follows.
\end{proof}

Our next result deals with certain companion matrices that are associated with irreducible polynomials with Type 1 with respect to $(p,q)$.

\begin{lemma}\label{qType1blocklemma}
Let $p$ and $q$ be monic polynomials of $\F[t]$ with degree $2$ such that $p(0)q(0)\neq 0$, and set $\delta:=p(0)q(0)^{-1}$.
Let $r$ be an irreducible monic polynomial of $\F[t]$ of Type 1 with respect to $(p,q)$.
Then, for all $n \in \N^*$, the companion matrix $C\bigl(R_\delta(r^n)\bigr)$ is a $(p,q)$-quotient.
\end{lemma}

\begin{proof}
We naturally identify $\L$ with the subalgebra $\F[C(r)]$ of $\Mat_d(\F)$, where $d$ denotes the degree of $r$.
Let $n \in \N^*$. Set $R:=\F[C(r^n)]$, seen as a subalgebra of $\Mat_{nd}(\F)$, and set $x:=C(r^n)$.
The $\F$-algebra $R$ is isomorphic to $\F[t]/(r^n)$.
By Proposition \ref{localstructureprop}, it follows that $\calW(p,q,q(0)x)_R$ is isomorphic to $\Mat_2(R)$.
We choose an isomorphism $\varphi : \calW(p,q,q(0)x)_R \overset{\simeq}{\longrightarrow} \Mat_2(R)$,
and we set $a:=\varphi(A)$ and $b:=\varphi(B)$. Note that $p(a)=q(b)=0$,
whereas $d:=ab^{-1}$ satisfies $q(0) d+p(0) d^{-1}=q(0)x\,I_2$, whence
$d^2=x d-\delta I_2$. Denote by $\L$ the residue class field of $R$.
The endomorphism $X \mapsto dX$ of $R^2$ induces an endomorphism
$\overline{d}$ of the $\L$-vector space $\L^2$. Since $(I_2,a,b,ab^{-1})$ is a basis of the $R$-module $\Mat_2(R)$,
the endomorphism $\overline{d}$ is not a scalar multiple of the identity of $\L^2$.
This yields a vector $e$ of $\L^2$ such that $\bigl(e,\overline{d}(e)\bigr)$ is a basis of $\L^2$.
Lifting $e$ to a vector $E$ of $R^2$, we deduce that $(E,dE)$ is a basis of the $R$-module $R^2$.
Hence, composing $\varphi$ with an additional interior automorphism of the $R$-algebra $\Mat_2(R)$,
we see that no generality is lost in assuming that the first column of $d$
reads $\begin{bmatrix}
0_R \\
1_R
\end{bmatrix}$. Then, the equality $d^2=x d-\delta I_2$ yields
$$d=\begin{bmatrix}
0 & -\delta 1_R \\
1_R & x
\end{bmatrix}.$$
It follows that the matrix
$\begin{bmatrix}
0 & -\delta I_{nd} \\
I_{nd} & C(r^n)
\end{bmatrix}$ of $\Mat_{2nd}(\F)$ is a $(p,q)$-quotient. 
By Lemma \ref{qblockmatrixlemma}, this matrix is similar to
$C\bigl(R_\delta(r^n)\bigr)$, which completes the proof.
\end{proof}

Combining Lemma \ref{qduplicationlemma} with Lemma \ref{qType1blocklemma}, we conclude that the implication (iii) $\Rightarrow$ (i)
in the following theorem holds true.

\begin{theo}[Classification of q-regular $(p,q)$-quotients]\label{q-regulartheo}
Let $p$ and $q$ be monic polynomials of degree $2$ in $\F[t]$ such that $p(0)q(0) \neq 0$.
Let $u$ be an endomorphism of a finite-dimensional vector space $V$ over
$\F$. Assume that $u$ is q-regular with respect to $(p,q)$ and set $\delta:=p(0)q(0)^{-1}$.
The following conditions are equivalent:
\begin{enumerate}[(i)]
\item The endomorphism $u$ is a $(p,q)$-quotient.
\item The invariant factors of $u$ read $R_\delta(p_1), \dots, R_\delta(p_{2n-1}),R_\delta(p_{2n}),\dots$
where, for every irreducible monic polynomial $r \in \F[t]$
that has Type $2$ with respect to $(p,q)$ and every positive integer $n$, the polynomials $p_{2n-1}$ and $p_{2n}$ have the
same valuation with respect to $r$.
\item There is a basis of $V$ in which $u$ is represented by a block-diagonal matrix where every diagonal block
equals either $C\bigl(R_\delta(r^n)\bigr)$ for some irreducible monic polynomial $r \in \F[t]$ of Type $1$ with respect to $(p,q)$
and some $n \in \N^*$,
or $C\bigl(R_\delta(r^n)\bigr) \oplus C\bigl(R_\delta(r^n)\bigr)$ for some irreducible monic polynomial $r \in \F[t]$ and some $n \in \N^*$.
\end{enumerate}
\end{theo}

Note that this result, combined with the observation that
$C\bigl(R_\delta(r^n)\bigr)$ is q-regular with respect to $(p,q)$ for every monic polynomial $r \in \F[t]$
such that $R_\delta(r)$ has no root in $\Root(p)\Root(q)^{-1}$, yields the classification of indecomposable q-regular $(p,q)$-quotients as given
in Table \ref{qfigure1}.
Moreover, by using the same method as in the last part of the proof of Proposition \ref{basicq-regularreduction},
it is easily seen that condition (ii) is equivalent to condition (iii).

In order to conclude on Theorem \ref{q-regulartheo}, it only remains to prove that condition (i) implies condition (ii),
which we shall now do thanks to the structural results on $\calW(p,q,x)_R$.

\begin{proof}[Proof of the implication (i) $\Rightarrow$ (ii)]
Let us assume that $u$ is a $(p,q)$-quotient.
Let $r$ be an irreducible monic polynomial of $\F[t]$ with Type 2 with respect to $(p,q)$, and let $n \in \N^*$.
All we need is to prove that, in the canonical form of $u$ from Proposition \ref{basicq-regularreduction},
the number $m$ of diagonal cells that equal $C\bigl(R_\delta(r)^n\bigr)$ is even.

Let us choose automorphisms $a$ and $b$ of $V$ such that $u=ab^{-1}$ and $p(a)=q(b)=0$.
By the Commutation Lemma (i.e.\ Lemma \ref{qcommutelemma}), we know that $a$ and $b$ commute with $v:=u+\delta u^{-1}$, and hence
all three endomorphisms $a,b,u$ yield endomorphisms $\overline{a}$, $\overline{b}$ and $\overline{u}$
of the vector space
$$E:=\Ker \bigl(r^n(v)\bigr)/\Ker \bigl(r^{n-1}(v)\bigr)
=\Ker(R_\delta(r^n)(u))/\Ker(R_\delta(r^{n-1})(u))$$
such that $\overline{u}=\overline{a}\overline{b}^{-1}$, and
$r$ annihilates $\overline{v}:=\overline{u}+\delta \overline{u}^{-1}$.
Again, $\overline{a}$ and $\overline{b}$ commute with $\overline{v}$, and hence they
are endomorphisms of the $\F[\overline{v}]$-module $E$.
Since $r$ is irreducible, we have $\F[\overline{v}] \simeq \F[t]/(r)$, and
$\L:=\F[\overline{v}]$ is a field. We shall write $y:=\overline{v}$, which we see as an element of $\L$.
Using the structure of $\L$-vector space, we can write $\overline{u}+\delta \overline{u}^{-1}=y\,\id_E$;
by combining this with $p(\overline{a})=q(\overline{b})=0$, we deduce that
$\overline{a}$ and $\overline{b}$ yield a representation of the $\L$-algebra $\calW\bigl(p,q,q(0)y\bigr)_{\L}$
on the $\L$-vector space $E$.

On the other hand, we see that $2m\deg(r)$ is the dimension of the $\F$-vector space $E$, and hence
$2m$ is the dimension of the $\L$-vector space $E$.

By Proposition \ref{structureofWprop}, the algebra
$\calW\bigl(p,q,q(0)y\bigr)_{\L}$ is a $4$-dimensional skew-field over $\L$, whence the $\L$-vector space
$E$ is isomorphic to a power of $\calW\bigl(p,q,q(0)y\bigr)_{\L}$ and it follows that its dimension is
a multiple of $4$. Therefore, $m$ is a multiple of $2$, which completes the proof.
\end{proof}

Therefore, we have completed the classification of q-regular $(p,q)$-quotients.

\subsection{Exceptional $(p,q)$-quotients (I): When both $p$ and $q$ are split}\label{q-exceptionalsectionI}

Here, we determine the q-exceptional $(p,q)$-quotients in the
case when both polynomials $p$ and $q$ are split. There are several scattered results in the literature but they
do not encompass all cases. Hence, we will start from scratch.

The following remark will be useful in what follows:
split $p(t)=(t-x_1)(t-x_2)$ and $q(t)=(t-y_1)(t-y_2)$ in $\F[t]$, and set
$z:=x_1y_1^{-1}$ and $z':=x_2 y_2^{-1}$.
Let $a,b$ be automorphisms of an $\F$-vector space $V$ such that $p(a)=q(b)=0$, and set $u:=ab^{-1}$.
Then, $zz'=p(0)q(0)^{-1}$, and hence
$$u+\delta u^{-1}=(u-z \id_V)(\id_V-z' u^{-1})+(z+z')\,\id_V.$$
Therefore, the Commutation Lemma shows that both $a$ and $b$ commute with $(u-z \id_V)(\id_V-z' u^{-1})$.

\subsubsection{Stating the results}

Now, we can restate what we wish to prove:

\begin{theo}\label{qtheoPandQsplitSoleRoot}
Let $p$ and $q$ be split monic polynomials with degree $2$ in $\F[t]$, each with a sole root, and with $p(0)q(0) \neq 0$.
Write $p(t)=(t-x)^2$ and $q(t)=(t-y)^2$.
Then, an endomorphism $u$ of a finite-dimensional vector space over $\F$  is a q-exceptional $(p,q)$-quotient if and only if
the characteristic polynomial of $u$ is a power of $t-xy^{-1}$.
\end{theo}

This result can be easily deduced from Botha's characterization of products of two unipotent endomorphisms of index $2$
\cite{Bothaunipotent} by using the homothety trick, but we shall give a very quick proof of it that takes advantage of some general results that
are featured in this section.

Next, the result when both $p$ and $q$ are split with simple roots can be formulated as follows:

\begin{theo}\label{qtheoPandQsplitSingleroots}
Let $p$ and $q$ be split monic polynomials with degree $2$ in $\F[t]$, both
with simple roots, such that $p(0)q(0) \neq 0$. Set $\delta:=p(0)q(0)^{-1}$.
Let $u$ be an endomorphism of a finite-dimensional vector space $V$ over $\F$.
Then, the following conditions are equivalent:
\begin{enumerate}[(i)]
\item $u$ is a q-exceptional $(p,q)$-quotient.
\item $u$ is triangularizable with all its eigenvalues in $\Root(p)\Root(q)^{-1}$ and,
for every $z \in \Root(p)\Root(q)^{-1}$ such that $z \neq \delta z^{-1}$, the sequences
$\bigl(n_k(u,z)\bigr)_{k \geq 1}$ and $\bigl(n_k(u,\delta z^{-1})\bigr)_{k \geq 1}$ are $1$-intertwined.
\item Every invariant factor $r$ of $u$ is split with all its roots in $\Root(p)\Root(q)^{-1}$
and, for every $z \in \Root(p)\Root(q)^{-1}$ such that $z \neq \delta z^{-1}$, the respective multiplicities $m$ and $n$ of $z$ and $\delta z^{-1}$ as roots of
$r$ satisfy $|m-n| \leq 1$.
\item In some basis of $V$, the endomorphism $u$ is represented by a block-diagonal matrix in which every diagonal block
equals either $C\bigl((t-z)^n\bigr)$ for some $n \in \N^*$ and some $z \in \Root(p)\Root(q)^{-1}$ such that $z =\delta z^{-1}$, or
$C\bigl((t-z)^{n+\epsilon}\bigr)\oplus C\bigl((t-\delta z^{-1})^n\bigr)$ for some $n \in \N$, some $\epsilon \in \{0,1\}$ and some $z \in \Root(p)\Root(q)^{-1}$ such that $z \neq \delta z^{-1}$.
\end{enumerate}
\end{theo}

We finish with the case when $p$ is split with simple roots and $q$ is split with a double root:

\begin{theo}\label{qtheoPandQsplitSoleRootforQ}
Let $p$ and $q$ be split polynomials with degree $2$ in $\F[t]$, with $p(0)q(0) \neq 0$.
Assume that $p$ has two simple roots $x_1$ and $x_2$ and that $q$ has a double root $y$.
Let $u$ be an endomorphism of a finite-dimensional vector space $V$ over $\F$.
Then, the following conditions are equivalent:
\begin{enumerate}[(i)]
\item $u$ is a q-exceptional $(p,q)$-quotient.
\item $u$ is triangularizable with eigenvalues in $\{x_1y^{-1},x_2y^{-1}\}$, and the sequences
$\bigl(n_k(u,x_1y^{-1})\bigr)_{k \geq 1}$ and $\bigl(n_k(u,x_2y^{-1})\bigr)_{k \geq 1}$ are $2$-intertwined.
\item Every invariant factor $r$ of $u$ is split with its roots in $\{x_1y^{-1},x_2y^{-1}\}$, and
the respective multiplicities $m$ and $n$ of $x_1 y^{-1}$ and $x_2y^{-1}$ as roots of $r$ satisfy $|m-n| \leq 2$.
\item In some basis of $V$, the endomorphism $u$ is represented by a block-diagonal matrix in which every diagonal block
equals $C\bigl((t-z)^{n+\epsilon}\bigr)\oplus C\bigl((t-\delta z^{-1})^n\bigr)$ for some $n \in \N$, some $\epsilon \in \{0,1,2\}$ and some $z \in \Root(p)\Root(q)^{-1}$.
\end{enumerate}
\end{theo}

Note, in all three theorems above, that $G_{p,q}$ is split and that its roots are exactly the elements of $\Root(p)\Root(q)^{-1}$.
Hence, an endomorphism $u$ of a finite-dimensional vector space is q-exceptional with respect to $(p,q)$ if and only if its characteristic polynomial
is split with all roots in $\Root(p)\Root(q)^{-1}$, i.e.\ if and only if $u$ is triangularizable with all its eigenvalues in $\Root(p)\Root(q)^{-1}$.

\subsubsection{Key lemmas for sufficiency}

\begin{lemma}\label{splitcompanionlemma1}
Let $x,x',y,y'$ be nonzero scalars, and set $p(t):=(t-x)(t-x')$ and $q(t):=(t-y)(t-y')$.
Then, for all $s \in \N^*$, the matrices $C\bigl((t-xy^{-1})^s(t-x'(y')^{-1})^s\bigr)$ and
$C\bigl((t-xy^{-1})^s(t-x'(y')^{-1})^{s-1}\bigr)$ are $(p,q)$-quotients.
\end{lemma}

\begin{proof}
Let $n \in \N^*$ and set $s:=\lfloor \frac{n}{2}\rfloor$ and $\epsilon:=n-2s$.
Set
$$K:=\begin{bmatrix}
x & 0 \\
1 & x'
\end{bmatrix} \quad \text{and} \quad
L:=\begin{bmatrix}
y' & 0 \\
1 & y
\end{bmatrix}.$$
Note that $p(K)=0$ and $q(L)=0$.
If $n$ is even, set $A:=K \oplus \cdots \oplus K$ (with $s$ copies of $K$) and
$B:=yI_1 \oplus L \oplus \cdots \oplus L \oplus y'I_1$ (with $s-1$ copies of $L$), otherwise
set $A:=K \oplus \cdots \oplus K \oplus x I_1$ (with $s$ copies of $K$) and
$B:=y I_1 \oplus L \oplus \cdots \oplus L$ (with $s$ copies of $L$).
In any case $A$ and $B$ are matrices of $\Mat_n(\F)$, we have $p(A)=0=q(B)$, and
$M=AB^{-1}$ is a lower-triangular Hessenberg matrix whose diagonal entries
are $x y^{-1}, x' (y')^{-1},\dots$ and whose entries in the first subdiagonal are all nonzero.
It follows that $M$ is a cyclic matrix with characteristic polynomial $(t-xy^{-1})^s (t-x'(y')^{-1})^s$ if $n$ is even, and
$(t-xy^{-1})^{s+1} (t-x'(y')^{-1})^s$ otherwise.
In any case, $M$ is similar to $C\bigl((t-xy^{-1})^{s+\epsilon} (t-x'(y')^{-1})^s\bigr)$.
Varying $n$ then yields the claimed statement.
\end{proof}

\begin{lemma}\label{splitcompanionlemma2}
Let $x,x',y$ be nonzero scalars, and set $p(t):=(t-x)(t-x')$ and $q(t):=(t-y)^2$.
Then, for all $n \in \N$, the matrix $C\bigl((t-xy^{-1})^{n+2}(t-x'y^{-1})^n\bigr)$
is a $(p,q)$-quotient.
\end{lemma}

\begin{proof}
Let $n \in \N$. Set
$$K:=\begin{bmatrix}
x' & 0 \\
1 & x
\end{bmatrix}, \quad
L:=\begin{bmatrix}
y & 0 \\
1 & y'
\end{bmatrix},$$
$$A:=xI_1 \oplus K \oplus \cdots \oplus K \oplus x I_1 \quad \text{(with $n$ copies of $K$)}$$
and
$$B:=L \oplus \cdots \oplus L \quad \text{(with $n+1$ copies of $L$).}$$
Then, $A$ and $B$ are matrices of $\Mat_{2n+2}(\F)$ and, just like in the proof of the preceding lemma,
we have $p(A)=q(B)=0$ and $AB^{-1}$ is a cyclic matrix with characteristic polynomial $(t-xy^{-1})^{n+2}(t-x'y^{-1})^n$.
The claimed result follows.
\end{proof}

\subsubsection{Key lemmas for necessity}

Here, the strategy will be similar to the one used in \cite{dSPsumoftwotriang} to
analyze the exceptional $(p,q)$-sums when both polynomials $p$ and $q$ are split.
The proofs use the following general result, which is a minor variation of a result of Wang \cite{Wang1} (Lemma 2.3 there):

\begin{lemma}[Wang]\label{WangLemma}
Let $M$ and $N$ be matrices of $\Mat_r(\F)$ and $\Mat_s(\F)$, respectively, and $k$ be a positive integer.
Assume that there are matrices $X \in \Mat_{r,s}(\F)$
and $Y \in \Mat_{s,r}(\F)$ such that
$$M^k=XY \quad \text{and} \quad \Ker(MX)=\Ker(XN).$$
Then, $n_{k+1}(M,0) \leq n_1(N,0)$.
\end{lemma}

\begin{proof}
Let $Z \in \Ker M^{k+1}$. Then, $M (XY)Z=0$,  and hence $YZ \in \Ker (MX)$.
Moreover, $YZ \in \Ker X$ if and only if $XYZ=0$, that is $Z \in \Ker M^k$.
Hence, the mapping $Z \mapsto YZ$ yields a linear injection from $\Ker M^{k+1}/\Ker M^k$
into $\Ker (MX)/\Ker X$, leading to
$$n_{k+1}(M,0) \leq \dim \Ker(MX)-\dim \Ker X=\dim \Ker(XN)-\dim \Ker X.$$
By the classical Frobenius inequality, we have $\dim \Ker(XN) \leq \dim \Ker X+\dim \Ker N$, and hence
$$n_{k+1}(M,0) \leq \dim \Ker N=n_1(N,0).$$
\end{proof}

\begin{lemma}\label{1intertwinedSpecialLemma}
Let $p$ and $q$ be split monic polynomials of $\F[t]$ with degree $2$ such that $p(0)q(0) \neq 0$, write $p=(t-x_1)(t-x_2)$ and
$q=(t-y_1)(t-y_2)$ and assume that $x_1 \neq x_2$ and $y_1 \neq y_2$. Set $z:=x_1 y_1^{-1}$ and $z':=x_2 y_2^{-1}$ and assume that $z \neq z'$.
Let $u$ be an endomorphism of a finite-dimensional vector space, and assume that $u$ is annihilated by $(t-z)^2(t-z')^2$ and that
$u$ is a $(p,q)$-quotient.
Then,
$$n_2(u,z) \leq n_1(u,z').$$
\end{lemma}

\begin{proof}
Let us write $p=t^2-\lambda t+\alpha$ and $q=t^2-\mu t+\beta$.
Denote by $V$ the domain of $u$, and
by $r$ and $s$ the respective dimensions of $\Ker(u-z\, \id_V)^2$ and $\Ker(u-z'\, \id_V)^2$.
In some basis of $V$, the endomorphism $u$ is represented by the matrix
$$M=\begin{bmatrix}
z I_r-N & 0_{r \times s} \\
0_{s \times r} & z'I_s+N'
\end{bmatrix}$$
for some matrices $N \in \Mat_r(\F)$ and $N' \in \Mat_s(\F)$ such that $N^2=0$ and $(N')^2=0$.
Note that $n_2(u,z)=n_2(N,0)$ and $n_1(u,z')=n_1(N',0)$. We shall prove that Lemma \ref{WangLemma} applies to the pair
$(N,N')$ with $k=1$.

Since $u$ is a $(p,q)$-quotient, there are matrices $A$ and $B$ of $\GL_{r+s}(\F)$ such that $p(A)=0$, $q(B)=0$ and $M=AB^{-1}$.
Put differently, we have $q(B)=0$ and $p(MB)=0$. We start by analyzing $B$. First of all, we write $B$ as a block-matrix along the same pattern as $M$:
$$B=\begin{bmatrix}
B_1 & B_3 \\
B_2 & B_4
\end{bmatrix}.$$
By the Commutation Lemma (see the remark at the start of Section \ref{q-exceptionalsectionI}), the matrix $B$ commutes with the matrix
$$(M-z I)(I-z'M^{-1})=\begin{bmatrix}
(z'-z)z^{-1} N & 0_{r \times s} \\
0_{s \times r} & (z'-z) (z')^{-1} N'
\end{bmatrix}.$$
Since $z \neq z'$, it follows that $B_1$ commutes with $N$ and
$z^{-1}NB_3=(z')^{-1} B_3 N'$, the latter of which leads to $\Ker(NB_3)=\Ker(B_3N')$.

Next, we know that $p(MB)=0$, which reads $(MB)^2-\lambda MB+\alpha I=0$.
Multiplying by $(MB)^{-1}$, this yields $MB-\lambda I+\alpha B^{-1}M^{-1}=0$.
Yet, $B^{-1}=-\beta^{-1} B+\mu \beta^{-1} I$, whence
$$MB-\alpha \beta^{-1} BM^{-1}=\lambda I-\mu \alpha \beta^{-1} M^{-1}.$$
Extracting the upper-left block in this identity and using the fact that $B_1$ commutes with $N$, we obtain
$$\bigl((z I-N)-\alpha \beta^{-1} (zI-N)^{-1}\bigr)B_1=\lambda I-\mu \alpha \beta^{-1}(z I-N)^{-1}$$
and hence
$$\bigl((z I-N)-\alpha \beta^{-1} (zI-N)^{-1}\bigr)(B_1-y_1I)=-y_1 (zI-N)+\lambda I-\alpha \beta^{-1} y_2 (zI-N)^{-1}.$$
Multiplying by $zI-N$, we obtain
$$\bigl((z I-N)^2-\alpha \beta^{-1}I\bigr)(B_1-y_1I)=-y_1 (zI-N)^2+\lambda (zI-N)-\alpha (y_1)^{-1} I,$$
that is
$$\bigl(z(z-z') I-2zN+N^2\bigr)(B_1-y_1I)=-(y_1)^{-1}\bigl(y_1(zI-N)-x_1 I\bigr)\bigl(y_1(zI-N)-x_2 I\bigr),$$
and finally
$$\bigl(z(z-z') I-2zN+N^2\bigr)(B_1-y_1I)=N\bigl((x_1-x_2)I -y_1 N\bigr).$$
Here, $N^2=0$, and since $z \neq z'$ and $z \neq 0$, we deduce that
$$B_1-y_1 I=(x_1-x_2)N\bigl(z(z-z') I-2zN\bigr)^{-1}=\frac{x_1-x_2}{z(z-z')} N.$$
Hence,
$$B_1-y_2I=(y_1-y_2)I+\frac{x_1-x_2}{z(z-z')} N$$
and finally
$$q(B_1)=(B_1-y_1I)(B_1-y_2I)=\frac{(x_1-x_2)(y_1-y_2)}{z(z-z')} N.$$
However, the upper-left block of $q(B)$ equals $B_3B_2+q(B_1)$, and it follows that
$$N=\frac{z(z'-z)}{(x_1-x_2)(y_1-y_2)}\,B_3B_2=B_3 \bigl(z(z'-z)(x_1-x_2)^{-1} B_2\bigr).$$
Remembering that $\Ker (NB_3)=\Ker (B_3N')$, we deduce from Lemma \ref{WangLemma} that $n_2(N,0) \leq n_1(N',0)$,
i.e.\ $n_2(M,z) \leq n_1(M,z')$.
\end{proof}

\begin{lemma}\label{2intertwinedSpecialLemma}
Let $p$ and $q$ be split monic polynomials of $\F[t]$ with degree $2$ such that $p(0)q(0) \neq 0$, and assume that $q$ has a double root and that $p$ has simple roots. Write $p=(t-x_1)(t-x_2)$ and $q=(t-y)^2$ and set $z:=x_1 y^{-1}$ and $z':=x_2 y^{-1}$.
Let $u$ be an endomorphism of a finite-dimensional vector space, and assume that $u$ is annihilated by $(t-z)^3(t-z')^3$
and that it is a $(p,q)$-quotient.
Then,
$$n_3(u,z) \leq n_1(u,z').$$
\end{lemma}

\begin{proof}
Denote by $V$ the domain of $u$, and by $r$ and $s$ the respective dimensions of $\Ker(u-z\, \id_V)^3$ and $\Ker(u-z'\, \id_V)^3$.
In some basis of $V$, the endomorphism $u$ is represented by the matrix
$$M=\begin{bmatrix}
z I_r-N & 0_{r \times s} \\
0_{s \times r} & z' I_s+N'
\end{bmatrix}$$
for some matrices $N \in \Mat_r(\F)$ and $N' \in \Mat_s(\F)$ such that $N^3=0$ and $(N')^3=0$.
Note that $n_3(u,z)=n_3(N,0)$ and $n_1(u,z')=n_1(N',0)$. We shall prove that Lemma
\ref{WangLemma} applies to the pair $(N,N')$ and to $k:=2$.

Since $u$ is a $(p,q)$-quotient, there are matrices $A$ and $B$ of $\GL_{r+s}(\F)$ such that $p(A)=0$, $q(B)=0$ and $M=AB^{-1}$.
Put differently, we have $q(B)=0$ and $p(MB)=0$. Next, we analyze $B$. First of all, we write $B$ as a block-matrix along the same pattern as $M$:
$$B=\begin{bmatrix}
B_1 & B_3 \\
B_2 & B_4
\end{bmatrix}.$$
By the Commutation Lemma, $B$ commutes with
$$(M-z I)(I-z'M^{-1})=\begin{bmatrix}
(z'-z)z^{-1} N+z' z^{-2} N^2 & 0_{r \times s} \\
0_{s \times r} & (z'-z) (z')^{-1} N'+z (z')^{-2} (N')^2
\end{bmatrix}.$$
In particular, $B_1$ commutes with $(z'-z)z^{-1} N+z' z^{-2} N^2$. Yet, since
$(z'-z)z^{-1} \neq 0$ and $N$ is nilpotent, the matrix $N$ is a polynomial in $(z'-z)z^{-1} N+z' z^{-2} N^2$,
and we deduce that $B_1$ commutes with $N$.

Next, we also obtain from the above commutation that
$$\bigl((z'-z) z^{-1}N+z' z^{-2} N^2\bigr) B_3=B_3\bigl((z'-z) (z')^{-1}N'+z (z')^{-2} (N')^2\bigr),$$
and we shall prove that this leads to $\Ker(NB_3)=\Ker(B_3N')$.

Set $\alpha:=(z'-z) (z')^{-1}$ and $\beta:=z (z')^{-2}$.
Then, as $(N')^3=0$ the polynomial $R:=\alpha^{-1} t- \beta \alpha^{-3} t^2$ satisfies
$R(\alpha N'+\beta (N')^2)=N'$. On the other hand, we also have
$$R\bigl((z'-z) z^{-1}N+z' z^{-2} N^2\bigr)=\gamma N+\delta N^2$$
for some $\gamma \in \F \setminus \{0\}$ and some $\delta \in \F$.
It follows that
$$(\gamma I+\delta N) NB_3=B_3 N'.$$
Finally, since $\gamma \neq 0$ and $N$ is nilpotent, the matrix $\gamma I+\delta N$ is invertible, and we conclude that
$$\Ker(NB_3)=\Ker(B_3N').$$

Next, we express $q(B_1)$ as a function of $N$.
With exactly the same computation as in the proof of Lemma \ref{1intertwinedSpecialLemma}, we arrive at
$$\bigl(z(z-z') I-2zN+N^2\bigr)(B_1-yI)=N\bigl((x_1-x_2)I -y N\bigr).$$
Taking squares on both sides and using the fact that $B_1$ commutes with $N$, together with $N^3=0$, we obtain
$$\bigl(z(z-z') I-2zN+N^2\bigr)^2q(B_1)=N^2\bigl((x_1-x_2)I -y N\bigr)^2=(x_1-x_2)^2\,N^2.$$
Since $z(z-z') \neq 0$, this yields
$$q(B_1)=(x_1-x_2)^2\,N^2\bigl(z(z-z') I-2zN+N^2\bigr)^{-2}
=\frac{1}{z^2} \frac{(x_1-x_2)^2}{(z-z')^2} N^2=y^4 (x_1)^{-2} N^2.$$
Since the upper-left block of $q(B)$ equals $q(B_1)+B_3B_2$, we deduce that
$$N^2=-x_1^2 y^{-4} B_3B_2=B_3\bigl(-x_1^2y^{-4} B_2\bigr).$$
Remembering that $B_3N'$ and $NB_3$ have the same kernel, we deduce from Lemma \ref{WangLemma} that
$n_3(N,0) \leq n_1(N',0)$, that is $n_3(M,z) \leq n_1(M,z')$.
\end{proof}

\begin{prop}\label{1intertwinedGeneralProp}
Let $p$ and $q$ be split monic polynomials with degree $2$ over $\F$ such that $p(0)q(0) \neq 0$, both with simple roots.
Write $p(t)=(t-x_1)(t-x_2)$ and $q(t)=(t-y_1)(t-y_2)$.
Assume that $x_1 y_1^{-1} \neq x_2 y_2^{-1}$. Let $u$ be an endomorphism of a finite-dimensional vector space $V$, and assume that it is a $(p,q)$-quotient. Then,
the sequences $\bigl(n_k(u,x_1y_1^{-1})\bigr)_{k \geq 1}$ and $\bigl(n_k(u,x_2y_2^{-1})\bigr)_{k \geq 1}$
are $1$-intertwined.
\end{prop}

\begin{proof}
Let $a,b$ be automorphisms of $V$ such that $p(a)=0$, $q(b)=0$ and $u=ab^{-1}$.
Set $z:=x_1y_1^{-1}$ and $z':=x_2y_2^{-1}$. Note that $zz'=p(0)q(0)^{-1}$.
By the Commutation Lemma, we deduce that both $a$ and $b$ commute with $w:=(u-z\id_V)(\id_V-z' u^{-1})$.

Let $k \in \N$. We have $w^k=u^{-k}(u-z\id_V)^k (u-z'\id_V)^k$, and since $u$ is invertible it follows that
$$\Ker w^k=\Ker\bigl((u-z\id_V)^k (u-z'\id_V)^k\bigr).$$
Since $z \neq z'$, we deduce that
$$\Ker w^k=\Ker(u-z \id_V)^k\oplus \Ker(u-z' \id_V)^k.$$
Moreover, as $a$, $b$ and $u$ all stabilize $\Ker w^k$ and $\Ker w^{k+2}$, they induce respective endomorphisms
$a_k,b_k,u_k$ of the quotient space $E:=\Ker w^{k+2}/\Ker w^k$. Noting that $u_k=a_k (b_k)^{-1}$, we find that $u_k$ is a $(p,q)$-quotient.
However, $\Ker (u_k-z\id_E)$ is naturally isomorphic to $\Ker (u-z\id_V)^{k+1}/\Ker (u-z \id_V)^k$ and
likewise $\Ker (u_k-z\id_E)^2$ is naturally isomorphic to $\Ker (u-z\id_V)^{k+2}/\Ker (u-z \id_V)^k$; ditto with $z'$ instead of $z$.

It follows that $n_2(u_k,z)=n_{k+2}(u,z)$ and $n_1(u_k,z')=n_{k+1}(u,z')$.
Lemma \ref{1intertwinedSpecialLemma} yields that $n_2(u_k,z) \leq n_1(u_k,z')$, and hence $n_{k+2}(u,z) \leq n_{k+1}(u,z')$.
Likewise, we obtain $n_{k+2}(u,z') \leq n_{k+1}(u,z)$, which completes the proof.
\end{proof}

\begin{prop}\label{2intertwinedGeneralProp}
Let $p$ and $q$ be split monic polynomials with degree $2$ over $\F$ such that $p(0)q(0) \neq 0$.
Assume that $p$ has simple roots and that $q$ has a double root, and write
$p(t)=(t-x_1)(t-x_2)$ and $q(t)=(t-y)^2$.
Let $u$ be an endomorphism of a finite-dimensional vector space $V$, and assume that it is a $(p,q)$-quotient.
Then, the sequences $\bigl(n_k(u,x_1y^{-1})\bigr)_{k \geq 1}$ and $\bigl(n_k(u,x_2y^{-1})\bigr)_{k \geq 1}$
are $2$-intertwined.
\end{prop}

\begin{proof}
Let $a,b$ be automorphisms of $V$ such that $p(a)=0$, $q(b)=0$ and $u=ab^{-1}$.
Set $z:=x_1y^{-1}$ and $z':=x_2y^{-1}$, and note that $zz'=p(0)q(0)^{-1}$.
By the Commutation Lemma, we deduce that both $a$ and $b$ commute with $w:=(u-z\id_V)(\id_V-z'u^{-1})$.

Let $k \in \N$. Then, $a$, $b$ and $u$ all stabilize $\Ker w^k$ and $\Ker w^{k+3}$, and hence they induce endomorphisms
$a_k$, $b_k$ and $u_k$ of the quotient space $E:=\Ker w^{k+3}/\Ker w^k$. Noting that $u_k=a_k (b_k)^{-1}$, we find that $u_k$ is a $(p,q)$-quotient.
As in the proof of Proposition \ref{1intertwinedGeneralProp}, it is easily checked that
$n_3(u_k,z)=n_{k+3}(u,z)$ and $n_1(u_k,z')=n_{k+1}(u,z')$, and hence by Lemma \ref{2intertwinedSpecialLemma}
we obtain $n_{k+3}(u,z) \leq n_{k+1}(u,z')$. Symmetrically, we find $n_{k+3}(u,z') \leq n_{k+1}(u,z)$, which completes the proof.
\end{proof}

\subsubsection{Proof of Theorem \ref{qtheoPandQsplitSoleRoot}}

Here, we assume that $p=(t-x)^2$ and $q=(t-y)^2$, where $x$ and $y$ are nonzero scalars in $\F$.
Hence, $G_{p,q}(t)=(t-xy^{-1})^4$.
Let $u$ be an endomorphism of a finite-dimensional vector space $V$ over $\F$.
If $u$ is a q-exceptional $(p,q)$-quotient, then we know that the characteristic polynomial of $u$
is a power of $t-xy^{-1}$.
Conversely, assume that the characteristic polynomial of $u$ is a power of $t-xy^{-1}$.
Then, all the invariant factors of $u$ are powers of $t-xy^{-1}$, and hence in some basis of $V$
the endomorphism $u$ is represented by a block-diagonal matrix in which every diagonal block equals
$C\bigl((t-xy^{-1})^n\bigr)$ for some positive integer $n$. Yet, for every $n \in \N^*$ we know from Lemma \ref{splitcompanionlemma1}
that $C\bigl((t-xy^{-1})^n\bigr)$ is a $(p,q)$-quotient, and hence so is $u$.

\subsubsection{Proof of Theorem \ref{qtheoPandQsplitSingleroots}}

Let $p$ and $q$ be split monic polynomials with degree $2$ in $\F[t]$, both
with simple roots only, such that $p(0)q(0) \neq 0$. Set $\delta:=p(0)q(0)^{-1}$.
Let $u$ be an endomorphism of a finite-dimensional vector space over $\F$.
Note that
$$G_{p,q}=\prod_{(x,y)\in \Root(p) \times \Root(q)}(t-xy^{-1}).$$

Assume that condition (i) holds for $u$. Then, we know that $u$ is triangularizable and all its eigenvalues belong to $\Root(p)\Root(q)^{-1}$.
Moreover, for all $z \in \Root(p)\Root(q)^{-1}$ such that $z \neq \delta z^{-1}$,
Proposition \ref{1intertwinedGeneralProp} yields that the sequences
$\bigl(n_k(u,z)\bigr)_{k \geq 1}$ and $\bigl(n_k(u,z')\bigr)_{k \geq 1}$ are $1$-intertwined. Hence, condition (ii) holds.

Assume now that condition (ii) holds. Then, every invariant factor of $u$ is split with all its roots in $\Root(p)\Root(q)^{-1}$.
Denote now by $r_1,\dots,r_k,\dots$ the invariant factors of $u$ and, for $z \in \Root(p)\Root(q)^{-1}$ and $k \in \N^*$, denote by $m_k(z)$
the multiplicity of $z$ as a root of $r_k$.
Let $z \in \Root(p)\Root(q)^{-1}$ be such that $z \neq \delta z^{-1}$, and assume that there exists a positive integer
$k$ such that $m_k(z)>m_k(\delta z^{-1})+1$. Then, by setting $l:=m_k(\delta z^{-1})+2$, we see that
$n_{l-1}(u,\delta z^{-1})<k \leq n_l(u,z)$, which contradicts condition (ii). Hence, $m_k(z) \leq m_k(\delta z^{-1})+1$.
Symmetrically, we find $m_k(\delta z^{-1}) \leq m_k(z)+1$, and hence $\big|m_k(z)-m_k(\delta z^{-1})\big| \leq 1$.
Hence, condition (iii) holds.

When we have pairwise coprime monic polynomials $s_1,\dots,s_k$, it is folklore that $C(s_1\cdots s_k) \simeq C(s_1) \oplus \cdots \oplus C(s_k)$,
and hence condition (iii) implies condition (iv).
With the help of the same remark, one deduces from Lemma \ref{splitcompanionlemma1} that condition (iv) implies condition (i).

\subsubsection{Proof of Theorem \ref{qtheoPandQsplitSoleRootforQ}}

The proof is very similar to the one of Theorem \ref{qtheoPandQsplitSingleroots}.
Let $p$ and $q$ be split monic polynomials with degree $2$ in $\F[t]$, with $p(0)q(0) \neq 0$.
Assume that $p$ has simple roots $x_1$ and $x_2$ and that $q$ has a double root $y$.

Let $u$ be an endomorphism of a finite-dimensional vector space over $\F$.
Note that $G_{p,q}=(t-x_1y^{-1})^2 (t-x_2 y^{-1})^2$.

Assume that condition (i) holds for $u$. Then, we know that $u$ is triangularizable and all its eigenvalues belong to $\{x_1y^{-1},x_2 y^{-1}\}$.
Moreover, Proposition \ref{2intertwinedGeneralProp} yields that the sequences
$\bigl(n_k(u,x_1y^{-1})\bigr)_{k \geq 1}$ and $\bigl(n_k(u,x_2y^{-1})\bigr)_{k \geq 1}$ are $2$-intertwined. Hence, condition (ii) holds.

Assume now that condition (ii) holds. Then, every invariant factor of $u$ is split with all its roots in $\{x_1y^{-1},x_2 y^{-1}\}$.
Denote now by $r_1,\dots,r_k,\dots$ the invariant factors of $u$ and, for $k \in \N^*$, write $r_k(t)=(t-x_1y^{-1})^{a_k} (t-x_2y^{-1})^{b_k}$
for some non-negative integers $a_k$ and $b_k$. If there existed a positive integer
$k$ such that $a_k>b_k+2$ then $l:=b_k+3$ would satisfy
$n_{l-2}(u,x_2y^{-1})<k \leq n_l(u,x_1y^{-1})$, which would contradict condition (ii). Hence, $a_k \leq b_k+2$, and likewise $b_k \leq a_k+2$.
Hence, condition (iii) holds.

Condition (iii) obviously implies condition (iv). Finally, for all $n \in \N$ and all $\epsilon \in \{0,1,2\}$, both matrices $C\bigl((t-x_1y^{-1})^{n+\epsilon}\bigr) \oplus C\bigl((t-x_2y^{-1})^{n}\bigr)$ and $C\bigl((t-x_2y^{-1})^{n+\epsilon}\bigr) \oplus C\bigl((t-x_1y^{-1})^{n}\bigr)$ are $(p,q)$-quotients
(by Lemma \ref{splitcompanionlemma1} if $\epsilon\neq 2$, and by Lemma \ref{splitcompanionlemma2} otherwise).
Hence, condition (iv) implies condition (i).

\subsection{Exceptional $(p,q)$-quotients (II): When $p$ is irreducible but $q$ is not}\label{q-exceptionalsectionII}

In this section and the following ones, we fix two monic polynomials $p$ and $q$ of degree $2$ in $\F[t]$, such that $p(0)q(0) \neq 0$.
We set $\delta:=p(0)q(0)^{-1}$.

Here, we assume that $p$ is irreducible but not $q$.
We split $q(t)=(t-y_1)(t-y_2)$ with $y_1$ and $y_2$ in $\F$.
Then, 
$$G_{p,q}(t)=H_{y_1}(p) H_{y_2}(p),$$
the polynomials $H_{y_1}(p)$ and $H_{y_2}(p)$ are monic with degree $2$ in $\F[t]$ and they are irreducible.

Hence:
\begin{itemize}
\item Either $H_{y_1}(p) \neq H_{y_2}(p)$, in which case an endomorphism of
a finite-dimensional vector space over $\F$ is exceptional with respect to $(p,q)$ if and only if
the irreducible monic divisors of its minimal polynomial belong to $\bigl\{H_{y_1}(p),H_{y_2}(p)\bigr\}$;
\item Or $H_{y_1}(p) = H_{y_2}(p)$, in which case an endomorphism of
a finite-dimensional vector space over $\F$ is q-exceptional with respect to $(p,q)$ if and only if
its minimal polynomial is a power of $H_{y_1}(p)$.
\end{itemize}

Note that, given $\lambda \in \F \setminus\{0\}$, one has $H_\lambda(p)=p$ if and only if
either $\lambda=1$, or $\lambda=-1$ and $\tr(p)=0$.
It follows that $H_{y_1}(p)=H_{y_2}(p)$ if and only if one of the following two conditions holds:
\begin{enumerate}[(i)]
\item $y_1=y_2$;
\item $\tr(p)=0$ and $y_2=-y_1$, i.e. $\tr(p)=\tr(q)=0$.
\end{enumerate}
Note that, in the second case, $z=\delta z^{-1}$ for every $z$ in $\Root(p)\Root(q)^{-1}$.

\subsubsection{A common lemma}

\begin{lemma}\label{qcommonQsplitsnotP}
Let $p$ be an irreducible monic polynomial with degree $2$ over $\F$, and let $y_1$ and $y_2$ be nonzero scalars in $\F$.
Set $q:=(t-y_1)(t-y_2)$. \\
Then, for all $n \in \N$, the companion matrices $C\bigl(H_{y_1}(p)^{n}\,H_{y_2}(p)^n\bigr)$ and
$C\bigl(H_{y_1}(p)^{n+1}\,H_{y_2}(p)^n\bigr)$ are $(p,q)$-quotients.
\end{lemma}

\begin{proof}
Let $s$ be a non-negative integer.
We extend $(y_1,y_2)$ into a $2$-periodical sequence $(y_k)_{k \geq 1}$.
Set $K:=\begin{bmatrix}
0 & 1 \\
0 & 0
\end{bmatrix} \in \Mat_2(\F)$,
$$A_s:=\bigl(y_1\,C(H_{y_1}(p))\bigr) \oplus \cdots \oplus \bigl(y_s\,C(H_{y_s}(p))\bigr) \in \Mat_{2s}(\F)$$
and
$$B_s:=\begin{bmatrix}
y_1^{-1} I_2 & 0_2 & \cdots & \cdots & (0) \\
K & y_2^{-1} I_2 & \ddots & & \vdots \\
0_2 & \ddots & \ddots & \ddots & \vdots \\
\vdots & & \ddots & y_{s-1}^{-1} I_2 & 0_2 \\
(0) & \cdots & 0_2 & K & y_s^{-1} I_2
\end{bmatrix}^{-1} \in \Mat_{2s}(\F).$$
For all $i \in \N^*$, we see that $p$ annihilates $y_i\,C(H_{y_i}(p))$, and hence $p(A_s)=0$.
On the other hand, one checks that $(B_s^{-1}-y_1^{-1} I_{2s})(B_s^{-1}-y_2^{-1} I_{2s})=0$,
which, by multiplying by $y_1y_2 B_s^2$, yields $q(B_s)=0$.

Set $L:=\begin{bmatrix}
0 & 0 \\
0 & 1
\end{bmatrix} \in \Mat_2(\F)$. Then, we see that
$$A_s B_s^{-1}=\begin{bmatrix}
C\bigl(H_{y_1}(p)\bigr) & 0_2 & \cdots & \cdots & (0) \\
y_2 L & C\bigl(H_{y_2}(p)\bigr) & \ddots & & \vdots \\
0_2 & \ddots & \ddots & \ddots & \vdots \\
\vdots & & \ddots & C\bigl(H_{y_{s-1}}(p)\bigr) & 0_2 \\
(0) & \cdots & 0_2 & y_s L &  C\bigl(H_{y_{s}}(p)\bigr)
\end{bmatrix}.$$
The matrix $A_s B_s^{-1}$ is a $(p,q)$-quotient and
its characteristic polynomial equals $\underset{k=1}{\overset{s}{\prod}} H_{y_k}(p)$.
If $s$ is even, this characteristic polynomial equals $H_{y_1}(p)^{s/2} H_{y_2}(p)^{s/2}$;
otherwise it equals $H_{y_1}(p)^{(s+1)/2} H_{y_2}(p)^{(s-1)/2}$.

In order to conclude, it suffices to proves that $A_s B_s^{-1}$ is cyclic.
To do so, we distinguish between two cases, whether $H_{y_1}(p)$ equals $H_{y_2}(p)$ or not.

\noindent \textbf{Case 1:  $H_{y_1}(p)=H_{y_2}(p)$.} \\
Set then $r:=H_{y_1}(p)$ and write $r=t^2-\gamma t-\eta$.
For all $k \in \lcro 2,s\rcro$, we see that
$$-\gamma L+L C(r)+C(r)L=C(r).$$
Hence,
$$r(A_s B_s^{-1})=\begin{bmatrix}
0_2 & 0_2 & \cdots & \cdots & (0) \\
y_2 C(r) & 0_2 & \ddots & & \vdots \\
? & \ddots & \ddots & \ddots & \vdots \\
\vdots & & \ddots & 0_2 & 0_2 \\
(?) & \cdots & ? & y_s C(r) &  0_2
\end{bmatrix}.$$
Since $C(r)$ is invertible and all the scalars $y_k$ are nonzero, we find that $r(A_s B_s^{-1})$ has rank $2s-2$, and hence the kernel of $r(A_s B_s^{-1})$ has dimension $2$.
Yet, every invariant factor of $A_s B_s^{-1}$ is a power of $r$, and if we denote by $c$ the number of those invariant factors, we have
$2c=\dim \Ker\bigl(r(A_s B_s^{-1})\bigr)$. Hence, $c=1$, which yields that $A_s B_s^{-1}$ is cyclic, as claimed.

\vskip 2mm
\noindent \textbf{Case 2: $H_{y_1}(p)\neq H_{y_2}(p)$.} \\
Set $r_1:=H_{y_1}(p)$ and $r_2:=H_{y_2}(p)$.
If $s \leq 2$ then it is obvious that $A_s B_s^{-1}$ has its minimal polynomial equal to its characteristic polynomial, and hence $A_s B_s^{-1}$
is cyclic. In the remainder of the proof, we assume that $s \geq 3$.
Write $r_1=t^2-\gamma_1 t-\delta_1$ and $r_2=t^2-\gamma_2 t-\delta_2$. Then,
$-\gamma_1 L+ LC(r_1)+C(r_2)L=C(r_2)$ and $-\gamma_1 L+ LC(r_2)+C(r_1)L=U$, where $U:=\begin{bmatrix}
0 & \delta_1 \\
1 & \gamma_2
\end{bmatrix}$. Hence,
$$r_1(A_s B_s^{-1})=\begin{bmatrix}
0_2 & 0_2 & \cdots & \cdots & \cdots & 0_2 \\
y_2 C(r_2) & r_1(C(r_2)) & 0_2 & \cdots & \cdots & 0_2 \\
y_2y_3 L & y_3 U & 0_2 & \cdots & \cdots & 0_2 \\
0_2 & y_3y_4 L & y_4 C(r_2) & r_1(C(r_2)) & 0_2 & \vdots \\
0_2 & 0_2 & y_4y_5 L & y_5 U & 0_2 & & \\
\vdots & \vdots & & \ddots & \ddots & \ddots
\end{bmatrix}.$$
Note that $r_1(C(r_2))=r_2(C(r_2))+(\gamma_2-\gamma_1) C(r_2)+(\delta_2-\delta_1) I_2=(\gamma_2-\gamma_1) C(r_2)+(\delta_2-\delta_1) I_2$,
and in particular the upper-left entry of $r_1(C(r_2))$ equals $\delta_2-\delta_1$.

For all $i\in \N^*$ such that $2i+1 \leq s$, set
$$H_i:=\begin{bmatrix}
y_{2i} C(r_2) & r_1(C(r_2)) \\
y_{2i}\,y_{2i+1} L & y_{2i+1} U
\end{bmatrix} \in \Mat_4(\F),$$
and note that $H_i$ is invertible:
Indeed,
$$\det H_i=y_{2i}^2 \,y_{2i+1}^2 \begin{vmatrix}
C(r_2) & r_1(C(r_2)) \\
L & U
\end{vmatrix};$$
Then, by successively developing along the first column and then the second row of the resulting matrix, we find
$$\begin{vmatrix}
C(r_2) & r_1(C(r_2)) \\
L & U
\end{vmatrix}=\delta_1 \begin{vmatrix}
\delta_2 & \delta_2-\delta_1 \\
1 & 1
\end{vmatrix}=\delta_1^2 \neq 0.$$
Hence, $H_i$ is invertible for all $i \geq 1$ such that $2i+1 \leq s$.

Now:
\begin{itemize}
\item If $s$ is odd then the rank of $r_1(A_s B_s^{-1})$ equals that of the matrix obtained from it by removing
the first two rows and the last two columns; Yet this matrix of $\Mat_{2s-2}(\F)$ is block-lower-triangular with diagonal
blocks $H_1,\dots,H_{(s-1)/2}$, and hence it is invertible. It follows that
the kernel of $r_1(A_s B_s^{-1})$ has dimension $2$.
\item If $s$ is even then the rank of $r_1(A_s B_s^{-1})$ is at most $2s-2$ (since the first two rows are zero),
and is at least the one of the matrix obtained from it by removing the first two rows and the last two columns; Yet this matrix
of $\Mat_{2s-2}(\F)$ is block-lower-triangular with diagonal blocks $H_1,\dots,H_{(s-2)/2},y_s C(r_2)$, all of them invertible.
Hence the kernel of $r_1(A_s B_s^{-1})$ has dimension $2$.
\end{itemize}
With a similar method, one checks that the kernel of $r_2(A_s B_s^{-1})$ has dimension $2$
(again, one needs to discuss whether $s$ is odd or even).
Finally, every elementary invariant of $A_s B_s^{-1}$ is a power of $r_1$ or $r_2$.
Denoting by $a$ (respectively, by $b$), the number of such invariants that are powers of $r_1$ (respectively, of $r_2$), we have
$\dim \Ker \bigl(r_1(A_s B_s^{-1})\bigr)=2a$ and $\dim \Ker \bigl(r_2(A_s B_s^{-1})\bigr)=2b$. Hence, $a=b=1$.
Since $r_1$ and $r_2$ are coprime, we conclude that $A_s B_s^{-1}$ is cyclic, as claimed.
This completes the proof.
\end{proof}

\subsubsection{The case when $H_{y_1}(p)=H_{y_2}(p)$}

The following theorem is an obvious consequence of Lemma \ref{qcommonQsplitsnotP} and of the considerations of the start of Section
\ref{q-exceptionalsectionII}:

\begin{theo}\label{qtheoQsplitswithdouble}
Let $p$ and $q$ be monic polynomials with degree $2$ of $\F[t]$ such that $p(0)q(0) \neq 0$.
Assume that $p$ is irreducible over $\F$ and that $q$ splits over $\F$.
Assume furthermore that $q$ has a double root or that $\tr(q)=\tr(p)=0$.
Choose a root $y$ of $q$ in $\F$.
Then, given an endomorphism $u$ of a finite-dimensional vector space $V$ over $\F$, the following conditions are equivalent:
\begin{enumerate}[(i)]
\item $u$ is a q-exceptional $(p,q)$-quotient.
\item The minimal polynomial of $u$ is a power of $H_y(p)$.
\item In some basis of $V$, the endomorphism $u$ is represented by a block-diagonal matrix in which every diagonal block
equals $C\bigl(H_y(p)^n\bigr)$ for some positive integer $n$.
\end{enumerate}
\end{theo}

\subsubsection{The case when $H_{y_1}(p)\neq H_{y_2}(p)$}

\begin{theo}
Let $p$ and $q$ be monic polynomials with degree $2$ of $\F[t]$ such that $p(0)q(0)\neq 0$.
Assume that $p$ is irreducible over $\F$ and that $q$ splits over $\F$ with distinct roots $y_1$ and $y_2$.
Assume furthermore that $(\tr p,\tr q) \neq (0,0)$.
Let $u$ be an endomorphism of  a finite-dimensional vector space $V$ over $\F$. The following conditions are then equivalent:
\begin{enumerate}[(i)]
\item $u$ is a q-exceptional $(p,q)$-quotient.
\item The irreducible monic divisors of the minimal polynomial of $u$ belong to $\bigl\{H_{y_1}(p),H_{y_2}(p)\bigr\}$.
Moreover, in writing the invariant factors of $u$ as $r_1=H_{y_1}(p)^{\alpha_1}H_{y_2}(p)^{\beta_1},\dots,r_k=H_{y_1}(p)^{\alpha_k}H_{y_2}(p)^{\beta_k},\dots$, we have $|\alpha_k-\beta_k| \leq 1$ for all $k \in \N^*$.
\item In some basis of $V$, the endomorphism $u$ is represented by a block-diagonal matrix in which every diagonal block
equals $C\bigl(H_{y_1}(p)^n\bigr) \oplus C\bigl(H_{y_2}(p)^n\bigr)$,
$C\bigl(H_{y_1}(p)^n\bigr) \oplus C\bigl(H_{y_2}(p)^{n-1}\bigr)$
or $C\bigl(H_{y_1}(p)^{n-1}\bigr) \oplus C\bigl(H_{y_2}(p)^n\bigr)$ for some positive integer $n$.
\end{enumerate}
\end{theo}

\begin{proof}
Since $H_{y_1}(p)$ and $H_{y_2}(p)$ are coprime,
Lemma \ref{qcommonQsplitsnotP} shows that condition (iii) implies condition (i).
Moreover, it is also clear that condition (ii) implies condition (iii).

Now, we assume that condition (i) holds, and we aim at proving that condition (ii) also does.

For $k \in \N^*$, denote by $a_k$ (respectively, by $b_k$)
the number of elementary invariants of $u$ of the form $H_{y_1}(p)^l$ (respectively, of the form $H_{y_2}(p)^l$) for some integer $l \geq k$.
Let us temporarily assume that $(a_k)_{k \geq 1}$ and $(b_k)_{k \geq 1}$ are $1$-intertwined.
Then, we claim that condition (ii) holds.
Indeed, since $u$ is q-exceptional with respect to $(p,q)$ we already know that the monic irreducible divisors of its minimal polynomial
are among $H_{y_1}(p)$ and $H_{y_2}(p)$.
Next, denote by $r_1,\dots,r_k,\dots$ the invariant factors of $u$, and write further
$r_k=H_{y_1}(p)^{\alpha_k}H_{y_2}(p)^{\beta_k}$ for some non-negative integers $\alpha_k$ and $\beta_k$.
Assume that, for some positive integer $k$, we have $\alpha_k>\beta_k+1$. Setting $s:=\beta_k+2$, we get that
$$a_s \geq k \quad \text{and} \quad b_{s-1}<k,$$
which contradicts the assumption that $(a_i)_{i \geq 1}$ and $(b_i)_{i \geq 1}$ be $1$-intertwined. Hence, $\alpha_k \leq \beta_k+1$, and likewise
$\beta_k \leq \alpha_k+1$. Hence condition (ii) holds.

It remains to prove that $(a_k)_{k \geq 1}$ and $(b_k)_{k \geq 1}$ are $1$-intertwined.
To do so, we need to distinguish between two cases.
In both, we denote by $\L$ the splitting field of $p$ (as defined by $\L:=\F[t]/(p)$) and
we shall consider the $\L$-vector space $V^{\L}:=V \otimes_{\F} \L$ and the endomorphism
$u^{\L}$ of $V^{\L}$ deduced from $u$ by extending the field of scalars. Note that $u_\L$ is a $(p,q)$-quotient.

\begin{itemize}
\item Assume first that $p$ has two distinct roots $x_1$ and $x_2$ in $\L$. \\
First of all, we prove that $x_1 y_1^{-1} \neq x_2 y_2^{-1}$. Assume indeed that the contrary holds.
Then, $x_1+x_2 \in \F$ and $x_1 y_2-x_2 y_1 \in \F$: since $(x_1,x_2)\not\in \F^2$ this yields
$\begin{vmatrix}
1 & 1 \\
y_2 & -y_1
\end{vmatrix}=0$, whence $y_1+y_2=0$, and, by using $x_1 y_1^{-1} = x_2 y_2^{-1}$, this further leads to $x_1+x_2=0$, contradicting the assumption that $(\tr p,\tr q)\neq (0,0)$.
Hence,  $x_1 y_1^{-1} \neq x_2 y_2^{-1}$. Likewise, $x_2 y_2^{-1} \neq x_2 y_1^{-1}$, and it follows that
$x_1y_1^{-1}, x_2 y_1^{-1}, x_1 y_2^{-1},x_2 y_2^{-1}$ are pairwise distinct.
Therefore, for all $k \in \N^*$,
$$n_k\bigl(u^\L,x_1y_1^{-1}\bigr)=a_k \quad \text{and} \quad n_k\bigl(u^\L,x_2y_2^{-1}\bigr)=b_k,$$
and since $x_1y_1^{-1} \neq x_2 y_2^{-1}$ Theorem \ref{theoPandQsplitSingleroots} yields that
the sequences $\Bigl(n_k\bigl(u^\L,x_1y_1^{-1}\bigr)\Bigr)_{k \geq 1}$ and $\Bigl(n_k\bigl(u^\L,x_2y_2^{-1}\bigr)\Bigr)_{k \geq 1}$
are $1$-intertwined.

\item Assume now that $p$ has a sole root $x$ in its splitting field $\L$ (note that this can happen only
if $\F$ has characteristic $2$). \\
Then, as $xy_1^{-1} \neq xy_2^{-1}$, one sees that, for all $k \in \N^*$,
$$n_{2k}\bigl(u^\L,xy_1^{-1}\bigr)=a_k \quad \text{and} \quad
n_{2k}\bigl(u^\L,xy_2^{-1}\bigr)=b_k.$$
Here, Theorem \ref{theoPandQsplitSoleRootforQ} applies to $(u^{\L})^{-1}$ and to the pair $(q,p)$, and hence the sequences
$\Bigl(n_k\bigl(u^\L,x y_1^{-1}\bigr)\Bigr)_{k \geq 1}$ and
$\Bigl(n_k\bigl(u^\L,x y_2^{-1}\bigr)\Bigr)_{k \geq 1}$ are $2$-intertwined.
\end{itemize}
Hence, in any case we see that $(a_k)_{k \geq 1}$ and $(b_k)_{k \geq 1}$ are $1$-intertwined, which completes the proof that condition (i)
implies condition (ii).
\end{proof}

\subsection{Exceptional $(p,q)$-quotients (III): When $p$ and $q$ are irreducible with the same splitting field}\label{q-exceptionalsectionIII}

Here, we assume that $p$ and $q$ are both irreducible, with the same splitting field $\L \subset \overline{\F}$.
Note then that $p(0)q(0) \neq 0$.
Denote by $\sigma$ the non-identity automorphism of $\L$ over $\F$ if $\L$ is separable over $\F$,
otherwise set $\sigma:=\id_\L$. In any case, splitting $p(t)=(t-x_1)(t-x_2)$ and $q(t)=(t-y_1)(t-y_2)$
in $\L[t]$, we find that $\sigma$ exchanges $x_1$ and $x_2$, and that it exchanges $y_1$ and $y_2$.
Hence, $x_1y_1^{-1}+x_2y_2^{-1}=\Tr_{\L/\F}(x_1y_1^{-1})$ and $x_1y_2^{-1}+x_2y_1^{-1}=\Tr_{\L/\F}(x_1y_2^{-1})$.
Hence, with $\delta:=p(0)q(0)^{-1}$, we have
$$G_{p,q}(t)=\bigl(t^2-\Tr_{\L/\F}(x_1y_1^{-1})\,t+\delta\bigr)\bigl(t^2-\Tr_{\L/\F}(x_1y_2^{-1})\,t+\delta\bigr)$$
and
$$\Theta_{p,q}(t)=\bigl(t-\Tr_{\L/\F}(x_1y_2)\bigr)\bigl(t-\Tr_{\L/\F}(x_1y_1)\bigr).$$

Next, assume that $G_{p,q}$ has a root in $\F$. Then,
we have respective roots $x$ and $y$ of $p$ and $q$, together with some nonzero scalar $d\in \F$ such that $x=dy$.
Hence, $\sigma(x)=d\sigma(y)$ and it follows that $q=H_d(p)$.

Conversely, assume that $q=H_d(p)$ for some $d \in \F \setminus \{0\}$.
Then, we see that
$$G_{p,q}(t)=(t-d)^2 \Bigl(t-\frac{dy_1}{y_2}\Bigr)\Bigl(t-\frac{dy_2}{y_1}\Bigr).$$
Note then that $\tr p=d\,\tr q$ and $d^2=\frac{p(0)}{q(0)}$, so that
$$\frac{dy_1}{y_2}+\frac{d y_2}{y_1}=
d \,\frac{y_1^2+y_2^2}{y_1y_2}=d\,\frac{(\tr q)^2-2\,q(0)}{y_1y_2}=\frac{\tr p \tr q-2dq(0)}{q(0)}\cdot$$
Hence,
$$\Bigl(t-\frac{dy_1}{y_2}\Bigr)\Bigl(t-\frac{dy_2}{y_1}\Bigr)=t^2-\frac{\tr p \tr q-2d q(0)}{q(0)}\,t+\frac{p(0)}{q(0)}$$
and
$$G_{p,q}(t)=(t-d)^2\Bigl(t^2-\frac{\tr p \tr q-2d q(0)}{q(0)}\,t+\frac{p(0)}{q(0)}\Bigr).$$
Assume that $\frac{dy_1}{y_2}$ is an element of $\F$, which we denote by $e$.
Then, as $y_1+y_2 \in \F$ and $dy_1-e y_2=0 \in \F$, whereas $y_1 \not\in \F$, the only possibility is that
$\begin{vmatrix}
1 & 1 \\
d &  -e
\end{vmatrix}=0$, whence $e=-d$ and $y_1=-y_2$. It would follow that we also have $x_1=dy_1=-dy_2=-x_2$. Hence, $\tr p=\tr q=0$.
Therefore:
\begin{itemize}
\item Either $(\tr p,\tr q) \neq (0,0)$, in which case $t^2-\frac{\tr p \tr q-2d q(0)}{q(0)}\,t+\frac{p(0)}{q(0)}$
is irreducible over $\F$.

\item Or $(\tr p,\tr q)=(0,0)$, in which case
$$G_{p,q}(t)=(t-d)^2(t+d)^2.$$
\end{itemize}

Conversely, if $\tr p=\tr q=0$ and $\F$ does not have characteristic $2$, then, for every $(x,y)\in \Root(p) \times \Root(q)$,
we have $\sigma(xy^{-1})=\sigma(x)\sigma(y)^{-1}=(-x)(-y)^{-1}=xy^{-1}$, whence $xy^{-1} \in \F$.
Hence, in that case $G_{p,q}$ splits over $\F$, and from the above we see that $G_{p,q}(t)=(t-d)^2(t+d)^2$ for some $d \in \F \setminus \{0\}$.

We note that $t^2-\Tr_{\L/\F}(x_1y_1^{-1})\,t+\delta$ and $t^2-\Tr_{\L/\F}(x_1y_2^{-1})\,t+\delta$
have respective roots $x_1y_1^{-1},x_2y_2^{-1}$ and $x_1y_2^{-1},x_2 y_1^{-1}$, and hence they are equal only if
they have a common root, which implies that $x_1=x_2$ or $y_1=y_2$. In any of those cases, one of $p$ and $q$ is inseparable, and since they have
the same splitting field both are inseparable. Hence,
$t^2-\Tr_{\L/\F}(x_1y_1^{-1})t+\delta \neq t^2-\Tr_{\L/\F}(x_1y_2^{-1})t+\delta$ unless $\F$ has characteristic $2$ and $p$ and $q$ are inseparable.

\vskip 3mm
Finally, let $u$ be an endomorphism of a finite-dimensional vector space $V$ over $\F$ and assume that $u$ is q-exceptional with respect to $(p,q)$.
Then, $v:=q(0)u+p(0)u^{-1}$ is annihilated by some power of $\Theta_{p,q}$, which splits over $\F$;
hence, $v$ is triangularizable. In the next section, we study the case when $v$ has a sole eigenvalue.

\subsubsection{When $v$ has a sole eigenvalue}

We keep the assumptions and notation from the last paragraph of the preceding section and now we assume that $v$
has a sole eigenvalue. By renaming the roots of $p$ and $q$ in $\L$, we can assume that this eigenvalue is
$$z:=\Tr_{\L/\F}(x_1y_2)=x_1y_2+x_2y_1.$$

\begin{lemma}\label{qsamesplitconstructivelemma}
Let $p$ and $q$ be monic irreducible polynomials with degree $2$ over $\F$.
Assume that $p$ and $q$ have the same splitting field $\L$ in $\overline{\F}$. Let $x$ and $y$ be respective roots of $p$ and $q$ in $\L$ such that
$xy^{-1} \not\in \F$. Set $z:=\Tr_{\L/\F}(x y^{-1})$ and $\delta:=p(0)q(0)^{-1}$.
Then, for all $n \in \N$, the matrix $C\bigl((t^2-z t+\delta)^{n+1}\bigr) \oplus C\bigl((t^2-zt+\delta)^n\bigr)$
is a $(p,q)$-quotient (and a q-exceptional one).
\end{lemma}

\begin{proof}
Set $r:=t^2-zt+\delta$ and note that $r=(t-xy^{-1})(t-\delta x^{-1}y)$, whence $r$ is irreducible over $\F$
and its splitting field is $\L$.
We can therefore assume that $\L$ is the subalgebra of $\Mat_2(\F)$ generated by $C(r)$.

Let $n \in \N$. We distinguish between two cases.

\vskip 2mm
\noindent \textbf{Case 1. $\L$ is separable over $\F$.} \\
By Theorem \ref{theoPandQsplitSingleroots}, the direct sum
$\bigl(xy^{-1} I_{n+1}+C(t^{n+1})\bigr) \oplus \bigl(\delta yx^{-1} I_{n}+C(t^{n})\bigr)$ in $\Mat_{2n+1}(\L)$
is a $(p,q)$-quotient. Viewing every entry of this matrix as a $2$-by-$2$ matrix with entries in $\F$,
we gather that, for some elements $R_1$ and $R_2$ of $\Mat_2(\F)$ that are annihilated by $r$,
the matrix
$$S:=\underbrace{
\begin{bmatrix}
R_1 & 0_2 & \cdots & \cdots & (0) \\
I_2 & R_1 & \ddots & & \vdots \\
0_2 & \ddots & \ddots & \ddots & \vdots \\
\vdots & & \ddots & R_1 & 0_2 \\
(0) & \cdots & 0_2 & I_2 & R_1
\end{bmatrix}}_{\in \Mat_{2n+2}(\F)} \oplus
\underbrace{
\begin{bmatrix}
R_2 & 0_2 & \cdots & \cdots & (0) \\
I_2 & R_2 & \ddots & & \vdots \\
0_2 & \ddots & \ddots & \ddots & \vdots \\
\vdots & & \ddots & R_2 & 0_2 \\
(0) & \cdots & 0_2 & I_2 & R_2
\end{bmatrix}}_{\in \Mat_{2n}(\F)}$$
is a $(p,q)$-quotient.
The polynomial $r$ is irreducible over $\F$ and splits over $\L$, whence $r$ is separable over $\F$. It then follows from
Proposition \ref{blockcyclicprop} that $S$ is similar to $C(r^{n+1}) \oplus C(r^n)$, and hence this last matrix is a $(p,q)$-quotient.

\vskip 2mm
\noindent \textbf{Case 2. $\L$ is inseparable over $\F$.} \\
Then, $\tr(p)=\tr(q)=0$.
Over $\L$, the polynomials $p$ and $q$ are split with sole respective roots $x$ and $y$. By Theorem \ref{theoPandQsplitSoleRoot}, the matrix
$xy^{-1} I_{2n+1}+C(t^{2n+1})$ of $\Mat_{2n+1}(\L)$ is a $(p,q)$-quotient.
Viewing $xy^{-1}$ as a matrix $R\in \Mat_2(\F)$ that is annihilated by $r$, we gather that the matrix
$$T:=\begin{bmatrix}
R & 0_2 & \cdots & \cdots & (0) \\
I_2 & R & \ddots & & \vdots \\
0_2 & \ddots & \ddots & \ddots & \vdots \\
\vdots & & \ddots & R & 0_2 \\
(0) & \cdots & 0_2 & I_2 & R
\end{bmatrix}$$
of $\Mat_{4n+2}(\F)$ is a $(p,q)$-quotient.
As the root of $r$ has multiplicity $2$, we deduce from Proposition \ref{blockcyclicprop} that
$T \simeq C(r^{n+1}) \oplus C(r^n)$, which proves that the latter matrix is a $(p,q)$-quotient.
\end{proof}

\subsubsection{On the case when $p=q$}

\begin{lemma}\label{qstablep=qlemma}
Let $p$ be a monic polynomial with degree $2$ such that $p(0) \neq 0$, and let $a$ and $b$ be
endomorphisms of a finite-dimensional vector space $V$ such that $p(a)=p(b)=0$. Then, $\Ker(ab^{-1}-\id_V)$ is stable under $a$ and $b$.
\end{lemma}

\begin{proof}
Since $V$ is finite-dimensional and $a$ and $b$ are automorphisms of $V$, $a$ is a polynomial in $a^{-1}$, and $b$ is a polynomial in $b^{-1}$. Hence, it suffices to prove that $\Ker(ab^{-1}-\id_V)$ is stable under $a^{-1}$ and $b^{-1}$.

Note that $\Ker(ab^{-1}-\id_V)=\Ker(b^{-1}-a^{-1})$.
Let us write $p(t)=t^2-\lambda t+\alpha$. Let $x \in \Ker(b^{-1}-a^{-1})$ and set $y:=a^{-1}(x)=b^{-1}(x)$.
Then,
$$a^{-1}(y)=a^{-2}(x)=\alpha^{-1} (\lambda a^{-1}(x)-x)=\alpha^{-1} (\lambda b^{-1}(x)-x)=b^{-2}(x)=b^{-1}(y).$$
Hence, $y \in \Ker (b^{-1}-a^{-1})$.
This proves the claimed statement.
\end{proof}

\begin{prop}\label{qstablep=qeven}
Let $p$ be an irreducible monic polynomial with degree $2$, and let $a$ and $b$ be
endomorphisms of a finite-dimensional vector space $V$ such that $p(a)=p(b)=0$.
Then, for every integer $k \geq 1$, the endomorphism $ab^{-1}-\id_V$ has an even number of Jordan cells of size $k$
for the eigenvalue $0$.
\end{prop}

\begin{proof}
Classically, this amounts to proving that $\Ker\bigl((ab^{-1}-\id_V)^n\bigr)$ is even-dimensional for all $n \in \N$.

Fix $n \in \N$. We note that
$$\Ker\bigl((ab^{-1}-\id_V)^{2n}\bigr)=\Ker\bigl((ab^{-1})^{-n}(ab^{-1}-\id_V)^{2n}\bigr)=\Ker\bigl((ab^{-1}+(ab^{-1})^{-1}-2\,\id_V)^n\bigr).$$
Yet, by the Commutation Lemma, we know that $a$ and $b$ commute with $ab^{-1}+(ab^{-1})^{-1}$, and hence with $ab^{-1}+(ab^{-1})^{-1}-2\,\id_V$.
It follows that $\Ker \bigl((ab^{-1}-\id_V)^{2n}\bigr)$ is stable under both $a$ and $b$.
The endomorphism of $\Ker \bigl((ab^{-1}-\id_V)^{2n}\bigr)$
induced by $a$ is annihilated by $p$, which is irreducible with degree $2$, and hence the dimension of
$\Ker \bigl((ab^{-1}-\id_V)^{2n}\bigr)$ is a multiple of $2$.

Next, $a$ and $b$ induce automorphisms $a'$ and $b'$ of the quotient space
$\Ker \bigl((ab^{-1}-\id_V)^{2n+2}\bigr)/\Ker \bigl((ab^{-1}-\id_V)^{2n}\bigr)$, and the kernel of $a'(b')^{-1}-\id_V$ is the subspace
$\Ker \bigl((ab^{-1}-\id_V)^{2n+1}\bigr)/\Ker \bigl((ab^{-1}-\id_V)^{2n}\bigr)$.
Hence, by Lemma \ref{qstablep=qlemma}, this subspace is stable under $a'$. Again, since $p(a')=0$ we deduce that the dimension of
$\Ker \bigl((ab^{-1}-\id_V)^{2n+1}\bigr)/\Ker \bigl((ab^{-1}-\id_V)^{2n}\bigr)$ is even, and hence the one of
$\Ker \bigl((ab^{-1}-\id_V)^{2n+1}\bigr)$ is also even.
\end{proof}

Now, we prove a converse statement:

\begin{lemma}\label{qCSp=q}
Let $p$ be an irreducible monic polynomial with degree $2$ over $\F$, and $n$ be a positive integer.
Then, the matrix $C\bigl((t-1)^n\bigr) \oplus C\bigl((t-1)^n\bigr)$ is a $(p,p)$-quotient.
\end{lemma}

Note that if $n$ is even this result is a special case of the Duplication Lemma.

\begin{proof}
Since $C(p)$ is annihilated by $p$, which is irreducible with degree $2$, the $\F$-algebra $\L:=\F\bigl[C(p)\bigr]$
is a splitting field of $p$.
Hence, by Theorem \ref{qtheoPandQsplitSoleRoot} if $p$ is inseparable, and by Theorem \ref{qtheoPandQsplitSingleroots} otherwise,
we know that the matrix
$$
\begin{bmatrix}
1_\L & 0_\L & \cdots & \cdots & (0) \\
1_\L & 1_\L & \ddots & & \vdots \\
0_\L & \ddots & \ddots & \ddots & \vdots \\
\vdots & & \ddots & 1_\L & 0_\L \\
(0) & \cdots & 0_\L & 1_\L & 1_\L
\end{bmatrix}$$
of $\Mat_n(\L)$ is a $(p,p)$-quotient.
Hence, the matrix
$$M_n:=\begin{bmatrix}
I_2 & 0_2 & \cdots & \cdots & (0) \\
I_2 & I_2 & \ddots & & \vdots \\
0_2 & \ddots & \ddots & \ddots & \vdots \\
\vdots & & \ddots & I_2 & 0_2 \\
(0) & \cdots & 0_2 & I_2 & I_2
\end{bmatrix}$$
of $\Mat_{2n}(\F)$ is a $(p,p)$-quotient. By permuting the basis vectors, one sees that $M_n$ is similar to $C\bigl((t-1)^n\bigr) \oplus C\bigl((t-1)^n\bigr)$, and hence
this last matrix is a $(p,p)$-quotient.
\end{proof}

\subsubsection{Conclusion}

We are now ready to conclude the study of the case when $p$ and $q$ are irreducible with the same splitting field.
There are five different subcases to consider:
\begin{itemize}
\item $q$ is not homothetic to $p$, and $p$ is separable;
\item $q$ is not homothetic to $p$, and $p$ is inseparable;
\item $q$ is homothetic to $p$ and $\tr p \neq 0$;
\item $q$ is homothetic to $p$, and $\tr p=0$ and $\car(\F) \neq 2$;
\item $q$ is homothetic to $p$, and $\tr p=0$ and $\car(\F)=2$.
\end{itemize}

\begin{theo}\label{qtheoSamesplitNotHomotheticSeparable}
Let $p$ and $q$ be irreducible monic polynomials with degree $2$ over $\F$, and assume that they have the same splitting field $\L$. Assume further that $q$ is not homothetic to $p$ (over $\F$) and that $p$ is separable, and set $\delta:=p(0)q(0)^{-1}$.
Denote by $x_1,x_2$ (respectively, by $y_1,y_2$) the roots of $p$ (respectively, of $q$) in $\L$.
Let $u$ be an endomorphism of a finite-dimensional vector space $V$ over $\F$. Then, the following conditions are equivalent:
\begin{enumerate}[(i)]
\item $u$ is a q-exceptional $(p,q)$-quotient.
\item The minimal polynomial of $u$ is the product of a power of $t^2-\Tr_{\L/\F}(x_1y_1^{-1})\,t+\delta$
with a power of $t^2-\Tr_{\L/\F}(x_1y_2^{-1})\,t+\delta$.
Moreover, the invariant factors of $u$ read $r_1,r_2,\dots,r_{2k-1},r_{2k},\dots$ where, for every positive integer $k$,
there are integers $a$ and $b$ in $\{0,1\}$ such that
$$r_{2k-1}=r_{2k} \times \bigl(t^2-\Tr_{\L/\F}(x_1y_1^{-1})\,t+\delta\bigr)^a \bigl(t^2-\Tr_{\L/\F}(x_1y_2^{-1})\,t+\delta\bigr)^b.$$
\item In some basis of $V$, the endomorphism $u$ is represented by a block-diagonal matrix in which every diagonal block
equals
$$C\bigl((t^2-\Tr_{\L/\F}(xy^{-1})\,t+\delta)^{k+\epsilon}\bigr)\oplus C\bigl((t^2-\Tr_{\L/\F}(\delta x^{-1}y)\,t+\delta)^k\bigr)$$
for some root $x$ of $p$, some root $y$ of $q$, and some pair $(k,\epsilon)\in \N \times \{0,1\}$.
\end{enumerate}
\end{theo}

\begin{proof}
With the study from the start of Section \ref{q-exceptionalsectionII}, we know that $r_1:=t^2-\Tr_{\L/\F}(x_1y_1^{-1})\,t+\delta$ and
$r_2:=t^2-\Tr_{\L/\F}(x_1y_2^{-1})\,t+\delta$
are the distinct monic irreducible factors of $G_{p,q}$, and hence $u$ is q-exceptional with respect to $(p,q)$
if and only if the irreducible monic factors of its minimal polynomial are among $r_1$ and $r_2$.

Next, it is obvious that condition (ii) implies condition (iii).
Moreover, by Lemma \ref{qsamesplitconstructivelemma} and the Duplication Lemma, we know that condition (iii) implies condition (i).

In order to conclude, we assume that condition (i) holds and we prove that condition (ii) holds.
For $j \in \{1,2\}$ and $k \in \N^*$, denote by $n_k^{(j)}$ the number of diagonal blocks
equal to $C\bigl((t^2-\Tr_{\L/\F}(x_1y_j^{-1})\,t+\delta)^l\bigr)$, for some $l \geq k$, in the primary canonical form of $u$.
Set $v:=q(0)u+p(0)u^{-1}$ and choose endomorphisms $a$ and $b$ of $V$ such that $u=ab^{-1}$ and $p(a)=q(b)=0$.
Since $\Tr_{\L/\F}(x_1y_1) \neq \Tr_{\L/\F}(x_1y_2)$, we have
$V=V_1 \oplus V_2$ where $V_j$ denotes the characteristic subspace of $v$ for the eigenvalue $\Tr_{\L/\F}(x_1y_{3-j})$.
As $a$ and $b$ commute with $v$, we find that they stabilize $V_1$ and $V_2$.
Fixing $j \in \{1,2\}$ and denoting by $u_j$ (respectively, by $v_j$) the endomorphism of $V_j$ induced by $u$ (respectively, by $v$), we get that
the minimal polynomial of $u_j$ is a power of $r_j$ and that $u_j$ is a $(p,q)$-quotient.
Moreover, for all $k \in \N^*$ and all $j \in \{1,2\}$, we have
$$2 n_k^{(j)}=n_k\bigl(v_j,\Tr_{\L/\F}(x_1y_{3-j})\bigr).$$
Hence, by Lemma \ref{CNsamesplit}, we find that if $n_{k+1}^{(j)}$ is odd then $n_{k}^{(j)}>n_{k+1}^{(j)}$.

Finally, let us consider the invariant factors $s_1,\dots,s_k,\dots$ of $u$.
Since $u$ is q-exceptional with respect to $(p,q)$ we know that each $s_k$ is the product of a power of $r_1$ with a power of $r_2$.
Let $k \in \N^*$ and write $s_{2k-1}=r_1^{N_1} r_2^{N_2}$ and $s_{2k}=r_1^{N'_1} r_2^{N'_2}$,
so that $N_1 \geq N'_1$ and $N_2 \geq N'_2$.
If $N_1 > N'_1+1$ then $n_{N'_1+1}^{(1)}=n_{N'_1+2}^{(1)}=2k-1$ and we have a contradiction with the above result.
Hence, $N_1 \in \{N'_1+1,N'_1\}$. Likewise, we obtain $N_2 \in \{N'_2+1,N'_2\}$.
This yields condition (ii), which completes the proof.
\end{proof}

\begin{theo}\label{qtheoSamesplitNotHomotheticInseparable}
Let $p$ and $q$ be irreducible monic polynomials with degree $2$ over $\F$, and assume that they have the same splitting field $\L$. Assume further that $q$ is not homothetic to $p$ (over $\F$) and that $p$ is inseparable.
Write $p=t^2+\alpha$ and $q=t^2+\beta$ for some $(\alpha,\beta)\in \F^2$.
Let $u$ be an endomorphism of a finite-dimensional vector space over $\F$. Then, the following conditions are equivalent:
\begin{enumerate}[(i)]
\item $u$ is a q-exceptional $(p,q)$-quotient.
\item The minimal polynomial of $u$ is a power of $t^2+\alpha\beta^{-1}$. Moreover,
the invariant factors of $u$ read $r_1,r_2,\dots,r_{2k-1},r_{2k},\dots$ where, for every positive integer $k$,
we have $r_{2k-1}=r_{2k}$ or $r_{2k-1}=(t^2+\alpha\beta^{-1})r_{2k}$.
\item In some basis of $V$, the endomorphism $u$ is represented by a block-diagonal matrix in which every diagonal block
equals
$$C\bigl((t^2+\alpha\beta^{-1})^{n+\epsilon}\bigr) \oplus C\bigl((t^2+\alpha\beta^{-1})^n\bigr)$$
for some $n \in \N$ and some $\epsilon \in \{0,1\}$.
\end{enumerate}
\end{theo}

\begin{proof}
Here, $G_{p,q}=(t^2+\alpha\beta^{-1})^2$, $\Theta_{p,q}=t^2$ and $t^2+\alpha\beta^{-1}$ is irreducible over $\F$ because $q$ is not
homothetic to $p$.
From there, the proof is essentially similar to the one of Theorem \ref{qtheoSamesplitNotHomotheticSeparable}, the only difference being that, in the proof
that (i) implies (ii), the endomorphism $v:=q(0)\,u+p(0)\,u^{-1}$ has $0$ as its sole eigenvalue.
\end{proof}

\begin{theo}
Let $p$ and $q$ be irreducible monic polynomials with degree $2$ over $\F$.
Assume that $q=H_d(p)$ for some $d \in \F \setminus \{0\}$, and that $\tr p \neq 0$.
Set $r:=t^2-\frac{\tr p \tr q-2d q(0)}{q(0)}\,t+\frac{p(0)}{q(0)}$.
Let $u$ be an endomorphism of a finite-dimensional vector space $V$ over $\F$. Then, the following conditions are equivalent:
\begin{enumerate}[(i)]
\item $u$ is a q-exceptional $(p,q)$-quotient.
\item The minimal polynomial of $u$ is the product of a power of $t-d$ with a power of $r$.
Moreover, the invariant factors of $u$ read $s_1,s_2,\dots,s_{2k-1},s_{2k},\dots$ where, for every positive integer $k$,
we have $s_{2k-1}=s_{2k}$ or $s_{2k-1}=r\,s_{2k}$.
\item In some basis of $V$, the endomorphism $u$ is represented by a block-diagonal matrix in which every diagonal block
equals either $C(r^{n+\epsilon}) \oplus C(r^n)$ for some $n \in \N$ and some $\epsilon \in \{0,1\}$, or
$C\bigl((t-d)^n\bigr) \oplus C\bigl((t-d)^n\bigr)$ for some $n \in \N^*$.
\end{enumerate}
\end{theo}

\begin{proof}
As we have seen in the beginning of Section \ref{q-exceptionalsectionIII},
$G_{p,q}(t)=(t-d)^2 r$ and $r$ is irreducible.
Moreover, $r=t^2-\Tr_{\L/\F}(dx_1x_2^{-1})\,t+\delta$, where $\L$ denotes the splitting field of
$p$, and $p(t)=(t-x_1)(t-x_2)$ and $\delta:=p(0)q(0)^{-1}$.

Moreover, as we can safely replace $u$ with $d^{-1}u$, we lose no generality in assuming that $p=q$, in which case $d=1$.
From there, the implication (ii) $\Rightarrow$ (iii) is obvious, and (iii) $\Rightarrow$ (i)
is readily obtained by applying the Duplication Lemma together with Lemmas \ref{qsamesplitconstructivelemma} and \ref{qCSp=q}.

Assume finally that condition (i) holds. For $k \in \N^*$, denote in the primary canonical form of $u$, by $n_k$
the number of blocks of type $C\bigl((t-1)^k\bigr)$, and by $m_k$
the number of blocks of type $C(r^l)$ for some $l \geq k$.

By Proposition \ref{qstablep=qeven}, we find that $n_k$ is even for all $k \in \N^*$.

Noting that $\Theta_{p,q}=\bigl(t-\Tr_{\L/\F}(x_1^2)\bigr)\bigl(t-\Tr_{\L/\F}(x_1x_2)\bigr)$, we
use the same line of reasoning as in the proof of Theorem \ref{qtheoSamesplitNotHomotheticSeparable} to gather that
$m_{k}>m_{k+1}$ whenever $m_{k+1}$ is odd. The derivation of (ii) is then achieved as in the proof of
Theorem \ref{qtheoSamesplitNotHomotheticSeparable}.
\end{proof}

\begin{theo}\label{qtheoSamesplitOppositeRoots}
Let $p$ and $q$ be irreducible monic polynomials with degree $2$ over $\F$.
Assume further that $q=H_d(p)$ for some $d \in \F \setminus \{0\}$, and that $\tr p=0$ and $\car(\F) \neq 2$.
Let $u$ be an endomorphism of a finite-dimensional vector space $V$ over $\F$. Then, the following conditions are equivalent:
\begin{enumerate}[(i)]
\item $u$ is a q-exceptional $(p,q)$-quotient.
\item The minimal polynomial of $u$ is the product of a power of $t-d$ with a power of $t+d$.
Moreover, the invariant factors of $u$ read $s_1,s_2,\dots,s_{2k-1},s_{2k},\dots$ where, for every positive integer $k$,
we have $s_{2k-1}=s_{2k}$.
\item In some basis of $V$, the endomorphism $u$ is represented by a block-diagonal matrix in which every diagonal block
equals $C\bigl((t-d)^n\bigr) \oplus C\bigl((t-d)^n\bigr)$ or
$C\bigl((t+d)^n\bigr) \oplus C\bigl((t+d)^n\bigr)$ for some $n \in \N^*$.
\end{enumerate}
\end{theo}

\begin{proof}
Here, we see that $G_{p,q}=(t-d)^2(t+d)^2$ and $\Theta_{p,q}=\bigl(t-2dp(0)\bigr)\bigl(t+2dp(0)\bigr)$.
From there, the proof is similar to the ones of the previous three theorems.
\end{proof}

With a similar proof, we obtain:

\begin{theo}\label{qtheoSamesplitTranslationOppositeRootsCar2}
Let $p$ and $q$ be irreducible monic polynomials with degree $2$ over $\F$.
Assume further that $q=H_d(p)$ for some $d \in \F \setminus \{0\}$, and that $\tr p=0$ and $\car
(\F)=2$.
Let $u$ be an endomorphism of a finite-dimensional vector space $V$ over $\F$. Then, the following conditions are equivalent:
\begin{enumerate}[(i)]
\item $u$ is a q-exceptional $(p,q)$-quotient.
\item The minimal polynomial of $u$ is a power of $t-d$.
Moreover, the invariant factors of $u$ read $s_1,s_2,\dots,s_{2k-1},s_{2k},\dots$ where, for every positive integer $k$,
we have $s_{2k-1}=s_{2k}$.
\item In some basis of $V$, the endomorphism $u$ is represented by a block-diagonal matrix in which every diagonal block
equals $C\bigl((t-d)^n\bigr) \oplus C\bigl((t-d)^n\bigr)$ for some $n \in \N^*$.
\end{enumerate}
\end{theo}

Using the above five theorems, it is easy to derive the classification of
q-exceptional $(p,q)$-quotients when $p$ and $q$ are both irreducible and have the same splitting field,
as given in Table \ref{qfigure7}.

\subsection{Exceptional $(p,q)$-quotients (IV): When $p$ and $q$ are irreducible with distinct splitting fields}\label{q-exceptionalsectionIV}

In this final section, we tackle the case when both polynomials $p$ and $q$ are irreducible with distinct
splitting fields in $\overline{\F}$.
The discussion will basically be split into three subcases, whether both $p$ and $q$ are separable,
both are inseparable, or exactly one of them is separable.
In the first two cases, one basic trick will be to extend the field of scalars by using $\Theta_{p,q}$ and
the Jordan-Chevalley decomposition of $v:=q(0)u+p(0)u^{-1}$,
which will allow us to use the results from Section \ref{q-exceptionalsectionIII}.

\subsubsection{Case 1. Both $p$ and $q$ are separable}\label{qbothPandQseparableDistinctsplittingfields}

In that case, the splitting field $\L$ of $pq$ in $\overline{\F}$ is known to be a Galois extension of $\F$
with degree $4$. Moreover, the Galois group of $\L$ over $\F$ contains two elements $\sigma$ and $\tau$
such that $\sigma$ exchanges the two roots of $q$ and fixes the ones of $p$,
and $\tau$ exchanges the two roots of $p$ and fixes the ones of $q$. It follows that $\Gal(\L/\F)$
acts transitively on $\Root(p)\Root(q)^{-1}$, whence $G_{p,q}$ is a power of some irreducible monic polynomial of
$\F[t]$. Once more, we set $\delta:=p(0)q(0)^{-1}$.

Let us split $p=(t-x_1)(t-x_2)$ and $q=(t-y_1)(t-y_2)$ in $\L[t]$.
If $x_1y_1^{-1},x_1y_2^{-1},x_2y_1^{-1},x_2y_2^{-1}$ are pairwise distinct, then $G_{p,q}$
is irreducible over $\F$ and separable.

\begin{lemma}\label{qFpqirrLemma}
The polynomial $\Theta_{p,q}$ is irreducible (and separable) over $\F$.
The polynomial $G_{p,q}$ is irreducible over $\F$ if and only if $(\tr p,\tr q) \neq (0,0)$.
Moreover, if $\tr p=\tr q=0$ then $G_{p,q}=(t^2-\delta)^2$, and $\Theta_{p,q}=t^2-4p(0)q(0)$.
\end{lemma}

\begin{proof}
Note first that the Galois group of $\L$ over $\F$ acts transitively on $\{x_1y_1+x_2y_2,x_1y_2+x_2y_1\}$.
Moreover, $(x_1y_1+x_2y_2)-(x_1y_2+x_2y_1)=(x_1-x_2)(y_1-y_2) \neq 0$. Hence, $\Theta_{p,q}$
is irreducible over $\F$.

Assume that $G_{p,q}$ is reducible over $\F$.
Remember that $G_{p,q}$ has no root in $\F$ and that the Galois group of $\L$ over $\F$ acts transitively on the set of its roots.
Hence, two scalars among $x_1y_1^{-1},x_1y_2^{-1},x_2y_1^{-1},x_2y_2^{-1}$
are equal. Since $y_1 \neq y_2$, $x_1 \neq x_2$, and the $x_i$'s and the $y_j$'s are all nonzero, it follows that $x_1y_1^{-1}=x_2y_2^{-1}$ or $x_1y_2^{-1}=x_2y_1^{-1}$.
Applying the $\sigma$ automorphism, we see that both equalities actually hold, whence $y_1 y_2^{-1}=x_1 x_2^{-1}=y_2 y_1^{-1}$. Hence $y_2=\pm y_1$, and since $y_1 \neq y_2$ we obtain $y_2=-y_1$. Likewise $x_2=-x_1$ and we conclude that
$\tr p=\tr q=0$.

Conversely, assume that $\tr p=\tr q=0$.
Then, $y_2=-y_1$, whence, for all $i \in \{1,2\}$,
$$(t-x_i y_1^{-1})(t-x_i y_2^{-1})=t^2+\frac{x_i^2}{q(0)}=t^2-p(0)q(0)^{-1}=t^2-\delta$$
and hence
$$G_{p,q}=(t^2-\delta)^2.$$
Finally,
$$\Theta_{p,q}=(t-2x_1y_1)(t-2x_1y_2)=t^2-4p(0)q(0).$$
\end{proof}

Note that if $\tr p=\tr q=0$ then $\car(\F) \neq 2$ because we have assumed that $p$ and $q$ must be separable.

Now, we distinguish between two cases, whether $(\tr p,\tr q)$ equals $(0,0)$ or not.

\paragraph{Case 1.1. $(\tr p,\tr q) \neq (0,0)$}

\begin{theo}\label{qtheoFpq4roots}
Let $p$ and $q$ be monic polynomials with degree $2$ over $\F$,
and assume that they are both irreducible with distinct splitting fields and that $(\tr p,\tr q) \neq (0,0)$.
Let $u$ be an endomorphism of a finite-dimensional vector space $V$ over $\F$. Then, the following conditions are equivalent:
\begin{enumerate}[(i)]
\item The endomorphism $u$ is a q-exceptional $(p,q)$-quotient.
\item The minimal polynomial of $u$ is a power of $G_{p,q}$ and if we denote by $r_1,\dots,r_k,\dots$
the invariant factors of $u$, then $r_{2k-1}=r_{2k}$ or $r_{2k-1}=r_{2k}\, G_{p,q}$, for all $k \in \N^*$.
\item In some basis of $V$, the endomorphism $u$ is represented by a block-diagonal matrix in which
each diagonal block equals $C\bigl(G_{p,q}^n\bigr) \oplus C\bigl(G_{p,q}^n\bigr)$
or $C\bigl(G_{p,q}^n\bigr) \oplus C\bigl(G_{p,q}^{n-1}\bigr)$ for some positive integer $n$.
\end{enumerate}
\end{theo}

\begin{proof}
It is easily seen that condition (ii) implies condition (iii).
Conversely, assume that (iii) holds but that (ii) does not.
First of all, it is obvious from condition (iii) that all the invariant factors of $u$ are powers of $G_{p,q}$.
Next, there is a least positive integer $k$ such that $G_{p,q}^2 r_{2k}$ divides $r_{2k-1}$.
Write then $r_{2k-1}=G_{p,q}^{\ell+1}$ for some integer $\ell \geq 0$. Then,
in any decomposition given by condition (iii), there is no diagonal block of the form
$C\bigl(G_{p,q}^m\bigr) \oplus C\bigl(G_{p,q}^{n}\bigr)$ in which one of $m$ and $n$ equals $\ell$.
Hence, each one of those diagonal blocks equals either $C\bigl(G_{p,q}^n\bigr) \oplus C\bigl(G_{p,q}^n\bigr)$
for some integer $n \neq \ell$, or $C\bigl(G_{p,q}^n\bigr) \oplus C\bigl(G_{p,q}^{n-1}\bigr)$
for some positive integer $n$ distinct from $\ell$ and $\ell+1$; in the latter case either both $n$ and $n-1$ are greater than $\ell$, or both
are less than $\ell$. It follows that there is an even number of invariant factors of $u$ that equal $G_{p,q}^n$ for some $n>\ell$, contradicting
the fact that there are $2k-1$ such invariants factors. Hence, condition (iii) implies condition (ii).

It only remains to prove that conditions (i) and (iii) are equivalent.
Some general work is required before we tackle this equivalence.

As we have seen in Lemma \ref{qFpqirrLemma}, our assumptions imply that both
$G_{p,q}$ and $\Theta_{p,q}$ are irreducible over $\F$.

Denote by $\K$ the splitting field of $\Theta_{p,q}$ in $\L$. Without loss of generality, we can consider that
$\K$ is the subalgebra of $\Mat_2(\F)$ generated by the companion matrix of $\Theta_{p,q}$.
Obviously $\K$ is the Galois subfield of $\L$ associated with the subgroup of $\Gal(\L/\F)$
generated by the Galois automorphism that exchanges the two roots of $p$ and that exchanges the two roots of $q$.
Hence, $p$ and $q$ remain irreducible over $\K$.

Let us split $p(t)=(t-x_1)(t-x_2)$ and $q(t)=(t-y_1)(t-y_2)$ over $\L$.
We claim that, over $\K$, the polynomials $p$ and $q$ satisfy the assumptions of Theorem \ref{qtheoSamesplitNotHomotheticSeparable}.
Indeed, it only remains to check that $p$ and $q$ are not homothetic as polynomials over $\K$;
yet, if the contrary held then we would have $x_1 x_2^{-1}=y_1 y_2^{-1}$ and hence $x_1 y_1^{-1} = x_2 y_2^{-1}$, contradicting the known fact that
$G_{p,q}$ has four distinct roots in $\L$.

Now, set
$$r:=\bigl(t-x_1y_1^{-1}\bigr)\bigl(t-x_2y_2^{-1}\bigr)\in \K[t] \quad \text{and} \quad s:=\bigl(t-x_1y_2^{-1}\bigr)\bigl(t-x_2y_1^{-1}\bigr)\in \K[t].$$
Let $n \in \N^*$.
Throughout the proof, $C(r^n)$ will be interpreted, depending on the context (which will always be obvious) either as a matrix of $\Mat_{2n}(\K)$ or as one of $\Mat_{4n}(\F)$.
Since $r$ is separable, Proposition \ref{blockcyclicprop} shows that
$C(r^n)$ is similar to
$$\begin{bmatrix}
C(r) & 0_2 & \cdots & \cdots & (0) \\
I_2 & C(r) & \ddots & & \vdots \\
0_2 & \ddots & \ddots & \ddots & \vdots \\
\vdots & & \ddots & C(r) & 0_2 \\
(0) & \cdots & 0_2 & I_2 & C(r)
\end{bmatrix} \in \Mat_{2n}(\K).$$
This matrix can be interpreted as the matrix
$$M:=\begin{bmatrix}
P & 0_4 & \cdots & \cdots & (0) \\
I_4 & P & \ddots & & \vdots \\
0_4 & \ddots & \ddots & \ddots & \vdots \\
\vdots & & \ddots & P & 0_4 \\
(0) & \cdots & 0_4 & I_4 & P
\end{bmatrix} \in \Mat_{4n}(\F)$$
for some $P \in \Mat_4(\F)$ that is annihilated by $G_{p,q}$.
Since $G_{p,q}$ is irreducible and separable, it follows once more from Proposition \ref{blockcyclicprop}
that $M$ is similar to $C(G_{p,q}^n)$ in $\Mat_{4n}(\F)$.
Hence, $C(r^n)$ is similar to $C(G_{p,q}^n)$ in $\Mat_{4n}(\F)$.
Note that this remains (trivially) true if $n=0$, and that this remains true if $r$ is replaced with $s$.

We are now ready to prove both implications (iii) $\Rightarrow$ (i) and (i) $\Rightarrow$ (ii).

In order to prove (iii) $\Rightarrow$ (i), it suffices to show that, for all $n \in \N^*$, both matrices
$C(G_{p,q}^n) \oplus C(G_{p,q}^n)$ and  $C(G_{p,q}^n) \oplus C(G_{p,q}^{n-1})$ are $(p,q)$-quotients.
For the first one it suffices to use the Duplication Lemma, since $G_{p,q}(t)=R_\delta\bigl(H_{q(0)}(\Theta_{p,q})\bigr)$.
Next, let $n \in \N^*$. Since $p$ and $q$ satisfy the assumptions of Theorem \ref{qtheoSamesplitNotHomotheticSeparable} over $\K$, we
get that the matrix $C(r^n) \oplus C(s^{n-1})$ is a $(p,q)$-quotient in $\Mat_{4n-2}(\K)$.
Hence, it is also a $(p,q)$-quotient in $\Mat_{8n-4}(\F)$.
Yet, we have just seen that this matrix is similar to $C(G_{p,q}^n) \oplus C(G_{p,q}^{n-1})$ in $\Mat_{8n-4}(\F)$.
It follows that the latter is a $(p,q)$-quotient. Hence, (iii) $\Rightarrow$ (i) is proved.

Conversely, assume that $u$ is a q-exceptional $(p,q)$-quotient.
We choose automorphisms $a$ and $b$ of the $\F$-vector space $V$ such that $u=ab^{-1}$ and $p(a)=q(b)=0$.
The endomorphism $v:=q(0)u+p(0)u^{-1}$ is annihilated by some power of the separable polynomial $\Theta_{p,q}$.
Hence, by the Jordan-Chevalley decomposition, we have a splitting $v=S+N$ in which $S$ is a semi-simple endomorphism of $V$
that is annihilated by $\Theta_{p,q}$ and that belongs to $\F[v]$, and $N$ is a nilpotent endomorphism of $V$.
By the Commutation Lemma, $a$ and $b$ turn out to be endomorphisms of the $\F[S]$-vector space $V$, and so does $u$.
We denote by $V^S$ the $\F[S]$-vector space $V$ to differentiate it from the $\F$-vector space $V$.
Now, $\F[S]$ is isomorphic to the splitting field of $\Theta_{p,q}$ over $\F$, and hence $\F[S] \simeq \K$.
Applying Theorem \ref{qtheoSamesplitNotHomotheticSeparable} to $V^S$,
we get that the endomorphism $u$ of $V^S$ is represented by a block-diagonal matrix in which each diagonal block has
one of the forms $C(r^k) \oplus C(s^{k-1})$, $C(r^k) \oplus C(s^k)$, or $C(s^k) \oplus C(r^{k-1})$ for some positive integer $k$.
It follows from our initial study that condition (iii) is satisfied by $u$, which completes the proof.
\end{proof}

\paragraph{Case 1.2. $\tr p=\tr q=0$}

Here, $\car(\F) \neq 2$ since $p$ and $q$ are separable.
Moreover, $G_{p,q}=r^2$ for $r:=t^2-p(0)q(0)^{-1}$, and $r$ is irreducible over $\F$.
Finally, $\Theta_{p,q}=t^2-4p(0)q(0)$ is irreducible over $\F$. Note that
$r$ and $\Theta_{p,q}$ have the same splitting field in $\L$ since
$\Theta_{p,q}=H_{1/(2q(0))}(r)$.

\begin{theo}\label{qtheoFpqdoubleroots}
Let $p$ and $q$ be monic polynomials with degree $2$ over $\F$,
and assume that they are both irreducible with distinct splitting fields, that they are separable and that
$\tr p=\tr q=0$. Set $r:=t^2-p(0)q(0)^{-1}$.
Let $u$ be an endomorphism of a finite-dimensional vector space $V$ over $\F$. Then, the following conditions are equivalent:
\begin{enumerate}[(i)]
\item The endomorphism $u$ is a q-exceptional $(p,q)$-quotient.
\item The minimal polynomial of $u$ is a power of $r$, and if we denote by $r_1,\dots,r_k,\dots$
the invariant factors of $u$, then $r_{2k-1}=r_{2k}$ for all $k \in \N^*$.
\item In some basis of $V$, the endomorphism $u$ is represented by a block-diagonal matrix in which
each diagonal block equals $C(r^n) \oplus C(r^n)$ for some $n \in \N^*$.
\end{enumerate}
\end{theo}

\begin{proof}
Conditions (ii) and (iii) are obviously equivalent.

Next, we prove that conditions (i) and (iii) are equivalent.
As in the proof of Theorem \ref{qtheoFpq4roots}, some preliminary work is required.

Remember that $\car(\F)\neq 2$.
The splitting field $\K$ of $\Theta_{p,q}$ over $\L$ can be identified with
a subalgebra of $\Mat_2(\F)$. We have seen that $\K$ is also the splitting field of $r$ over $\F$.
Let us split $p(t)=(t-x)(t+x)$ and $q(t)=(t-y)(t+y)$ in $\L[t]$.
Note that $xy^{-1} \in \K$.
In particular, $p$ and $q$ are homothetic as polynomials of $\K[t]$.

Next, fix $n \in \N^*$. Set
$$A_n:=\begin{bmatrix}
xy^{-1} & 0 & \cdots & \cdots & (0) \\
1 & xy^{-1} & \ddots & & \vdots \\
0 & \ddots & \ddots & \ddots & \vdots \\
\vdots & & \ddots & xy^{-1} & 0 \\
(0) & \cdots & 0 & 1 & xy^{-1}
\end{bmatrix}\in \Mat_n(\K)$$ and
$$B_n:=\begin{bmatrix}
-xy^{-1} & 0 & \cdots & \cdots & (0) \\
1 & -xy^{-1} & \ddots & & \vdots \\
0 & \ddots & \ddots & \ddots & \vdots \\
\vdots & & \ddots & -xy^{-1} & 0 \\
(0) & \cdots & 0 & 1 & -xy^{-1}
\end{bmatrix}\in \Mat_n(\K),$$
and set further
$$M_n:=A_n \oplus A_n \quad \text{and} \quad N_n:=B_n \oplus B_n,$$
which we see as matrices with entries in $\K$.
Here, $xy^{-1}$ is a root of $r$. Seeing $\K$ as a subalgebra of $\Mat_2(\F)$, we have
$$A_n=\begin{bmatrix}
xy^{-1} & 0_2 & \cdots & \cdots & (0) \\
I_2 & xy^{-1} & \ddots & & \vdots \\
0_2 & \ddots & \ddots & \ddots & \vdots \\
\vdots & & \ddots & xy^{-1} & 0_2 \\
(0) & \cdots & 0_2 & I_2 & xy^{-1}
\end{bmatrix}.$$
Since $xy^{-1}$ is annihilated by $r$, which is separable with degree $2$,
Proposition \ref{blockcyclicprop} shows that
$A_n$ is similar to $C(r^n)$ in $\Mat_{2n}(\F)$. Likewise, $B_n$ is similar to $C(r^n)$ in $\Mat_{2n}(\F)$,
and we conclude that both $M_n$ and $N_n$ are similar to $C(r^n) \oplus C(r^n)$ in $\Mat_{4n}(\F)$.

We are now ready to conclude. For all $n \in \N^*$, we know from Theorem \ref{qtheoSamesplitOppositeRoots}
that $M_n$ is a $(p,q)$-quotient in $\Mat_{2n}(\K)$, and hence it is also a $(p,q)$-quotient in $\Mat_{4n}(\F)$.
Hence, condition (iii) implies condition (i).
Conversely, assume that $u$ is a q-exceptional $(p,q)$-quotient.
Let $a,b$ be automorphisms of $V$ such that $u=ab^{-1}$ and $p(a)=q(b)=0$.
Setting $v:=q(0)u+p(0)u^{-1}$, we see that $v$ is annihilated by some power of $\Theta_{p,q}$. Since $\Theta_{p,q}$ is separable, we can use the Jordan-Chevalley decomposition $v=S+N$ in which $S$ is semi-simple, $N$ is nilpotent and $S$ is a polynomial in $v$.
Note that $\Theta_{p,q}(S)=0$.
By the Commutation Lemma, both $a$ and $b$ commute with $S$, and hence $a$ and $b$ are endomorphisms of the $\F[S]$-vector space
$V$, which we denote by $V^S$. Hence, $u$ is a q-exceptional $(p,q)$-quotient in the algebra of all endomorphisms of $V^S$.
Yet, $\F[S] \simeq \K$, and hence, by Theorem \ref{qtheoSamesplitOppositeRoots}, in some basis of $V^S$ the endomorphism
$u$ is represented by a block-diagonal matrix in which every diagonal block equals either $C\bigl((t-xy^{-1})^k\bigr) \oplus
C\bigl((t-xy^{-1})^k\bigr)$ or $C\bigl((t+xy^{-1})^k\bigr) \oplus C\bigl((t+xy^{-1})^k\bigr)$ for some $k \in \N^*$.
Hence, there is a basis of the $\F$-vector space $V$ in which $u$ is represented by a block-diagonal
matrix in which every diagonal block equals $M_k$ or $N_k$ for some $k \in \N^*$. Our preliminary work on block matrices then yields that condition (iii) holds.
\end{proof}

\paragraph{Case 2. Both $p$ and $q$ are inseparable}

\begin{theo}\label{qtheoPandQinseparable}
Assume that $\car(\F)=2$.
Let $\alpha$ and $\beta$ be elements of $\F$, set $p(t):=t^2-\alpha$ and $q(t):=t^2-\beta$
and assume that both $p$ and $q$ are irreducible over $\F$ and that they have distinct splitting fields in $\overline{\F}$.
Set $\delta:=p(0)q(0)^{-1}=\alpha \beta^{-1}$.
Let $u$ be an endomorphism of a finite-dimensional vector space $V$. Then, the following conditions are equivalent:
\begin{enumerate}[(i)]
\item $u$ is a q-exceptional $(p,q)$-quotient.
\item Every invariant factor of $u$ is a power of $t^2-\delta$, and, if we denote by $r_1,\dots,r_k,\dots$
those invariant factors, we have $r_{2k-1}=r_{2k}$ for all $k \in \N^*$.
\item In some basis of $V$, the endomorphism $u$ is represented by a block-diagonal matrix in which every diagonal block
equals $C\bigl((t^2-\delta)^n\bigr) \oplus C\bigl((t^2-\delta)^n\bigr)$ for some $n \in \N^*$.
\end{enumerate}
\end{theo}

\begin{proof}
The equivalence between conditions (ii) and (iii) is obvious.

Before we prove that conditions (i) and (ii) are equivalent, some preliminary work is required.
First of all, here we have $G_{p,q}=(t^2-\delta)^2$. Set $r:=t^2-\delta=t^2-\alpha \beta^{-1}$.
Since $p$ and $q$ have distinct splitting fields, their respective roots $\sqrt{\alpha}$ and $\sqrt{\beta}$ in $\L$
are linearly independent over $\F$, and hence $r$ is irreducible over $\F$.
It follows that $u$ is q-exceptional with respect to $(p,q)$ if and only if it is annihilated by some power of $r$.

From there, as $t^2-\delta=R_\delta(t)$ (because $\F$ has characteristic $2$),
the Duplication Lemma yields that condition (iii) implies condition (i).
In order to conclude, we prove that condition (i) implies condition (ii).
Assume that condition (i) holds. First of all, we know that each invariant factor of $u$
is a power of $r$. Let $a$ and $b$ be endomorphisms of $V$ such that $p(a)=0=q(b)$ and $u=ab^{-1}$.
Note that $a^2=p(0)\,\id_V$ and $b^2=q(0)\,\id_V$, whence
$ab+ba=q(0)\,ab^{-1}+p(0)\,ba^{-1}=q(0)\,u+p(0)\,u^{-1}$. We deduce that $v:=ab+ba$ commutes with $u$ and that
$q(0)(u^2-\delta\,\id)=u\,v$. It follows that, for all $n \in \N^*$, we have
$r(u)^n=q(0)^{-n}\,u^n\,v^n$, whence $\Ker r(u)^n=\Ker v^n$.
Let $n \in \N^*$. Denote by $N_n$ the number of invariant factors of $u$ that equal $r^k$ for some $k \geq n$.
Then, $2N_n=\dim \Ker r(u)^n-\dim \Ker r(u)^{n-1}=\dim \Ker v^n-\dim \Ker v^{n-1}$.
By Lemma \ref{LemmaPandQinseparable}, we deduce that $N_n$ is even.
Finally, for all $n \in \N^*$, the number of invariant factors of $u$ that equal $r^n$
is $N_n-N_{n+1}$, and hence it is even. It follows that condition (ii) holds.
\end{proof}

Combining Theorems \ref{qtheoFpqdoubleroots} and \ref{qtheoPandQinseparable}, we
deduce the classification of indecomposable q-exceptional $(p,q)$-quotients
in the case when $p$ and $q$ have distinct splitting fields and $\tr p=\tr q=0$,
as given in Table \ref{qfigure9}.

\subsubsection{Case 3. $p$ is separable and $q$ is not}

This is both the last remaining case and the most difficult one. Fortunately, most of the necessary preliminary work for this case has been done in 
our study of $(p,q)$-differences. 

Here, $\car(\F)=2$. The splitting field $\L$ of $pq$ is not a Galois extension of $\F$.
Yet, it is not a radicial extension either because $p$ is irreducible and separable.
Hence, we have a decomposition $\F - \K - \L$ where $\K$ is a radicial quadratic extension of $\F$
and $\L$ is a separable extension of $\K$. Explicitly, $\K$ is the set of all $x \in \L$ such that $x^2 \in \F$.

Since $q$ is inseparable, $\K$ is its splitting field in $\L$.

Let us split $p(t)=(t-x_1)(t-x_2)$ and $q(t)=(t-y)^2$ in $\L[t]$.
Let $u$ be an endomorphism of a finite-dimensional vector space $V$, and
set $v:=q(0) u+p(0) u^{-1}$.
Here,
$$G_{p,q}(t)=\bigl(t^2-(\tr p) y^{-1} t+\delta\bigr)^2 \quad \text{and} \quad \Theta_{p,q}(t)=
\bigl(t-(x_1+x_2)y\bigr)^2=t^2-(\tr p)^2 y^2.$$
Note that $G_{p,q}$ is irreducible: indeed, it is split over $\L$, the Galois group of $\L$ over $\F$ acts transitively on the set $\{x_1y^{-1},x_2y^{-1}\}$
of its roots in $\L$, and hence the only possible monic irreducible proper divisor of $G_{p,q}$ would be $(t-x_1y^{-1})(t-x_2y^{-1})=
t^2-(\tr p) y^{-1} t+\delta$; Yet, since $p$ is separable and irreducible we have $\tr p \neq 0$, whereas $y \not\in \F$, whence
$(t-x_1y^{-1})(t-x_2y^{-1})$ does not belong to $\F[t]$.

Hence, the minimal polynomial of $u$ is a power of $G_{p,q}$ if and only if $u$ is q-exceptional with respect to $(p,q)$.

With that result in mind, we are now ready to classify the q-exceptional $(p,q)$-quotients.

\begin{theo}\label{qtheolastcase}
Let $p$ and $q$ be irreducible monic polynomials with degree $2$ over $\F$.
Assume that $p$ is separable and that $q$ is inseparable.
Let $u$ be an endomorphism of a finite-dimensional vector space $V$ over $\F$.
Then, the following conditions are equivalent:
\begin{enumerate}[(i)]
\item $u$ is a q-exceptional $(p,q)$-quotient.
\item The minimal polynomial of $u$ is a power of $G_{p,q}$ and
the invariant factors of $u$ read $r_1,\dots,r_k,\dots$, where, for each positive integer $k$,
either $r_{2k}=r_{2k-1}$ or $r_{2k-1}=r_{2k}\,G_{p,q}$.
\item In some basis of $V$, the endomorphism $u$ is represented by a block-diagonal matrix
in which every diagonal block equals $C(G_{p,q}^{n+\epsilon}) \oplus C(G_{p,q}^n)$ for some non-negative integer
$n$ and some $\epsilon \in \{0,1\}$.
\end{enumerate}
\end{theo}

\begin{proof}
Let us write $q=(t-y)^2=t^2-y^2$.

The equivalence between conditions (ii) and (iii) is obtained in exactly the same way as in the proof of Theorem \ref{qtheoFpq4roots}.

Next, we prove that condition (iii) implies condition (ii).
Obviously, it suffices to fix a positive integer $n$ and to prove that both matrices
$C(G_{p,q}^{n}) \oplus C(G_{p,q}^{n-1})$ and $C(G_{p,q}^{n}) \oplus C(G_{p,q}^n)$ are $(p,q)$-quotients
(for the second one, this can be directly obtained as a consequence of the Duplication Lemma, but we will give a general proof that encompasses both cases).
First, fix a positive integer $k$. Without loss of generality, we can assume that $\K=\F[C(q)]$.
Over $\K$, the polynomial $q$ splits with a double root, whereas $p$ remains irreducible and separable.
Hence, by Theorem \ref{qtheoQsplitswithdouble} the matrix $C\bigl(H_y(p)^k\bigr)$ of $\Mat_{2k}(\K)$ is a $(p,q)$-quotient.
Since $H_y(p)$ is separable over $\K$, Proposition \ref{blockcyclicprop} yields that the
matrix
$$M_k=\begin{bmatrix}
C\bigl(H_y(p)\bigr) & 0_2 & \cdots & \cdots & (0) \\
I_2 & C\bigl(H_y(p)\bigr) & \ddots & & \vdots \\
0_2 & \ddots & \ddots & \ddots & \vdots \\
\vdots & & \ddots & C\bigl(H_y(p)\bigr) & 0_2 \\
(0) & \cdots & 0_2 & I_2 & C\bigl(H_y(p)\bigr)
\end{bmatrix}\in \Mat_{2k}(\K)$$
is similar to $C(H_y(p)^k)$, and is therefore a $(p,q)$-quotient.

Viewing $C\bigl(H_y(p)\bigr)$ as a matrix $P$ of $\Mat_4(\F)$, we deduce that $G_{p,q}$ annihilates $P$ and that the matrix
$$\begin{bmatrix}
P & 0_4 & \cdots & \cdots & (0) \\
I_4 & P & \ddots & & \vdots \\
0_4 & \ddots & \ddots & \ddots & \vdots \\
\vdots & & \ddots & P & 0_4 \\
(0) & \cdots & 0_4 & I_4 & P
\end{bmatrix}\in \Mat_{4k}(\F)$$
is a $(p,q)$-quotient.
Here $G_{p,q}$ is irreducible with double roots. Therefore, Proposition \ref{blockcyclicprop}
yields that the above matrix of $\Mat_{4k}(\F)$ is similar to $C(G_{p,q}^{s+\epsilon}) \oplus C(G_{p,q}^s)$,
where $s$ and $\epsilon$ respectively denote the quotient and the remainder of $k$ modulo $2$.
Varying $k$ then yields the claimed result, and we conclude that condition (iii) implies condition (i).

It remains to prove that condition (i) implies condition (ii).
Assume therefore that $u$ is a q-exceptional $(p,q)$-quotient.
Since $G_{p,q}$ is monic and irreducible, the minimal polynomial of $u$ must be a power of it.
From there, we set $v:=q(0) u +p(0) u^{-1}$. Let $a,b$ be automorphisms of $V$ such that $u=ab^{-1}$ and $p(a)=q(b)=0$.
Note that $\tr p \neq 0$ since $\F$ has characteristic $2$ and $p$ is inseparable.
Set $a':=(\tr p)^{-1}\,a$ and note that $a'$ is annihilated by $p_1:=t^2+t+\frac{p(0)}{(\tr p)^2}\cdot$
Denoting by $a^\star$ the $p$-conjugate of $a$, by $b^\star$ the $q$-conjugate of $b$ and by
$(a')^\star$ the $p_1$-conjugate of $a'$, one sees that $a b^\star+b a^\star=(\tr p) (a' b^\star+b (a')^\star)$,
and hence by setting $w:=a' b^\star+b (a')^\star$, we have
$$v=a b^\star+b a^\star=(\tr p)\, w.$$
It follows that
$$q(w)=(\tr p)^{-2}\,(v^2-(\tr p)^2 y^2\id)=(\tr p)^{-2}\Theta_{p,q}(v)$$
and
$$G_{p,q}(u)=q(0)^{-2} (\tr p)^2 u^2 q(w).$$
Since $u$ commutes with $v$ and hence with $w$, we obtain, for all $k \in \N$,
$$G_{p,q}(u)^k=q(0)^{-2k} (\tr p)^{2k} u^{2k} q(w)^k$$
whence
$$\Ker G_{p,q}(u)^k=\Ker q(w)^k.$$
Then, Proposition \ref{lastprop} yields that, for all $k \in \N$, if $\dim \bigl(\Ker G_{p,q}(u)^{k+1}/\Ker G_{p,q}(u)^k\bigr)=\dim \bigl(\Ker G_{p,q}(u)^{k+2}/\Ker G_{p,q}(u)^{k+1}\bigr)$,
then $\dim \bigl(\Ker G_{p,q}(u)^{k+1}/\Ker G_{p,q}(u)^k\bigr)$ is a multiple of $8$.

From there, we obtain condition (ii):
Suppose indeed that there exists an integer $s \geq 1$ such that $r_{2s-1}=G_{p,q}^k$ and $r_{2s}=G_{p,q}^l$ for some integers
$k,l$ with $k>l+1\geq 1$. Then, the number $N$ of invariant factors of $u$ that are multiples of $G_{p,q}^{l+2}$ equals $2s-1$, which is odd,
and it is also the number of invariant factors of $u$ that are multiples of $G_{p,q}^{l+1}$. As $G_{p,q}$ has degree $4$
one finds
$$4N=\dim \Ker\bigl(G_{p,q}^{l+2}(u)\bigr)-\dim \Ker\bigl(G_{p,q}^{l+1}(u)\bigr)
=\dim \Ker\bigl(G_{p,q}^{l+1}(u)\bigr)-\dim \Ker\bigl(G_{p,q}^{l}(u)\bigr).$$
Yet $4N=4(2s-1)$ is not a multiple of $8$, which contradicts the above result. Hence, condition (ii) holds, which completes our proof.
\end{proof}

Combining Theorem \ref{qtheoFpq4roots} with Theorem \ref{qtheolastcase}, we obtain the classification of indecomposable q-exceptional $(p,q)$-quotients when
$p$ and $q$ are irreducible with distinct splitting fields and $(\tr p,\tr q) \neq (0,0)$, as given in Table \ref{qfigure8}.

This completes our classification of $(p,q)$-quotients.

\appendix

\section{Appendix: A result on block-cyclic matrices}

Here, we state and prove some results that are used to study the d-exceptional $(p,q)$-differences and
q-exceptional $(p,q)$-quotients.

First of all, we need to handle multiple roots of polynomials over fields whose characteristic might be positive.
This requires that we adopt a non-traditional variation of the differentials of a polynomial.

\begin{Def}
Given a polynomial $r \in \F[t]$, we choose another indeterminate $x$ and we obtain
a sequence $(G^n(r))_{n \in \N}$ of polynomials of $\F[t]$ (which terminates at $0$) such that
$$r(t+x)=\sum_{n=0}^{+\infty} G^n(r)\, x^n.$$
\end{Def}

It is easily seen that
$$\forall n \in \N, \; D^n(r)=n!\,G^n(r)$$
where $D^n(r)$ denotes the differential of order $n$ of $r$.
In particular $G^0(r)=r$ and $G^1(r)=r'$.
The following result is then essentially obvious:

\begin{prop}
Let $r \in \F[t]$, $\lambda$ be a scalar in $\overline{\F}$, and $n$ be a positive integer.
The following conditions are then equivalent:
\begin{enumerate}[(i)]
\item $\lambda$ is a root of $r$ with multiplicity at least $n$.
\item For all $k \in \lcro 0,n-1\rcro$, one has $G^k(r)[\lambda]=0$.
\end{enumerate}
\end{prop}

From there, we can prove the following key result:

\begin{prop}\label{blockcyclicprop}
Let $d$ and $n$ be positive integers, and $P$ be an irreducible monic polynomial with degree $d$
in $\F[t]$. Let $N \in \Mat_d(\F)$ be annihilated by $P$.
 Consider the following matrix
$$M:=\begin{bmatrix}
N & 0_d & \cdots & \cdots & (0) \\
I_d & N
 & \ddots & & \vdots \\
0 & \ddots & \ddots & \ddots & \vdots \\
\vdots & & \ddots & N & 0_d \\
(0) & \cdots & 0 & I_d & N
\end{bmatrix}$$
of $\Mat_{nd}(\F)$.
Denote by $m$ the multiplicity of the roots of $P$, and consider the Euclidian division
$n=qm+r$.
Then, $M$ is similar to the direct sum of
$m-r$ copies of $C(P^q)$ and of
$r$ copies of $C(P^{q+1})$.
\end{prop}

\begin{proof}
Since $P$ is irreducible and annihilates $N$, which is a $d$-by-$d$ matrix, we see that the characteristic polynomial of
$N$ equals $P$ and that $\F[N]$ is a field.
The characteristic polynomial of $M$ equals $P^n$, and hence
every invariant factor of $M$ is a power of $P$.
Given an integer $k \in \N^*$, we denote by $n_k$ the number of invariant factors of
$M$ that equal $P^k$. Classically, we have
$$dn_k=\rk \bigl(P(M)^{k+1}\bigr)+\rk\bigl(P(M)^{k-1}\bigr)-2 \rk \bigl(P(M)^k\bigr).$$

Setting
$$A:=N \oplus \cdots \oplus N \qquad \text{(with $n$ copies of $N$)}$$
and
$$B:=\begin{bmatrix}
0_d & 0_d & \cdots & \cdots & (0) \\
I_d & 0_d & \ddots & & \vdots \\
0 & \ddots & \ddots & \ddots & \vdots \\
\vdots & & \ddots & 0_d & 0_d \\
(0) & \cdots & 0_d & I_d & 0_d
\end{bmatrix} \in \Mat_{nd}(\F),$$
we see that $A$ commutes with $B$, while $M=A+B$. Fixing $R \in \F[t]$, we get
$$R(M)=\sum_{k=0}^{+\infty} G^k(R)[A]\,B^k
=\begin{bmatrix}
R(N) & 0_d & \cdots & \cdots & (0) \\
G^1(R)[N] & R(N) & \ddots & & \vdots \\
\vdots & \ddots & \ddots & \ddots & \vdots \\
\vdots & & \ddots &  & 0_d \\
G^{n-1}(R)[N] & \cdots & \cdots & G^1(R)[N] & R(N)
\end{bmatrix}.$$
Remembering that $\F[N]$ is a subfield of the ring $\Mat_d(\F)$, it follows that:
\begin{itemize}
\item If $G^i(R)[N]=0$ for all $i \in \lcro 0,n-1\rcro$, then $R(M)=0$.
\item If there is a least integer $\ell \in \lcro 0,n\rcro$ for which
$G^\ell(R)[N]\neq 0$, then $\rk R(M)=(n-\ell)d$.
\end{itemize}
The first situation occurs if and only if any root of $P$ is a root of $R$
with multiplicity at least $n$. Moreover, if the second one occurs then $\ell$
is actually the multiplicity of any root of $P$ as a root of $R$.

In particular, the first situation occurs if $R=P^{q+1}$, and the second one occurs whenever $R$ divides $P^q$.
It follows that the minimal polynomial of $M$ divides $P^{q+1}$. Moreover, for any non-negative integer $k$,
$$\rk P^k(M)=\begin{cases}
(n-km)d & \text{if $k \leq q$} \\
0 & \text{otherwise.}
\end{cases}$$
It follows that
$$n_{q+1}=r, \quad n_q=m-r, \quad \text{and} \quad \forall k \in \N \setminus \{q,q+1\}, \; n_k=0,$$
which yields the claimed result.
\end{proof}

\end{document}